\documentclass[12.0 pt]{article}
\usepackage{amssymb,amsmath, mathtools,amsthm,graphicx, subfigure, color, latexsym, hyperref, nameref, url}
\usepackage{bm}
\usepackage{amsfonts}
\usepackage{authblk}
\usepackage{amsbsy}
\usepackage{wrapfig}
\usepackage{longtable}
\usepackage{euscript}
\usepackage{epsfig}
\usepackage{mathrsfs}
\usepackage{graphicx}
\usepackage{fancybox}
\usepackage{varioref}
\usepackage{enumerate}
\usepackage{enumitem}
\usepackage{verbatim}
\usepackage{makeidx}
\usepackage[draft]{fixme}
\usepackage[english]{babel}
\usepackage{natbib}
\newtheorem{thm}{Theorem}[section]

\newtheorem{Prop}{Proposition}[section]
\newtheorem{lem}{Lemma}[section]
\newtheorem{Cor}{Corollary}[section]
\newtheorem{Rk}{Remark}[section]
\newtheorem{defn}{Definition}

\newcommand{\continuation}{??}

\makeatletter
\newcommand{\bC}{\bm{Cov}}
\newcommand{\bD}{\bm{\Delta}}
\newcommand{\bd}{\bm{d}}
\newcommand{\bE}{\bm{E}}
\newcommand{\bG}{\bm{\Gamma}}
\newcommand{\bM}{\bm{M}}
\newcommand{\bP}{\bm{P}}
\newcommand{\bV}{\bm{Var}}

\newcommand{\norm}[1]{\lVert#1\rVert}
\newcommand{\abs}[1]{\lvert#1\rvert}
\DeclareMathOperator*{\diag}{Diag}

\makeatother
\newcounter{mylabelcounter}

\makeatletter
\newcommand{\labeltext}[2]{%
#1\refstepcounter{mylabelcounter}%
\immediate\write\@auxout{%
  \string\newlabel{#2}{{1}{\thepage}{{\unexpanded{#1}}}{mylabelcounter.\number\value{mylabelcounter}}{}}%
}%
}
\makeatother

\begin{document}

\title{\bf Branching Processes in Random Environments with Thresholds}

\author{Giacomo Francisci}
\author{Anand N.\ Vidyashankar}

\affil{George Mason University}

\date{\today}

\maketitle


\begin{abstract}
Motivated by applications to COVID dynamics, we describe a branching process in random environments model $\{Z_n\}$  whose characteristics change when crossing upper and lower thresholds. This introduces a cyclical path behavior involving periods of increase and decrease leading to supercritical and subcritical regimes. Even though the process is not Markov, we identify subsequences at random time points $\{(\tau_j, \nu_j)\}$ - specifically the values of the process at crossing times, {\it{viz.}}, $\{(Z_{\tau_j}, Z_{\nu_j})\}$ - along which the process retains the Markov structure. Under mild moment and regularity conditions, we establish that the subsequences possess a regenerative structure and prove that the limiting normal distribution of the growth rates of the process in supercritical and subcritical regimes decouple.
For this reason, we establish limit theorems concerning the length of supercritical and subcritical regimes and the proportion of time the process spends in these regimes. As a byproduct of our analysis, we explicitly identify the limiting variances in terms of the functionals of the offspring distribution, threshold distribution, and environmental sequences.
\end{abstract}


\vspace{0.1in}
\noindent{\bf{Math Subject Classification 2020.}} {\it{Primary}:} 60J80, 60F05, 60J10;
{\it{Secondary}:} 92D25, 92D30, 60G50, 62F10.

\section{Introduction} \label{section:introduction}

Branching processes and their variants are used to model various biological, biochemical, and epidemic processes (see \citet{Jagers-1975,Haccou-2007,Hanlon-2011,Kimmel-2015}). More recently, these methods have been used as a model for spreading COVID cases in a community during the early stages of the pandemic \citep{Yanev-2020,Atanasov-2021}. As the time progressed, the number of infected members in a community changed due to different containment efforts of the local communities \citep{Falco-2021,Sun-2022} leading to periods of increase and decrease. In this paper, we describe a stochastic process model built on a branching process model in random environments that explicitly takes into account periods of growth and decrease in the transmission rate of the virus.

Specifically, we consider a branching process model initiated by a random number of ancestors (thought of as initiators of the pandemic within a community). During the first several generations, the process grows uncontrolled, allowing immigration into the system. This initial phase is modeled using a supercritical branching process with immigration in random environments, specifically independent and identically distributed (i.i.d.) environments. When consequences of rapid spread become significant, policymakers introduce restrictions to reduce the rate of growth, hopefully resulting in a reduced number of infected cases. The limitations are modeled using upper thresholds on the number of infected cases, and beyond the threshold the process changes its character to evolve as a subcritical branching process in random environments. During this period - due to strict controls - immigration is also not allowed. In practical terms, this period typically involves a ``lockdown'' and other social containment efforts, the intensity of which varies across communities.

The period of restrictions is not sustainable for various reasons, including political, social, and economic pressures leading to the easing of controls. The policymakers use multiple metrics to gradually reduce controls, leading to an ``opening of communities", resulting in increased human interactions. As a result or due to changes undergone by the virus, the number of infected cases increases again. We use lower thresholds in the number of ``newly infected'' to model the period of change and let the process evolve again as a supercritical BPRE in i.i.d.\ environments after it crosses the lower threshold. The process continues to evolve in this manner alternating between periods of increase and decrease. In this paper, we provide a rigorous probabilistic analysis of this model.

Even though we used the dynamics of COVID spread as a motivation for the proposed model, the aforementioned cyclic behavior is often observed in other biological systems, such as those modeled as a predator-prey model or the SIR model. In some biological populations, the cyclical behavior can be attributed to the decline of fecundity as the population size approaches a threshold \citep{Klebaner-1993}. Deterministic models such as ordinary differential equations, dynamical systems, and corresponding discrete time models are used for analysis in the applications mentioned above \citep{Teschl-2012,Perko-2013,Iannelli-2014}. While many models described above yield good qualitative descriptions, uncertainty estimates are typically unavailable. It is worthwhile to point out that previously described branching process methods also produce reasonable point estimates for the mean growth during the early stages of the pandemic. However, the above-mentioned point estimates of the growth rate are unreliable during the later stages of the pandemic. In this paper, we address statistical estimation of the mean growth and characterize the variance of the estimates. We end the discussion with a  plot, Figure \ref{plot_COVID_cases_in_Italy}, of the total number of confirmed COVID cases per week in Italy from February 23, 2020 to July 20, 2022. The plot also includes the number of cases using the proposed model.
\begin{figure}
\centering
\includegraphics[width=\linewidth]{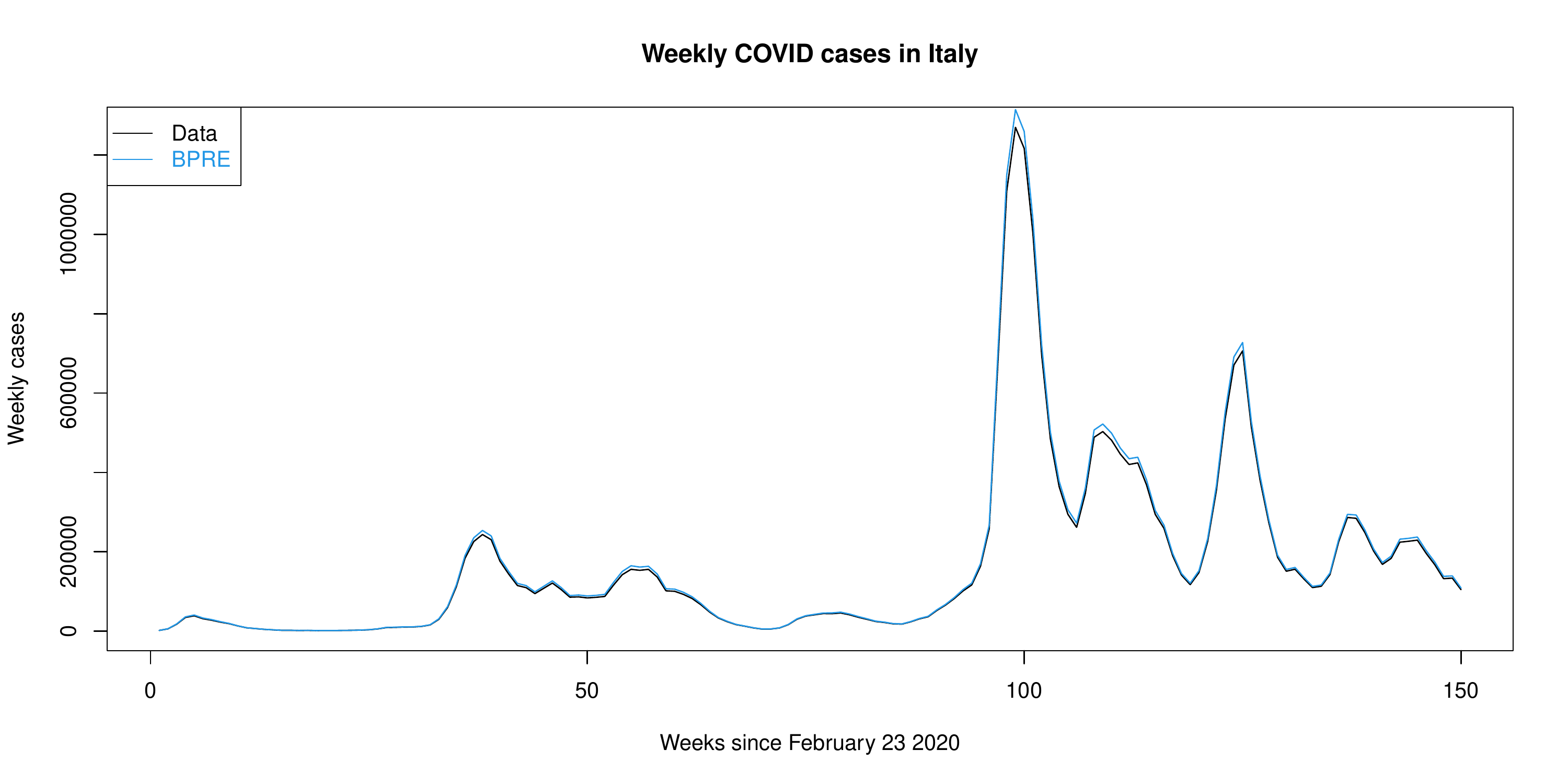}
\caption{In black weekly COVID cases in Italy from February 23, 2020 to February 3, 2023. In blue a BPRE starting with the same initial value and offspring mean having the negative binomial distribution with predefined number of successful trials $r=10$ and Gamma-distributed mean with shape parameter equal to the mean of the data and rate parameter $1$.}
\label{plot_COVID_cases_in_Italy}
\end{figure}
Other examples with similar plots include the hare-lynx predator-prey dynamics and measles cases \citep{Tyson-2010,Iannelli-2014,Hempel-2015}.

Before we provide a precise description of our model, we begin with a brief description of BPRE with immigration. Let $\Pi_{n} = (P_{n},Q_{n})$ be i.i.d.\ random variables taking values in $\mathcal{P} \times \mathcal{P}$, where $\mathcal{P}$ is the space of probability distributions on $\mathbb{N}_{0}$; that is, $P_{n}=\{ P_{n,r} \}_{r=0}^{\infty}$ and $Q_{n}=\{ Q_{n,r} \}_{r=0}^{\infty}$ for some non-negative integers $P_{n,r}$ and $Q_{n,r}$ such that $\sum_{r=0}^{\infty} P_{n,r}=1$ and $\sum_{r=0}^{\infty} Q_{n,r}=1$. The process $\Pi=\{ \Pi_{n} \}_{n=0}^{\infty}$ is referred to as the environmental sequence. For each realization of $\Pi$, we associate a population process $\{ Z_{n} \}_{n=0}^{\infty}$ defined recursively as follows: let $Z_{0}$ take values on the positive integers and $n \geq 0$
\begin{equation*}
    Z_{n+1}=\sum_{i=1}^{Z_{n}} \xi_{n,i}+I_{n},
\end{equation*}
where, given $\Pi_{n}=(P_{n},Q_{n})$, $\{ \xi_{n,i} \}_{i=1}^{\infty}$ are i.i.d.\ with distribution $P_{n}$ and $I_{n}$ is an independent random variable with distribution $Q_{n}$. The random variable $Y_{n}=\log(\overline{P}_{n})$, where $\overline{P}_{n}=\sum_{r=0}^{\infty} r P_{n,r}$, plays an important role in classification of BPRE with immigration. It is well-known that when $\bE[Y_{0}] > 0$ the process diverges to infinity with probability one and if $\bE[Y_{0}] \leq 0$ and the immigration is degenerate at zero for all environments then the process becomes extinct with probability one \citep{Athreya-1971}. Furthermore, in the subcritical case, that is $\bE[Y_{0}] < 0$, one can further identify three distinct regimes: (i) weakly subcritical, (ii) moderately subcritical, and (iii) strongly subcritical. (i) corresponds to when there exists a $0 < \rho < 1$ such that $\bE[Y_{0} e^{\rho Y_{0}}]=0$, while (ii) corresponds to the case when $\bE[Y_{0} e^{Y_{0}}]=0$. Finally, (iii) corresponds to the case when $\bE[Y_{0} e^{Y_{0}}]<0$ \citep{Kersting-2017}. In this paper, when working with the subcritical regime, we will assume that the process is strongly subcritical and refer to it as subcritical process in the rest of the manuscript.

We now turn to a description of the model. Let $\Pi^{U} = \{ \Pi_{n}^{U} \}_{n=0}^{\infty}$, where $\Pi_{n}^{U}=(P_{n}^{U},Q_{n}^{U})$, denote a collection of supercritical environmental sequences. Here, $P_{n}^{U}=\{ P_{n,r}^{U} \}_{r=0}^{\infty}$ indicates the offspring distribution and $Q_{n}^{U}=\{ Q_{n,r}^{U} \}_{r=0}^{\infty}$ represents the immigration distribution. Also, let $\Pi^{L}=\{\Pi_{n}^{L}\}_{n=0}^{\infty}$, where $\Pi_{n}^{L}=P_{n}^{L}=\{ P_{n,r}^{L} \}_{r=0}^{\infty}$, denote a collection of subcritical environmental sequences. We now provide an evolutionary description of the process: at time zero the process starts with a random number of ancestors $Z_{0}$. Each of them live one unit of time and reproduce according to the distribution $P_{0}^{U}$. Thus, the size of the first generation population is 
\begin{equation*}
    Z_{1}=\sum_{i=1}^{Z_{0}} \xi_{0,i}^{U}+I_{0}^{U},
\end{equation*}
where, given $\Pi_{0}^{U}=(P_{0}^{U},Q_{0}^{U})$, $\xi_{0,i}^{U}$ are i.i.d.\ random variables with offspring distribution $P_{0}^{U}$ and independent of the immigration random variable $I_{0}^{U}$ with distribution $Q_{0}^{U}$. The random variable $\xi_{0,i}^{U}$ is interpreted as the number of children produced by the $i^{\text{th}}$ parent in the $0^{\text{th}}$ generation and $I_{0}^{U}$ is interpreted as the number of immigrants whose distribution is generated by the same environmental random variable $\Pi^{U}_{0}$.

Let $U_{1}$ denote the random variable representing the upper threshold. If $Z_{1} < U_{1}$, each of the first generation population live one unit of time and evolve, conditionally on the environment, as the ancestors independent of the population size at time one. That is, 
\begin{equation*}
   Z_{2} = \sum_{i=1}^{Z_{1}} \xi_{1,i}^{U}+I_{1}^{U}.
\end{equation*}
As before, given $\Pi_{1}^{U}=(P_{1}^{U},Q_{1}^{U})$, $\xi_{1,i}^{U}$ are i.i.d.\ with distribution $P_{1}^{U}$ and $I_{1}^{U}$ has distribution $Q_{1}^{U}$. The random variables $\xi_{1,i}^{U}$ are independent of $Z_{1}$, $\xi_{0,i}^{U}$, and $I_{0}^{U}$, $I_{1}^{U}$. If $Z_{1} \geq U_{1}$, then
\begin{equation*}
   Z_{2} = \sum_{i=1}^{Z_{1}} \xi_{1,i}^{L},
\end{equation*}
where, given $\Pi_{1}^{L}=P_{1}^{L}$, $\xi_{1,i}^{L}$ are i.i.d.\ with distribution $P_{1}^{L}$. Thus, the size of the second generation population is 
\begin{equation*}
   Z_{2} = \begin{cases}
   \sum_{i=1}^{Z_{1}} \xi_{1,i}^{U}+I_{1}^{U} &\text{ if } Z_{1}<U_{1}, \\
   \sum_{i=1}^{Z_{1}} \xi_{1,i}^{L} &\text{ if } Z_{1} \geq U_{1}.
   \end{cases}
\end{equation*}
The process $Z_{3}$ is defined recursively as before. As an example, if $Z_{1} < U_{1}$, $Z_{2} < U_{1}$ or $Z_{1} \geq U_{1}$, $Z_{2} \leq L_{1}$, for a random lower threshold $L_{1}$, then the process will evolve like a supercritical BPRE with offspring distribution $P_{2}^{U}$ and immigration distribution $Q_{2}^{U}$. Otherwise (that is, $Z_1< U_1$ and $Z_2 \ge U_1$ or $Z_1 \ge U_1$ and $Z_2 >L_1$), the process will evolve like a subcritical BPRE with offspring distribution $P_{2}^{L}$. This dynamics continues with different thresholds $(U_j, L_j)$ yielding the process $\{ Z_{n} \}_{n=0}^{\infty}$ which we refer to as branching process in random environments with thresholds (BPRET). The consecutive set of generations where the reproduction is governed by a supercritical BPRE is referred to as the \emph{supercritical regime}, while the other is referred to as the \emph{subcritical regime}. As we will see below,  non-trivial immigration in the supercritical regime is required to obtain alternating periods of increase and decrease.

The model described above is related to size dependent branching processes with a threshold as studied by \citet{Klebaner-1993} and more recently by \citet{Athreya-2016}. Specifically, in that model the offspring distribution depends on a \emph{fixed threshold} $K$ and the size of the previous generation.  As observed in these papers, these Markov processes either explode to infinity or are absorbed at zero. In our model the thresholds are \emph{random and dynamic} resulting in a non-Markov process; however, the offspring distribution does not depend on the size of the previous generation \emph{as long as they belong to the same regime}. Indeed, when $U_{j}-1=L_{j}=K$ for all $j \geq 1$,  the immigration distribution is degenerate at zero, and the environment is fixed, one obtains as a special case the density dependent branching process (see for example \citet{Klebaner-1984,Klebaner-1993,Jagers-2011,Athreya-2016}). Additionally, while the model of \citet{Klebaner-1993} uses Galton-Watson process as a building block, our model uses branching processes in i.i.d.\ environments.

Continuing with our discussion on the literature, \citet{Athreya-2016} show that in the fixed environment case the special case of size-dependent process with a single threshold becomes extinct with probability one. We show that this is also the case for the BPRE when there is no immigration, and the details are in Theorem \ref{theorem:extinction_probability}. Similar phenomenon have been observed in slightly different contexts  in \cite{Jagers-2020,Jagers-2021}. Incorporation of immigration component ensures that the process is not absorbed at zero and hence may be useful for modeling stable populations at equilibrium as done in deterministic models.  For an additional discussion see Section \ref{section:discussion_and_concluding_remarks}.

For the ease of further discussions we introduce a few notations. Let $Y_{n}^{U} \coloneqq \log(\overline{P}_{n}^{U})$ and $Y_{n}^{L} \coloneqq \log(\overline{P}_{n}^{L})$, where
\begin{equation*}
    \overline{P}_{n}^{U}=\sum_{r=0}^{\infty} r P_{n,r}^{U} \text{ and } \overline{P}_{n}^{L}=\sum_{r=0}^{\infty} r P_{n,r}^{L};
\end{equation*}
that is, $\overline{P}_{n}^{U}$ and $\overline{P}_{n}^{L}$ represent the offspring means conditional on the environment $\Pi_{n}^{U}=(P_{n}^{U},Q_{n}^{U})$ and $\Pi_{n}^{L}=P_{n}^{L}$, respectively. Also, let $\overline{Q}_{n}^{U}=\sum_{r=0}^{\infty} r Q_{n,r}^{U}$ denote the immigration mean conditional on the environment; and 
\begin{equation*}
    \overline{\overline{P}}_{n}^{U}=\sum_{r=0}^{\infty} (r-\overline{P}_{n}^{U})^{2} P_{n,r}^{U} \text{ and } \overline{\overline{P}}_{n}^{L}=\sum_{r=0}^{\infty} (r-\overline{P}_{n}^{L})^{2} P_{n,r}^{L}
\end{equation*}
denote the conditional variance of the offspring distributions given the environment.

From the description, it is clear that the crossing times at the thresholds $(U_{j},L_{j})$ of $Z_{n}$, namely $\tau_{j}$ and $\nu_{j}$ will play a significant role in the analysis. It will turn out that $\{Z_{\tau_{j}}\}$ and $\{Z_{\nu_{j}}\}$ will form a time homogeneous Markov chain with state space $S^{L} \coloneqq \mathbb{N}_{0} \cap [0,L_{U}]$ and $S^{U} \coloneqq \mathbb{N} \cap [L_{U}+1,\infty)$, respectively, where we take $L_{j} \leq L_{U}$ and $U_{j} \geq L_{U}+1$ for all $j \geq 1$. Under additional conditions on the offspring distribution and the environment sequence, the processes $\{Z_{\tau_{j}}\}$ and $\{Z_{\nu_{j}}\}$ will be uniformly ergodic. These results are established in Section \ref{section:path_properties_of_BPRET}.

The amount of time the process spends in the supercritical and subcritical regimes, beyond its mathematical and scientific interest, will also arise when studying the central limit theorem for the estimates of $M^{U} \coloneqq \bE[\overline{P}_{n}^{U}]$ and $M^{L} \coloneqq \bE[\overline{P}_{n}^{L}]$. Using the uniformly ergodicity alluded to above, we will establish that the time averages of $\tau_{j}-\nu_{j-1}$ and $\nu_{j}-\tau_{j}$ converge to finite positive constants, $\mu^{U}$ and $\mu^{L}$. Additionally, we establish a central limit theorem (CLT) related to this convergence under a finite second moment hypothesis after an appropriate centering and scaling; that is,
\begin{align*}
    \frac{1}{\sqrt{n}}\sum_{j=1}^n (\tau_j -\nu_{j-1}) \xrightarrow[n \to \infty]{d} N(  \mu^U, \sigma^{2,U})
\end{align*}
and we characterize $\sigma^{2,U}$ in terms of the stationary distribution of the Markov chain. A similar result also holds for $\nu_j-\tau_{j}$. This, in turn, provides qualitative information regarding the proportion of time the process spends in these regimes. That is, if $C_{n}^{U}$ is the amount of time the process spends in the supercritical regime up to time $n-1$ we show that $n^{-1}C_n^{U}$ converges to $\mu^U(\mu^U+\mu^L)^{-1}$; a related central limit theorem is also established and in the process we also characterize the limiting variance. Interestingly, we show that the CLT prevails even for the joint distribution of the length of time and the proportion of time the process spends in supercritical and subcritical regimes. These results are described in Sections \ref{section:regenerative_property_of_crossing_times} and \ref{section:proportion_of_time_spent_in_supercritical_and_subcritical_regimes}.

An interesting question concerns the rate of growth of the BPRET in the supercritical and subcritical regimes described by the corresponding expectations, namely $M^{U}$ and $M^{L}$. Specifically, we establish that the limiting joint distribution of the estimators is bivariate normal with a diagonal covariance matrix yielding asymptotic independence of the mean estimators derived using data from supercritical and subcritical regimes. In the classical setting of a supercritical BPRE \emph{without immigration} this problem has received some attention (see for instance \citet{Dion-1979}). The problem considered here is different in the following four ways: (i) the population size does not converge to infinity, (ii) the lengths of the regimes are random, (iii) in the supercritical regime the population size may be zero, and (iv) there is an additional immigration term. While (iii) and (iv) can be accounted for in the classical settings as well, their effect on the point estimates is minimized due to the exponential growth of the population size. Here, while the exponential growth is ruled out, perhaps as anticipated, the Markov property of the process at crossing times, namely $\{Z_{\tau_j}\}$ and $\{Z_{\nu_j}\}$ and their associated regeneration times play a central role in the proof. It is important to note that, it is possible that both regimes occur between regeneration times. Hence, also the proportion of time the process spends in the supercritical and subcritical regime plays a vital role in deriving the asymptotic limit distribution. The limiting variance of the estimators depend additionally on $\mu^{U}$ and $\mu^{L}$, beyond $V_{1}^{U} \coloneqq \bV[\overline{P}_{0}^{U}]$, $V_{1}^{L} \coloneqq \bV[\overline{P}_{0}^{L}]$, $V_{2}^{U} \coloneqq \bE[\overline{\overline{P}}_{0}^{U}]$, and $V_{2}^{L} \coloneqq \bE[\overline{\overline{P}}_{0}^{L}]$. In the special case of fixed environments, the limit behavior of the estimators takes a different form compared to the traditional results as described for example in \citet{Heyde-1971}. These results are in Section \ref{section:estimating_the_mean_of_the_offspring_distribution}.

Finally, in Appendix \ref{section:numerical_experiments} we provide some numerical experiments illustrating the behavior of the model. Specifically, we illustrate the effects of different distributions on the path behavior of the process and describe how they change when the thresholds increase. The experiments also suggests that if different regimes are not taken into account the true growth rate of the virus may be underestimated. We now turn to Section \ref{section:main_results} where we develop additional notations and provide a precise statement of the main results.

\section{Main results} \label{section:main_results}

Branching process in random environments with thresholds (BPRET) is a supercritical BPRE with immigration until it reaches an upper threshold after which it transitions to a subcritical BPRE until it crosses a lower threshold. Beyond this time the process reverts to a supercritical BPRE with immigration and the above cycle continues. Specifically, let $\{ (U_{j},L_{j}) \}_{j=1}^{\infty}$ denote a collection of thresholds (assumed to be i.i.d.). Then, the BPRET evolves like a supercritical BPRE with immigration until it reaches the upper threshold $U_{1}$ at which time it becomes a subcritical BPRE. The process remains subcritical until it crosses the threshold $L_{1}$; after that it evolves again as a supercritical BPRE with immigration, and so on. We now provide a precise description of BPRET.

Let $\{ (U_{j},L_{j})\}_{j=1}^{\infty}$ be i.i.d.\ random vectors with support $S_{B}^{U} \times S_{B}^{L}$, where $S_{B}^{U} \coloneqq \mathbb{N} \cap [L_U+1,\infty)$, $S_{B}^{L} \coloneqq \mathbb{N} \cap [L_{0}, L_{U}]$, and $1 \leq L_{0} \leq L_{U}$ are fixed integers. We denote by $\Pi^{U}$ and $\Pi^{L}$ the supercritical and subcritical environmental sequences; that is, 
\begin{equation*}
    \Pi^{U} = \{ \Pi_{n}^{U} \}_{n=0}^{\infty}=\{ (P_{n}^{U},Q_{n}^{U}) \}_{n=0}^{\infty} \text{ and } \Pi^{L} = \{ \Pi_{n}^{L} \}_{n=0}^{\infty}= \{ P_{n}^{L} \}_{n=0}^{\infty}.
\end{equation*}
We use the notations $\bP_{E^{U}}$ and $\bP_{E^{L}}$ for probability statements with respect to (w.r.t.)\ the supercritical and subcritical environmental sequences. As in the introduction, given the environment, $\xi_{n,i}^{U}$ are i.i.d.\ random variables with distribution $P_{n}^{U}$ and are independent of the immigration random variable $I_{n}^{U}$. Similarly, conditionally on the environment, $\xi_{n,i}^{L}$ are i.i.d.\ random variables with offspring distribution $P_{n}^{L}$. Finally, let $Z_{0}$ be an independent random variable with support included in $\mathbb{N} \cap [1, L_{U}]$. We emphasize that the thresholds are independent of the environmental sequences, offspring random variables, immigration random variables, and $Z_{0}$. For technical details regarding the construction of the probability space we refer to Appendix \ref{subsection:probability_space}. We denote by $M^{T} \coloneqq \bE[\overline{P}_{0}^{T}]$, $T \in \{ L, U \}$, and $N^{U} \coloneqq \bE[\overline{Q}_{0}^{U}]$ the annealed (averaged over the environment) offspring mean and the annealed immigration mean respectively. Throughout the manuscript, we make the following assumptions on the environmental sequences. \\

\noindent {\bf {Assumptions:}}
\begin{enumerate}[label=(\textbf{H\arabic*})]
    \item $\Pi^{T} = \{ \Pi_{n}^{T} \}_{n=0}^{\infty}$ are i.i.d.\ environments such that $P_{0,0}^{T}<1$ and $0 < \overline{P}_{0}^{T} < \infty$ $\bP_{E^{T}}$-a.s. \label{H1}
    \item $\bE[Y_{0}^{L} e^{Y_{0}^{L}}] < 0$, $\bE[Y_{0}^{U}] >0$, $M^{U} < \infty$, and $\bE[\log(1-P_{0,0}^{U})]>-\infty$. \label{H2}
    \item $\bP_{E^{U}}(Q_{0,0}^{U}<1)>0$ and $N^{U}<\infty$. \label{H3}
    \item $\{(U_{j}, L_{j})\}_{j=1}^{\infty}$ are i.i.d.\ and have support $S_{B}^{U} \times S_{B}^{L}$, where $1 \leq L_{0} \leq M^{L} L_{U}$ and $\bE[U_{1}]<\infty$. \label{H4}
\end{enumerate}
The above assumptions rule out degenerate behavior of the process and are \emph{commonly used} in the literature on BPRE (see Assumption R and Theorem 2.2 of \citet{Kersting-2017}). Assumption \ref{H2} states that $\Pi_{n}^{U}$ is a supercritical environment and $\Pi_{n}^{L}$ is a (strongly) subcritical environment. Additionally, using Jensen's inequality it follows that $M^{L} < 1$ and $1 < M^{U} < \infty$. Assumption \ref{H3} states that immigration is positive with positive probability and has finite expectation $N^{U}$, while \ref{H4} states that the upper thresholds $U_{j}$ have finite expectation.

We are now ready to give a precise definition of the BPRET. Let $\nu_{0} \coloneqq 0$. Starting from $Z_{0}$, the BPRET $\{ Z_{n} \}_{n=0}^{\infty}$ is defined recursively over $j \geq 0$ as follows.
\begin{enumerate}
    \item[1j.] For $n \geq \nu_{j}$ and until $Z_{n} < U_{j+1}$
\begin{equation} \label{process_in_supercritical_regime}
    Z_{n+1} = \sum_{i=1}^{Z_n} \xi_{n,i}^{U} + I_{n}^{U}.
\end{equation}  
Next, let $\tau_{j+1} \coloneqq \inf\{n \ge \nu_{j}: Z_{n} \ge U_{j+1}\}$.
\item[2j.] For $n \geq \tau_{j+1}$ and until $Z_{n} > L_{j+1}$
\begin{equation} \label{process_in_subcritical_regime}
    Z_{n+1} = \sum_{i=1}^{Z_n} \xi_{n,i}^{L}.
\end{equation}
Next, let $\nu_{j+1} \coloneqq \inf \{ n \geq \tau_{j+1} : Z_{n} \leq L_{j+1} \}$.
\end{enumerate}
It is clear from the definition that $\nu_{j}$ and $\tau_{j}$ are stopping times w.r.t.\ the $\sigma$-algebra $\mathcal{F}_{n}$ generated by $\{ Z_{j} \}_{j=0}^{n}$ and the thresholds $\{ (U_{j},L_{j}) \}_{j=1}^{\infty}$. Thus, $Z_{\nu_{j}}$, $Z_{\tau_{j}}$, $\xi_{\nu_{j},i}^{U}$, $\xi_{\tau_{j+1},i}^{L}$, and $I_{\nu_{j}}^{U}$ are well-defined random variables.

It is also clear from the above definition that the intervals $[\nu_{j-1},\tau_{j})$ and $[\tau_{j},\nu_{j})$ represent supercritical and subcritical intervals, respectively. We show below that the process $\{ Z_{n} \}_{n=0}^{\infty}$ exits and enters the above intervals infinitely often. Let $\Delta_{j}^{U} \coloneqq \tau_{j}-\nu_{j-1}$ and $\Delta_{j}^{L} \coloneqq \nu_{j}-\tau_{j}$ denote the length of these intervals. Since a supercritical BPRE with immigration diverges with probability one (see Theorem 2.2 of \citet{Kersting-2017}), it follows that  $\tau_{j+1}$ is finite whenever $\nu_{j}$ is finite:
\begin{equation} \label{tau_finite}
    \bP(\Delta_{j+1}^{U}=\infty | \nu_{j}<\infty) = \bP( \cap_{l=1}^{\infty} \{ Z_{\nu_{j}+l} <U_{j+1} \} | \nu_{j}<\infty)=0.
\end{equation}
We emphasize that Assumption \ref{H3} is required since, otherwise, if $I_{0}^{U} \equiv 0$ the process may fail to cross the upper threshold and becomes extinct (see Theorem \ref{theorem:extinction_probability} below). On the other hand, since strongly subcritical BPRE becomes extinct with probability one, $\Delta_{j+1}^{L}<\infty$ whenever $\tau_{j+1}<\infty$, that is,
\begin{equation} \label{nu_finite}
    \bP(\Delta_{j+1}^{L} <\infty | \tau_{j+1}<\infty)=1.
\end{equation}
Using $\nu_{0}=0$ and induction over $j$ we see that $\Delta_{j+1}^{U}$, $\Delta_{j+1}^{L}$, $\tau_{j+1}$, and $\nu_{j+1}$ are finite almost surely. We emphases that \eqref{nu_finite} holds whenever $\Pi^{L}$ is a subcritical or critical (but not strongly critical) environmental sequence (see Definition 2.3 in \citet{Kersting-2017}). That is, it remains valid if the assumption $\bE[Y_{0}^{L} e^{Y_{0}^{L}}] < 0$ in \ref{H2} is weakened to $\bE[Y_{0}^{L}] \leq 0$ and $\bP_{E^{L}}(Y_{0}^{L} \neq  0)>0$, which leads to the following assumption:
\begin{enumerate}[label=($\textbf{H\arabic*}^{\bf{\prime}})$] \setcounter{enumi}{1}
  \item $\bE[Y_{0}^{L}] \leq 0$, $\bP_{E^{L}}(Y_{0}^{L} \neq  0)>0$, and $\bE[Y_{0}^{U}] >0$. \label{H2prime}
\end{enumerate}
The next theorem shows that if immigration is zero the process becomes extinct almost surely.
\begin{thm} \label{theorem:extinction_probability}
    Assume \ref{H1}, \ref{H2prime}, and $Q_{0,0}^{U} \equiv 1$ a.s. Let $\mathrm{T} \coloneqq \inf\{n \ge 1: Z_n=0 \}$. Then $\bP(\mathrm{T}<\infty)=1$.
\end{thm}
\noindent Theorem 1 of \citet{Athreya-2016} follows from the above theorem by taking $L_{U}=K$, $L_{j} \equiv K$, $U_{j} \equiv K+1$, where $K$ is a finite positive integer, and assuming that the environments are fixed in both regimes. 

\subsection{Path properties of BPRET} \label{subsection:path_properties_of_BPRET}

We now turn to transience and recurrence of the BPRET $\{Z_{n}\}_{n=0}^{\infty}$. Notice that even though $\{Z_{n}\}_{n=0}^{\infty}$ is not Markov the concepts of recurrence and transience can be studied using the definition given below (due to \citet{Lamperti-1960,Lamperti-1963}).
\begin{defn} \label{definition_recurrent}
A non-negative stochastic process $\{X_n\}_{n=0}^{\infty}$ satisfying $\bP(\limsup_{n \rightarrow \infty} X_n= \infty)=1$ is said to be recurrent if there exists an $ r < \infty$ such that $\bP(\liminf_{n \rightarrow \infty} X_n \le r)=1$ and transient if $\bP(\lim_{n \rightarrow \infty} X_n=\infty)=1$.
\end{defn}

Our next result is concerned with the path behavior of $\{ Z_{n} \}_{n =0 }^{\infty}$ and the stopped sequences $\{ Z_{\nu_{j}} \}_{j = 0 }^{\infty}$ and $\{ Z_{\tau_{j}} \}_{j = 1 }^{\infty}$.

\begin{thm} \label{theorem:recurrence_uniformly_ergodic_stationary_markov_chains}
Assume \ref{H1}-\ref{H4}. Then \\
(i) the process $\{ Z_{n} \}_{n=0}^{\infty}$ is recurrent; \\
(ii) $\{ Z_{\nu_{j}} \}_{j=0}^{\infty}$ and $\{ Z_{\tau_{j}} \}_{j=1}^{\infty}$ are time homogeneous Markov chains.
\end{thm}

We now turn to the ergodicity properties of $\{ Z_{\nu_{j}} \}_{j=0}^{\infty}$ and $\{ Z_{\tau_{j}} \}_{j=1}^{\infty}$. These rely on conditions on the offspring distribution that ensures the Markov chains $\{ Z_{\nu_{j}} \}_{j=0}^{\infty}$ and $\{ Z_{\tau_{j}} \}_{j=1}^{\infty}$ are irreducible and aperiodic. While several sufficient conditions are possible, we provide below some possible conditions.
\begin{enumerate}[label=(\textbf{H\arabic*})] \setcounter{enumi}{4}
    \item $\bP_{E^{L}}(\cap_{r=0}^{1} \{ P_{0,r}^{L} > 0 \} )>0$. \label{H5}
    \item  $\bP_{E^{U}}( \cap_{r=0}^{\infty} \{ P_{0,r}^{U}>0 \} \cap \{ Q_{0,0}^{U}>0 \})>0$ and $\bP_{E^{U}}(Q_{0,s}^{U}>0)>0$ for some $s \in \{ 1, \dots, L_{U} \}$. \label{H6}
    \item $\bP_{E^{U}}( \{P_{0,0}^{U}>0\} \cap \cap_{r=L_{U}+1}^{\infty} \{ Q_{0,r}^{U}>0\})> 0$. \label{H7}
\end{enumerate}
\ref{H5} requires that on a set of positive $\bP_{E^{L}}$ probability an individual can produce zero and one offspring while \ref{H6} requires that on a set of positive $\bP_{E^{U}}$ probability $P_{0,r}^{U}>0$ for all $r \in \mathbb{N}_{0}$ and $Q_{0,0}^{U}> 0$. Also, on a set of positive $\bP_{E^{U}}$ probability, $Q_{0,s}^{U}>0$ for some $s \in \{ 1, \dots, L_{U} \}$. Finally, \ref{H7} states that on a set of positive $\bP_{E^{U}}$ probability $P_{0,0}^{U}>0$ and $Q_{0,r}^{U}>0$ for all $r \geq L_{U}+1$. These are weak conditions on the environment sequences and are part of standard BPRE literature. We recall that $S^{L}$ is the set of non-negative integers not larger than $L_{U}$ and $S^{U}$ is the set of integers larger than $L_{U}$.

\begin{thm} \label{theorem:uniform_ergodicity}
Assume \ref{H1}-\ref{H4}. (i) If \ref{H5} also holds, then $\{ Z_{\nu_{j}} \}_{j=0}^{\infty}$ is a uniformly ergodic Markov chain with state space $S^{L}$. (ii) If \ref{H6} (or \ref{H7}) holds, then $\{ Z_{\tau_{j}} \}_{j=1}^{\infty}$ is a uniformly ergodic Markov chain with state space $S^{U}$.
\end{thm}

When the assumptions \ref{H1}-\ref{H6} (or \ref{H7}) hold, we denote by $\pi^{L} = \{ \pi^{L}_{i} \}_{i \in S^{L}}$ and $\pi^{U} = \{ \pi^{U}_{i} \}_{i \in S^{U}}$ the stationary distributions of the ergodic Markov chains $\{ Z_{\nu_{j}} \}_{j=0}^{\infty}$ and $\{ Z_{\tau_{j}} \}_{j=1}^{\infty}$, respectively. While $\pi^{L}$ has moments of all orders, we show in Proposition \ref{proposition:finite_expectation_stationary_distribution} below, that $\pi^{U}$ has a finite first moment. These distributions will play a significant role when studying the length of supercritical and subcritical regime which we now undertake.

\subsection{Length of supercritical and subcritical regime} \label{subsection:length_of_supercritical_and_subcritical_regime}

We now turn to the law of large numbers and central limit theorem for the differences $\Delta_{j}^{U}$ and $\Delta_{j}^{L}$. We denote by $\bP_{\pi^{L}}(\cdot)$, $\bE_{\pi^{L}}[\cdot]$, $\bV_{\pi^{L}}[\cdot]$, and $\bC_{\pi^{L}}[\cdot,\cdot]$ probability, expectation, variance, and covariance conditionally on $Z_{\nu_{0}} \sim \pi^{L}$. Similarly, when $\pi^{L}$ is replaced by $\pi^{U}$ in the above quantities, we understand that they are conditioned on $Z_{\tau_{1}} \sim \pi^{U}$. We define $\mu^{U} \coloneqq \bE_{\pi^{L}}[\Delta_{1}^{U}]$, $\mu^{L} \coloneqq \bE_{\pi^{U}}[\Delta_{1}^{L}]$,
\begin{align} \label{definition_sigma_L}
    \sigma^{2,U} \coloneqq & \bV_{\pi^{L}}[\Delta_{1}^{U}] + 2 \sum_{j=1}^{\infty} \bC_{\pi^{L}}[\Delta_{1}^{U}, \Delta_{j+1}^{U}], \text{ and } \\
  \sigma^{2,L} \coloneqq & \bV_{\pi^{U}}[\Delta_{1}^{L}] + 2 \sum_{j=1}^{\infty} \bC_{\pi^{U}}[\Delta_{1}^{L}, \Delta_{j+1}^{L} ]. \label{definition_sigma_U}
\end{align}
In the supercritical regime, we impose an additional assumption \ref{H8} below, so as to not qualify our statements with the phrase ``on the set of non-extinction''. Assumption \ref{H9} below ensures that the immigration distribution stochastically dominates the upper threshold.
\begin{enumerate}[label=(\textbf{H\arabic*})] \setcounter{enumi}{7}
    \item $P_{0,0}^{U} = 0$ $\bP_{E^{U}}$-a.s. \label{H8}
    \item $\bE[\frac{U_{1}}{\bP(I_{0}^{U} \geq U_{1} | U_{1})}]<\infty$. \label{H9}
\end{enumerate}
Let $S_{n}^{U} \coloneqq \sum_{j=1}^{n} \Delta_{j}^{U}$ and $S_{n}^{L} \coloneqq \sum_{j=1}^{n} \Delta_{j}^{L}$. We now state the main result of this subsection. 
\begin{thm} \label{theorem:law_of_large_numbers_and_central_limit_theorem_for_crossing_times}
Assume \ref{H1}-\ref{H4}. (i) If \ref{H5} and \ref{H8} hold, then,
\begin{equation*}
    \lim_{n \to \infty} \frac{1}{n} S_{n}^{U} = \mu^{U} \text{ a.s., and } \frac{1}{\sqrt{n}} (S_{n}^{U}- n\mu^{U}) \xrightarrow[n \to \infty]{d} N(0,\sigma^{2,U}).
\end{equation*}
(ii) If \ref{H6} (or \ref{H7}) and \ref{H9} hold, then, 
\begin{equation*}
    \lim_{n \to \infty} \frac{1}{n} S_{n}^{L} = \mu^{L} \text{ a.s., and } \frac{1}{\sqrt{n}} (S_{n}^{L}- n\mu^{L}) \xrightarrow[n \to \infty]{d} N(0,\sigma^{2,L}).
\end{equation*}
\end{thm}

\subsection{Proportion of time spent in supercritical and subcritical regime} \label{subsection:proportion_of_times_spent_in_supercritical_and_subcritical_regime}

We now consider the proportion of time the process spends in the subcritical and supercritical regimes. To this end for $n \geq 0$, let $\chi_{n}^U \coloneqq \mathbf{I}_{\cup_{j=1}^{\infty} [\nu_{j-1},\tau_{j})}(n)$ be the indicator function assuming value $1$ if at time $n$ the process is in the supercritical regime and $0$ otherwise. Similarly, let $\chi_{n}^L \coloneqq 1-\chi_{n}^U = \mathbf{I}_{\cup_{j=1}^{\infty} [\tau_{j},\nu_{j})}(n)$
take value $1$ if at time $n$ the process is in the subcritical regime and $0$ otherwise. Furthermore, let $C_{n}^U \coloneqq \sum_{j=1}^{n} \chi_{j-1}^U$ and $C_{n}^L \coloneqq \sum_{j=1}^{n} \chi_{j-1}^L = n-C_{n}^U$ be the total time that the process spends in the supercritical and subcritical regime, respectively, up to time $n-1$. Let
\begin{equation*}
    \theta_{n}^{U} \coloneqq \frac{C_{n}^{U}}{n} \text{ and } \theta_{n}^{L} \coloneqq \frac{C_{n}^{L}}{n}
\end{equation*}
denote the proportion of time the process spends in the supercritical and subcritical regimes. Our main result in this section is concerned with the central limit theorem for $\theta_{n}^{U}$ and $\theta_{n}^{L}$. To this end, let
\begin{equation*}
    \theta^{U} \coloneqq \frac{\mu^{U}}{\mu^{U}+\mu^{L}} \text{ and } \theta^{L} \coloneqq \frac{\mu^{L}}{\mu^{U}+\mu^{L}}.
\end{equation*}

\begin{thm} \label{theorem:law_of_large_numbers_and_central_limit_theorem_for_proportions}
Assume \ref{H1}-\ref{H6} (or \ref{H7}) and \ref{H8}-\ref{H9}. Then, for $T \in \{ L, U\}$, $\theta_{n}^{T}$ converges almost surely to $\theta^{T}$. Furthermore,
\begin{equation*}
    \sqrt{n} (\theta_{n}^{T} - \theta^{T} ) \xrightarrow[n \to \infty]{d} N(0,\eta^{2,T}),
\end{equation*}
where $\eta^{2,T}$ is defined in \eqref{definition:eta}.
\end{thm}
We use these results to now describe the growth rate of the process, as defined by their expectations, in the supercritical and subcritical regime; that is, $M^{U}$ and $M^{L}$.

\subsection{Offspring mean estimation} \label{subsection:offspring_mean_estimation}

We begin by noticing that $Z_{\tau_{j}} \geq L_{U}+1$ and $Z_{\tau_{j}+1}, \dots, Z_{\nu_{j}-1} \geq L_{0}$ are positive for all $j \in \mathbb{N}$. However, there may be instances where $Z_{\nu_{j}}, \dots, Z_{\tau_{j}-1}$ could be zero. To avoid division by zero in \eqref{offspring_mean_estimates} below, we let $\tilde{\chi}_{n}^{U} \coloneqq \chi_{n}^{U} \mathbf{I}_{ \{ Z_{n} \geq 1 \}}$, $\tilde{C}_{n}^{U} \coloneqq \sum_{j=1}^{n} \tilde{\chi}_{j-1}^{U}$, and use the convention that $0/0= 0 \cdot \infty=0$.  The generalized method of moments estimators of $M^{U}$ and $M^{L}$ are given by
\begin{equation} \label{offspring_mean_estimates}
    M_{n}^{U} \coloneqq \frac{1}{\tilde{C}_{n}^{U}} \sum_{j=1}^n \frac{Z_{j}-I_{j-1}^{U}}{Z_{j-1}} \tilde{\chi}_{j-1}^U \text{ and } M_{n}^L \coloneqq \frac{1}{C_{n}^{L}} \sum_{j=1}^n \frac{Z_{j}}{Z_{j-1}} \chi_{j-1}^L,
\end{equation}
where the last term is non-trivial whenever $C_{n}^{L} \geq 1$, that is $n \geq \tau_{1}+1$. Our assumptions will involve first and second moments assumptions on the centered offspring means $(\overline{P}_{n}^{T}-M^{T})$ and the centered offspring random variables $( \xi_{n,i}^{T} - \overline{P}_{n}^{T})$. To this end, we define the quantities $\Lambda_{n,1}^{T,s} \coloneqq \abs{\overline{P}_{n}^{T}-M^{T}}^{s}$ and $\Lambda_{n,2}^{T,s} \coloneqq \bE[\abs{\xi_{n,1}-\overline{P}_{n}^{T}}^{s} | \Pi_{n}^{T}]$. Next, let $\bM_{n} \coloneqq (M_{n}^{U}, M_{n}^{L})^{\top}$, $\bM \coloneqq (M^{U}, M^{L})^{\top}$, and $\Sigma$ be the $2 \times 2$ diagonal matrix with elements
\begin{equation*}
    \frac{1}{\tilde{\theta}^{U}} \biggl(V_{1}^{U} + \frac{\tilde{A}^{U} V_{2}^{U}}{\tilde{\mu}^{U}} \biggr) \text{ and } \frac{1}{\theta^{L}} \biggl( V_{1}^{L} + \frac{A^{L}V_{2}^{L}}{\mu^{L}} \biggr),
\end{equation*}
where $\tilde{\mu}^{U} \coloneqq \bE_{\pi^{L}}[\sum_{k=1}^{\tau_{1}} \tilde{\chi}_{k-1}^{U}]$ is the average length of supercritical regime not taking into account the times in which the process is zero, $\tilde{\theta}^{U} \coloneqq \frac{\tilde{\mu}^{U}}{\mu^{U}+\mu^{L}}$ is the average proportion of time the process spends in the supercritical regime and is positive, $\tilde{A}^{U} \coloneqq \bE_{\pi^{L}}[\sum_{k=1}^{\tau_{1}} \frac{\tilde{\chi}_{k-1}^{U}}{Z_{k-1}}]$ is the average sum of $\frac{1}{Z_{n}}$ over a supercritical regime discarding the times in which $Z_{n}$ is zero, and $A^{L} \coloneqq \bE_{\pi^{U}}[ \sum_{k=\tau_{1}+1}^{\nu_{1}} \frac{\chi_{k-1}^{L}}{Z_{k-1}}]$ is the average sum of $\frac{1}{Z_{n}}$ over a subcritical regime. Obviously, $0 \leq \tilde{\mu}^{U} \leq \mu^{U}$. Finally, we recall that $V_{1}^{T} = \bV[\overline{P}_{0}^{T}]$ is the variance of the random offspring mean $\overline{P}_{0}^{T}$ and $V_{2}^{T} = \bE[\overline{\overline{P}}_{0}^{T}]$ is the expectation of the random offspring variance $\overline{\overline{P}}_{0}^{T}$.

\begin{thm} \label{theorem:law_of_large_numnbers_and_central_limit_theorem_mean_estimation}
Assume \ref{H1}-\ref{H6} (or \ref{H7}). (i) If $\mu^{T}<\infty$ and for some $s > 1$ $\bE[\Lambda_{0,i}^{T,s}]<\infty$, where $i=1,2$ and $T \in \{L,U\}$, then $\bM_{n}$ is a strongly consistent estimator of $\bM$. (ii) If additionally for some $\delta >0$ $\bE[\Lambda_{0,i}^{T,2+\delta}]<\infty$ for $i=1,2$ and $T \in \{L,U\}$, then
\begin{equation*}
    \sqrt{n} (\bM_{n}-\bM) \xrightarrow[n \to \infty]{d} N(\bm{0},\Sigma).
\end{equation*}
\end{thm}

\begin{Rk} \label{remark:law_of_large_numnbers_and_central_limit_theorem_mean_estimation}
  In the fixed environment case $\overline{P}_{0}^{T} = M^{T}$ and $\overline{\overline{P}}_{0}^{T} = V_{2}^{T}$ are deterministic constants. Therefore, $V_{1}^{T}=0$ and $\Sigma$ is the $2 \times 2$ diagonal matrix with elements
\begin{equation*}
  \frac{\tilde{A}^{U} V_{2}^{U}}{\tilde{\theta}^{U} \tilde{\mu}^{U}} \text{ and } \frac{A^{L} V_{2}^{L}}{\theta^{L} \mu^{L}}.
\end{equation*}
\end{Rk}

\section{Path properties of BPRET} \label{section:path_properties_of_BPRET}

In this section we provide the proofs of Theorems \ref{theorem:extinction_probability}, \ref{theorem:recurrence_uniformly_ergodic_stationary_markov_chains}, and \ref{theorem:uniform_ergodicity} along with the required probability estimates. The proofs rely on the fact that both the environmental sequence and the thresholds are i.i.d. It follows that probability statements like $\bP(Z_{\tau_{j+1}} = k | Z_{\nu_{j}}=i, \nu_{j}<\infty)$ and $\bP(Z_{\nu_{j+1}}=i | Z_{\tau_{j+1}}=k, \tau_{j+1} < \infty)$ do not depend on the index $j$. This idea is made precise in Lemma \ref{lemma:stationarity} in Appendix \ref{subsection:time_homogeneity_of_Znu_and_Ztau} and will lead to time homogeneity of $\{ Z_{\nu_{j}} \}_{j=0}^{\infty}$ and $\{ Z_{\tau_{j}} \}_{j=1}^{\infty}$. Expectedly, this property does not depend on the process being strongly subcritical: Assumptions \ref{H1} and \ref{H2prime} are more than enough. We denote by $\bP_{\delta_{i}^{L}}(\cdot)$, $\bE_{\delta_{i}^{L}}[\cdot]$, $\bV_{\delta_{i}^{L}}[\cdot]$, and $\bC_{\delta_{i}^{L}}[\cdot,\cdot]$ probability, expectation, variance, and covariance conditionally on $Z_{\nu_{0}} \sim \delta_{i}^{L}$, where $\delta_{(\cdot)}^{L}$ is the restriction of the Dirac delta to $S^{L}$. Similarly, when $\delta_{i}^{L}$ is replaced by $\delta_{i}^{U}$ in the above quantities, we understand that they are conditioned on $Z_{\tau_{1}} \sim \delta_{i}^{U}$, where $\delta_{(\cdot)}^{U}$ is the restriction of the Dirac delta to $S^{U}$.

\subsection{Extinction when immigration is zero} \label{subsection:extinction_when_immigration_is_zero}

In this subsection, we provide the proof of Theorem \ref{theorem:extinction_probability}, which is an adaptation of Theorem 1 of \citet{Athreya-2016} for BPRE. Recall that for this theorem there is no immigration in the supercritical regime and hence the extinction time $\mathrm{T}$ is finite with probability one.

\begin{proof}[Proof of Theorem \ref{theorem:extinction_probability}]
Set, for simplicity, $\tau_{0} \coloneqq -1$. We partition the sample space as
\begin{equation*}
    \Omega = ( \cup_{j=0}^{\infty} \{ \tau_{j+1}=\infty, \tau_{j}<\infty \} ) \cup ( \cap_{j=1}^{\infty} \{ \tau_{j} < \infty \} )
\end{equation*}
and show that (i) $\{ \tau_{j+1}=\infty, \tau_{j}<\infty \} \subset \{ \mathrm{T}<\infty \}$ for all $j \in \mathbb{N}_{0}$ and (ii) $\bP(\cap_{j=1}^{\infty} \{ \tau_{j}< \infty \} )=0$. First, we notice that if $\tau_{j}<\infty$ then, using Theorem 2.1 of \citet{Kersting-2017}, $\nu_{j}<\infty$. Thus,
\begin{align*}
    \{ \tau_{j+1}=\infty, \tau_{j}<\infty \} &= \{ Z_{n} < U_{j+1} \, \forall n \geq \nu_{j}, \tau_{j}<\infty \},
\end{align*}
where $\{ Z_{n} \}_{n=\nu_{j}}^{\infty}$ is a supercritical BPRE until $U_{j+1}$ is reached. Since $Z_{n} < U_{j+1}$ for all $n \geq \nu_{j}$, (2.6) of \citet{Kersting-2017} yields that $\lim_{n \to \infty} Z_{n} = 0$ a.s.\ and $\{ \tau_{j+1}=\infty, \tau_{j}<\infty \} \subset \{ \mathrm{T} < \infty \}$. Turning to (ii), since the events $\{ \tau_{j} < \infty \}$ are nonincreasing, $\bP(\cap_{j=1}^{\infty} \{ \tau_{j}< \infty\})= \lim_{j \to \infty} \bP(\tau_{j+1}< \infty)$ and $\bP(\tau_{j+1}< \infty) = \bP(\tau_{j+1} < \infty | \tau_{j}<\infty) \bP(\tau_{j}<\infty)$. Since $\tau_{j}=\infty$, if $Z_{\nu_{j-1}}=0$, it follows that 
\begin{align*}
    \bP(\tau_{j+1} < \infty | \tau_{j}<\infty) &\leq \bP(\tau_{j+1} < \infty | \tau_{j}<\infty, Z_{\nu_{j-1}} \in [1,L_{U}]) \\
    &\leq \max_{i=1,\dots,L_{U}} \bP(\tau_{j+1} < \infty | \tau_{j}<\infty, Z_{\nu_{j-1}}=i) \\
    &\leq 1-\min_{i=1,\dots,L_{U}} \bP(Z_{\tau_{j}+1} = 0 | \tau_{j}<\infty, Z_{\nu_{j-1}}=i).
\end{align*}
Lemma \ref{lemma:stationarity} yields that for all $k \in S_{B}^{U}$
\begin{equation} \label{stationarity}
    \bP(Z_{\tau_{j}} = k | \tau_{j}<\infty, Z_{\nu_{j-1}}=i) = \bP_{\delta_{i}^{L}}(Z_{\tau_{1}} = k | \tau_{1}<\infty).
\end{equation}
Also, for all $j \geq 1$
\begin{align*}
    \bP(Z_{\tau_{j}+1}=0 | \tau_{j}<\infty, Z_{\tau_{j}}=k) &=\bP ( \cap_{i=1}^{k} \{ \xi_{\tau_{j},i}^{L}=0 \} | \tau_{j}<\infty ) \\
    &=\sum_{n=1}^{\infty} \bP ( \cap_{i=1}^{k} \{ \xi_{n,i}^{L}=0 \} | \tau_{j}=n ) \bP(\tau_{j}=n) = \bP(\xi_{0,1}^{L}=0)^{k}.
\end{align*}
Multiplying by $\bP(Z_{\tau_{j}+1}=0 | \tau_{j}<\infty, Z_{\tau_{j}}=k)$ and $\bP(Z_{\tau_{1}+1}=0 | \tau_{1}<\infty, Z_{\tau_{1}}=k)$ and summing over $k \geq L_{U}+1$ in \eqref{stationarity}, we obtain that
\begin{align*}
    \bP(Z_{\tau_{j}+1}=0 | \tau_{j}<\infty, Z_{\nu_{j-1}}=i) &=\bP_{\delta_{i}^{L}}(Z_{\tau_{1}+1}=0 | \tau_{1} < \infty) \\
    &=\sum_{k=L_{U}+1}^{\infty} \bP_{\delta_{i}^{L}}(Z_{\tau_{1}}=k | \tau_{1} < \infty) \bP(\xi_{0,1}^{L}=0)^{k}.
\end{align*}
Set $\underline{p} \coloneqq \min_{i=1,\dots,L_{U}} p_{i}$, where $p_{i} \coloneqq \bP_{\delta_{i}^{L}}(Z_{\tau_{1}+1}=0 | \tau_{1} < \infty)$. Since $\bP(\xi_{0,1}^{L}=0)>0$ and $\bP_{\delta_{i}^{L}}(Z_{\tau_{1}}=k | \tau_{1} < \infty)>0$ for some $k$, we have that $p_{i}>0$ and $\underline{p}>0$. Hence, $\bP(\tau_{j+1}< \infty) \leq (1-\underline{p}) \bP(\tau_{j}< \infty)$. Iterating the above argument it follows that $\bP(\tau_{j+1}< \infty) \leq (1-\underline{p})^{j} \bP(\tau_{1}<\infty)$ yielding $\lim_{j \to \infty} \bP(\tau_{j+1}< \infty)=0$.
\end{proof}

\subsection{Markov property at crossing times} \label{subsection:markov_property_at_crossing_times}

\begin{proof}[Proof of Theorem \ref{theorem:recurrence_uniformly_ergodic_stationary_markov_chains}]
We begin by proving (i). We first notice that since $\{U_{j}\}_{j=1}^{\infty}$ are i.i.d.\ random variables with unbounded support $S_{B}^{U}$, $\limsup_{j \rightarrow \infty} U_{j} = \infty$ with probability one. Next, observe that along the subsequence $\{\tau_{j}\}_{j=1}^{\infty}$, $Z_{\tau_{j}} \ge U_{j}$. Hence, $\limsup_{n \rightarrow \infty} Z_{n} =\infty$. On the other hand, along the subsequence $\{\nu_{j} \}_{j=1}^{\infty}$ $Z_{\nu_{j}} \le L_{j}$. Thus, $0 \le \liminf_{j\rightarrow \infty} Z_{j} \le L_U<\infty$. It follows that $\{ Z_{n} \}_{n=0}^{\infty}$ is recurrent in the sense of Definition \ref{definition_recurrent}. Turning to (ii), we first notice that, since $Z_{0} \leq L_{U}$, $Z_{\nu_{j}} \leq L_{j} \leq L_{U}$ and $Z_{\tau_{j}} \geq U_{j} \geq L_{U}+1$ for all $j \geq 1$, the state spaces $S^{L}$ of $\{ Z_{\nu_{j}} \}_{j=0}^{\infty}$ and $S^{U}$ of $\{ Z_{\tau_{j}} \}_{j=1}^{\infty}$ are included in $S^{L}$ and $S^{U}$, respectively. We now establish the Markov property of $\{ Z_{\nu_{j}} \}_{j=0}^{\infty}$. For all $j \geq 0$, $k \in S^{L}$, and $i_{0},i_{1},\dots,i_{j} \in S^{L}$, we consider the probability $\bP(Z_{\nu_{j+1}}=k | Z_{\nu_{0}}=i_{0},\dots,Z_{\nu_{j}}=i_{j})$. By the law of total expectation this is equal to
\begin{equation*}
    \bE [\bP(Z_{\nu_{j+1}}=k | Z_{\nu_{0}}=i_{0},\dots,Z_{\nu_{j}}=i_{j},L_{j+1},U_{j+1},\pi^{L},\pi^{U})].
\end{equation*}
Now, setting $A_{\nu_{j},s}(u) \coloneqq \{ Z_{\nu_{j}+s} \geq u, Z_{\nu_{j}+s-1} <u, \dots, Z_{\nu_{j}+1} < u \}$, $B_{\nu_{j},s,t}(l) \coloneqq \{ Z_{\nu_{j}+t}=k, Z_{\nu_{j}+t-1} > l, \dots, Z_{\nu_{j}+s+1} > l \}$, we have that
\begin{align*}
    &\bP(Z_{\nu_{j+1}}=k | Z_{\nu_{0}}=i_{0},\dots,Z_{\nu_{j}}=i_{j},L_{j+1},U_{j+1},\pi^{L}, \pi^{U} ) \\
    =& \sum_{s=1}^{\infty} \sum_{t=s+1}^{\infty} \bP(A_{\nu_{j},s}(U_{j+1}) | Z_{\nu_{j}}=i_{j},U_{j+1},\pi^{U})  \bP( B_{\nu_{j},s,t}(L_{j+1}) | A_{\nu_{j},s}, L_{j+1},\pi^{L}) \\
    =&\bP(Z_{\nu_{j+1}}=k | Z_{\nu_{j}}=i_{j},L_{j+1},U_{j+1},\pi^{L}, \pi^{U}),
\end{align*}
where in the second line we have used that $\{ Z_{n} \}_{n=\nu_{j}}^{\infty}$ is a supercritical BPRE with immigration until it crosses the threshold $U_{j+1}$ at time $\tau_{j+1}=\nu_{j}+s$, and similarly $\{ Z_{n} \}_{n=\tau_{j+1}}^{\infty}$ is a subcritical BPRE until it crosses the threshold $L_{j+1}$ at time $\nu_{j+1}=\nu_{j}+t$. By taking expectation on both sides, we obtain that 
\begin{equation*}
    \bP(Z_{\nu_{j+1}}=k | Z_{\nu_{0}}=i_{0},\dots,Z_{\nu_{j}}=i_{j})=\bP(Z_{\nu_{j+1}}=k | Z_{\nu_{j}}=i_{j}).
\end{equation*}
Turning to the time homogeneity property, we obtain from Lemma \ref{lemma:stationarity} (iii) that 
\begin{align*}
    \bP(Z_{\nu_{j+1}}=k | Z_{\nu_{j}}=i_{j}) &= \sum_{l=L_{U}+1}^{\infty} \bP(Z_{\nu_{j+1}}=k | Z_{\tau_{j+1}}=l) \bP(Z_{\tau_{j+1}}=l | Z_{\nu_{j}}=i_{j}) \\
    &= \sum_{l=L_{U}+1}^{\infty} \bP(Z_{\nu_{1}}=k | Z_{\tau_{1}}=l) \bP(Z_{\tau_{1}}=l | Z_{\nu_{0}}=i_{j}) \\
    &=\bP(Z_{\nu_{1}}=k | Z_{\nu_{0}}=i_{j}).
\end{align*}
The proof for $\{ Z_{\tau_{j}} \}_{j=1}^{\infty}$ is similar.
\end{proof}

\subsection[Uniform ergodicity]{Uniform ergodicity of $\{Z_{\nu_{j}}\}_{j=0}^{\infty}$ and $\{ Z_{\tau_{j}} \}_{j=1}^{\infty}$} \label{subsection:uniform_ergodicity_of_Z_tau_and_Z_nu}

In this subsection we prove Theorem \ref{theorem:uniform_ergodicity}. The proof relies on the following lemma. We denote by $p_{ik}^{L}(j)=\bP_{\delta_{i}^{L}}(Z_{\nu_{j}}=k)$, $i,k \in S^{L}$, and $p_{ik}^{U}(j)=\bP_{\delta_{i}^{U}}(Z_{\tau_{j+1}}=k)$, $i,k \in S^{U}$, the $j$-step transition probability of the (time homogeneous) Markov chains $\{ Z_{\nu_{j}} \}_{j=0}^{\infty}$ and $\{ Z_{\tau_{j}} \}_{j=1}^{\infty}$. For $j=1$, we also write $p_{ik}^{L}=p_{ik}^{L}(1)$ and $p_{ik}^{U}=p_{ik}^{U}(1)$. Finally, let $p_{i}^{L}(j)=\{ p_{ik}^{L}(j) \}_{k \in S^{L}}$ and $p_{i}^{U}(j)=\{ p_{ik}^{U}(j) \}_{k \in S^{U}}$ be the $j$-step transition probability of the Markov chains $\{ Z_{\nu_{j}} \}_{j=0}^{\infty}$ and $\{ Z_{\tau_{j}} \}_{j=1}^{\infty}$ from state $i \in S^{L}$ (resp.\ $i \in S^{U}$).

\begin{lem} \label{lemma:aperiodic_irreducible_and_positive_recurrent}
Assume \ref{H1}-\ref{H4}. Then, (i) if \ref{H5} also holds, then $p_{ik}^{L} \geq \underline{p}^{L}>0$ for all $i,k \in S^{L}$. Also, (ii) if \ref{H6} (or \ref{H7}) holds, then $p_{ik}^{U} \geq \underline{p}_{k}^{U}>0$ for all $i,k \in S^{U}$.
\end{lem}
\begin{proof}[Proof of Lemma \ref{lemma:aperiodic_irreducible_and_positive_recurrent}]
The idea of proof is to establish a lower bound on $p_{ik}^{L}$ and $p_{ik}^{U}$ using \eqref{transition_probability_L} and \eqref{transition_probability_U} below, respectively. We begin by proving (i). Using Assumptions \ref{H5}, let $A$ be a measurable subset of $\mathcal{P}$ satisfying $\bP_{E^{L}}(\Pi_{0}^{L} \in A)>0$ and $P_{0,r}^{L} > 0$ for $r=0,1$ and $\Pi_{0}^{L} \in A$. By the law of total expectation
\begin{equation} \label{transition_probability_L}
    p_{ik}^{L} = \bE_{\delta_{i}^{L}}[\bP(Z_{\nu_{1}}=k | Z_{\tau_{1}}, L_{1},\pi^{L})].
\end{equation}
Since $\bP(Z_{\nu_{1}}=k | Z_{\tau_{1}}, L_{1},\pi^{L})=0$  on the event $\{L_{1} < k\}$, it follows that
\begin{equation*}
    \bP(Z_{\nu_{1}}=k | Z_{\tau_{1}}, L_{1},\pi^{L}) = \bP(Z_{\nu_{1}}=k | Z_{\tau_{1}}, L_{1},\pi^{L}) \mathbf{I}_{\{L_{1} \geq k\}}.
\end{equation*}
Now, notice that on the event $\{L_{1} \geq k\}$ the term $\bP(Z_{\nu_{1}}=k | Z_{\tau_{1}}, L_{1},\pi^{L})$ is bounded below by the probability of reaching state $k$ from $Z_{\tau_{1}}$ in one step; that is, 
\begin{equation*}
    \bP(Z_{\nu_{1}}=k | Z_{\tau_{1}}, L_{1},\pi^{L}) \geq  \bP(Z_{\nu_{1}}=k, \nu_{1}=\tau_{1}+1 | Z_{\tau_{1}}, L_{1},\pi^{L}) \mathbf{I}_{\{L_{1} \geq k\}}.
\end{equation*}
The RHS of the above inequality is bounded below by the probability that the first $k$ individuals have exactly one offspring and the remaining $Z_{\tau_{1}}-k$ have no offspring, that is,
\begin{equation*}
    \mathbf{I}_{\{L_{1} \geq k\}} \prod_{r=1}^{k} \bP(\xi_{\tau_{1},r}^{L}=1 | \Pi_{\tau_{1}}^{L}) \prod_{r=k+1}^{Z_{\tau_{1}}} \bP(\xi_{\tau_{1},r}^{L}=0 | \Pi_{\tau_{1}}^{L}).
\end{equation*}    
Once again, using that conditional on the environment $\Pi_{\tau_{1}}^{L}$, $\xi_{\tau_{1},r}^{L}$ are i.i.d., this is equal to
\begin{equation*}
    \mathbf{I}_{\{L_{1} \geq k\}} (P_{\tau_{1},1}^{L})^{k} (P_{\tau_{1},0}^{L})^{Z_{\tau_{1}}-k}.
\end{equation*}
Since $\mathbf{I}_{\{\Pi_{\tau_{1}}^{L} \in A\}} \leq 1$ and $\{L_{1} \geq k\} \supset \{L_{1}=L_{U}\}$ because $k \leq L_{U}$, the last term is bounded below by
\begin{equation*}
   \mathbf{I}_{\{L_{1} = L_{U}\}} (P_{\tau_{1},1}^{L})^{k} (P_{\tau_{1},0}^{L})^{Z_{\tau_{1}}-k} \mathbf{I}_{\{\Pi_{\tau_{1}}^{L} \in A\}}.
\end{equation*}
Finally, again using that $\Pi_{n}^{L}$ are i.i.d.\ and taking the expectation $\bE_{\delta_{i}^{L}}[\cdot]$ as in \eqref{transition_probability_L}, we obtain that $p_{ik}^{L} \geq \underline{p}_{ik}^{L}$, where
\begin{equation*}
    \underline{p}_{ik}^{L} \coloneqq \bP(L_{1}=L_{U}) \bE_{\delta_{i}^{L}}[(P_{\tau_{1},1}^{L})^{k} (P_{\tau_{1},0}^{L})^{Z_{\tau_{1}}-k} \mathbf{I}_{\{\Pi_{\tau_{1}}^{L} \in A\}}].
\end{equation*}
Notice that $\underline{p}_{k}^{L}$ is positive because $\bP(L_{1}=L_{U}) > 0$, $\bP_{E^{L}}(\Pi_{0}^{L} \in A)>0$, $P_{0,r}^{L} > 0$ for $r=0,1$ and $\Pi_{0}^{L} \in A$, and the environments $\Pi_{n}^{L}$ are i.i.d. Finally, since $S^{L}$ is finite, $\underline{p}^{L} \coloneqq \min_{i,k \in S^{L}} \underline{p}_{ik}^{L} > 0$.

We now turn to the proof of (ii), which is similar to the proof of (i). Using \ref{H6}, let $A$ and $B$ be measurable subsets of $\mathcal{P} \times \mathcal{P}$ satisfying:
\begin{itemize}
\item[(a)] $\bP_{E^{U}}( \Pi_{0}^{U} \in A) > 0$ and $Q_{0,0}^{U}>0$, $P_{0,r}^{U}>0$ for all $r \in \mathbb{N}_{0}$ and $\Pi_{0}^{U} \in A$; and
\item[(b)] $\bP_{E^{U}}( \Pi_{0}^{U} \in B) > 0$ and $Q_{0,s}^{U}>0$ for (a fixed) $s \in \{1,\dots, L_{u}\}$ and $\Pi_{0}^{U} \in B$.
\end{itemize}
Again using the law of total expectation, we obtain
\begin{equation} \label{transition_probability_U}
    p_{ik}^{U} = \bE_{\delta_{i}^{U}}[\bP(Z_{\tau_{2}}=k | Z_{\nu_{1}}, U_{2},\pi^{U})].
\end{equation}
Since $\bP(Z_{\tau_{2}}=k | Z_{\nu_{1}}, U_{2},\pi^{U})=0$  on the event $\{U_{2} > k\}$, it follows that
\begin{equation} \label{conditional_transition_probability_U}
    \bP(Z_{\tau_{2}}=k | Z_{\nu_{1}}, U_{2},\pi^{U}) = \sum_{z=0}^{L_{U}} \bP(Z_{\tau_{2}}=k | Z_{\nu_{1}}, U_{2},\pi^{U}) \mathbf{I}_{\{U_{2} \leq  k\}} I_{\{Z_{\nu_{1}}=z\}}.
\end{equation}
If $U_{2} \leq k$ and $Z_{\nu_{1}}=z>0$, then $\bP(Z_{\tau_{2}}=k | Z_{\nu_{1}}, U_{2},\pi^{U})$ is bounded below by the probability that $z$ individuals have a total of exactly $k$ offspring and no immigration occurs; that is,
\begin{equation*}
    \bP(Z_{\tau_{2}}=k | Z_{\nu_{1}}, U_{2},\pi^{U}) \mathbf{I}_{\{U_{2} \leq  k\}} I_{\{Z_{\nu_{1}}=z\}} \geq \bP( \sum_{r=1}^{z} \xi_{\nu_{1},r}^{U}=k, I_{\nu_{1}}^{U}=0 | Z_{\nu_{1}}, U_{2},\pi^{U}) \mathbf{I}_{\{U_{2} \leq  k\}} I_{\{Z_{\nu_{1}}=z\}}.
\end{equation*}
The RHS of the above inequality is bounded below by the probability that the first $z_{1} \coloneqq (k_{1}+1)z-k$ individuals have $k_{1} \coloneqq \lfloor \frac{k}{z} \rfloor$ offspring and the last $z_{2} \coloneqq z-z_{1}$ individuals have $k_{2} \coloneqq k_{1}+1$ offspring (indeed $k_{1} z_{1} + k_{2} z_{2}=k$) and no immigration occurs; that is,
\begin{equation*}
    \bP( \cap_{r=1}^{z_{1}} \{\xi_{\nu_{1},r} =k_{1}\}, \cap_{r=z_{1}+1}^{z_{2}} \{\xi_{\nu_{1},r} =k_{2}\}, I_{\nu_{1}}^{U}=0 | Z_{\nu_{1}}, U_{2},\pi^{U}) \mathbf{I}_{\{U_{2} \leq  k\}} I_{\{Z_{\nu_{1}}=z\}}.
\end{equation*}
Using that conditional on the environment $\Pi_{\nu_{1}}^{U}$, $\xi_{\nu_{1},r}^{U}$ are i.i.d., the above is equal to
\begin{equation} \label{transition_probability_U_zeta_positive}
    (P_{\nu_{1},k_{1}}^{U})^{z_{1}} (P_{\nu_{1},k_{2}}^{U})^{z_{2}} Q_{\nu_{1},0}^{U} \mathbf{I}_{\{U_{2} \leq  k\}} I_{\{Z_{\nu_{1}}=z\}}.
\end{equation}
Next, if $U_{2} \leq k$ and $Z_{\nu_{1}}=z=0$ then $\bP(Z_{\tau_{2}}=k | Z_{\nu_{1}}, U_{2},\pi^{U})$ is bounded below by the probability $\bP(Z_{\tau_{2}}=k, \tau_{2}=\nu_{1}+2 | Z_{\nu_{1}}, U_{2},\pi^{U})$. Now, this probability is bounded below by the probability that there are $s$ immigrants at time $\nu_{1}+1$, these immigrants have a total of exactly $k$ offspring, and no immigration occurs at time $\nu_{1}+2$; that is,
\begin{equation*}
    \bP (I_{\nu_{1}}^{U}=s, \sum_{r=1}^{s} \xi_{\nu_{1}+1,r}^{U}=k, I_{\nu_{1}+1}^{U}=0 | Z_{\nu_{1}}, U_{2},\pi^{U} ) \mathbf{I}_{\{U_{2} \leq  k\}} I_{\{Z_{\nu_{1}}=0\}}.
\end{equation*}
As before, this last probability is bounded below by probability that $s_{1} \coloneqq (t_{1}+1)s-k$ individuals have $t_{1} \coloneqq \lfloor \frac{k}{s} \rfloor$ offspring and $s_{2} \coloneqq s-s_{1}$ individuals have $t_{2} \coloneqq t_{1}+1$ offspring. Thus, the above probability is bounded below by
\begin{equation} \label{transition_probability_U_zeta_zero}
    Q_{\nu_{1},s}^{U} (P_{\nu_{1}+1,t_{1}}^{U})^{s_{1}} (P_{\nu_{1}+1,t_{2}}^{U})^{s_{2}} Q_{\nu_{1}+1,0}^{U} \mathbf{I}_{\{U_{2} \leq k\}} \mathbf{I}_{\{Z_{\nu_{1}} = 0\} }.
\end{equation}
Combing \eqref{conditional_transition_probability_U}, \eqref{transition_probability_U_zeta_positive}, and \eqref{transition_probability_U_zeta_zero} and using that $\mathbf{I}_{\{\Pi_{\nu_{1}}^{U} \in A\}}, \mathbf{I}_{\{\Pi_{\nu_{1}+1}^{U} \in A\}}, \mathbf{I}_{\{\Pi_{\nu_{1}}^{U} \in B\}} \leq 1$, we obtain that $\bP(Z_{\tau_{2}}=k | Z_{\nu_{1}}, U_{2},\pi^{U})$ is bounded below by
\begin{align*}
    & \mathbf{I}_{\{ U_{2} \leq k\}} \sum_{z=1}^{L_{U}} \mathbf{I}_{\{Z_{\nu_{1}} = z\} } (P_{\nu_{1},k_{1}}^{U})^{z_{1}} (P_{\nu_{1},k_{2}}^{U})^{z_{2}} Q_{\nu_{1},0}^{U} \mathbf{I}_{\{\Pi_{\nu_{1}}^{U} \in A\}} \\
  +&\mathbf{I}_{\{U_{2} \leq k\}} \mathbf{I}_{\{Z_{\nu_{1}} = 0\} } Q_{\nu_{1},s}^{U} (P_{\nu_{1}+1,t_{1}}^{U})^{s_{1}} (P_{\nu_{1}+1,t_{2}}^{U})^{s_{2}} Q_{\nu_{1}+1,0}^{U} \mathbf{I}_{\{\Pi_{\nu_{1}}^{U} \in B\}} \mathbf{I}_{\{\Pi_{\nu_{1}+1}^{U} \in A\}}.
\end{align*}
Using that $\Pi_{n}^{U}$ are i.i.d.\ and taking the expectation $\bE_{\delta_{i}^{U}}[\cdot]$ as in \eqref{transition_probability_U}, we obtain that
\begin{equation*}
    p_{ik}^{U} \geq \bP(U_{1} \leq k) \sum_{z=0}^{L_{U}} H_{k}(z) \bP_{\delta_{i}^{U}}(Z_{\nu_{1}}=z),
\end{equation*}
where $H_{k}: S^{L} \to \mathbb{R}$ is given by
\begin{equation*}
  H_{k}(z) = \begin{cases}
    \bE[Q_{\nu_{1},s}^{U} (P_{\nu_{1}+1,t_{1}}^{U})^{s_{1}} (P_{\nu_{1}+1,t_{2}}^{U})^{s_{2}} Q_{\nu_{1}+1,0}^{U} \mathbf{I}_{\{\Pi_{\nu_{1}}^{U} \in B\}} \mathbf{I}_{\{\Pi_{\nu_{1}+1}^{U} \in A\}}] & \text{ if } z = 0, \\
    \bE[(P_{\nu_{1},k_{1}}^{U})^{z_{1}} (P_{\nu_{1},k_{2}}^{U})^{z_{2}} Q_{\nu_{1},0}^{U} \mathbf{I}_{\{\Pi_{\nu_{1}}^{U} \in A\}}] & \text{ if } z \neq 0.
\end{cases}
\end{equation*}
Since $\sum_{z=0}^{L_{U}} \bP_{\delta_{i}^{U}}(Z_{\nu_{1}}=z) = 1$, we conclude that $p_{ik}^{U} \geq \underline{p}_{k}^{U}$, where
\begin{equation*}
    \underline{p}_{k}^{U} \coloneqq \bP(U_{1} \leq k) \min_{z \in S^{L}} H_{k}(z) > 0.
\end{equation*}
This concludes the proof of (ii). If instead of \ref{H6}, \ref{H7} holds the proof is similar by noticing that for all $z \in S^{L}$, on the event $\{U_{1} \leq k\} \cap \{Z_{\nu_{1}}=z\}$, there's a positive probability that at time $\nu_{1}+1$ there are $k$ immigrants and the $z$ individuals have no offspring. A detailed proof can be obtained in the same manner.
\end{proof}

Before we turn to the proof of Theorem \ref{theorem:uniform_ergodicity} we introduce few notations. Let $T_{i,0}^{L} \coloneqq 0$ and $T_{i,l}^{L} \coloneqq \inf \{ j > T_{i,l-1}^{L} : Z_{\nu_{j}}=i  \}$, $l \geq 1$, be the random times in which the Markov chain $\{ Z_{\nu_{j}} \}_{j=0}^{\infty}$ enters state $i \in S^{L}$ when the initial state is $Z_{\nu_{0}}=i$. Similarly, we let $T_{i,0}^{U} \coloneqq 1$ and $T_{i,l}^{U} \coloneqq \inf \{ j > T_{i,l-1}^{U} : Z_{\tau_{j}}=i  \}$, $l \geq 1$, be the random times in which the Markov chain $\{ Z_{\tau_{j}} \}_{j=1}^{\infty}$ enters state $i \in S^{U}$ when the initial state is $Z_{\tau_{1}}=i$. The expected times to visit state $k$ starting from $i$ is denoted by $f_{ik}^{L} \coloneqq \bE_{\delta_{i}^{L}}[T_{k,1}^{L}]$ and $f_{ik}^{U} \coloneqq \bE_{\delta_{i}^{U}}[T_{k,1}^{U}-1]$, respectively.

\begin{proof}[Proof of Theorem \ref{theorem:uniform_ergodicity}]
Lemma \ref{lemma:aperiodic_irreducible_and_positive_recurrent} implies that the state spaces of $\{ Z_{\nu_{j}} \}_{j=0}^{\infty}$ and $\{ Z_{\tau_{j}} \}_{j=1}^{\infty}$ are $S^{L}$ and $S^{U}$, respectively. Next, to establish ergodicity of these Markov chains, it is sufficient to verify irreducibility, aperiodicity, and positive recurrence. Irreducibility and aperiodicity follow from Lemma \ref{lemma:aperiodic_irreducible_and_positive_recurrent} in both cases. Now, turning to positive recurrence let $I_{n}^{L}(k) \coloneqq \{ (i_{1},i_{2}, \dots, i_{n-1}) : i_{j} \in S^{L} \setminus \{ k \} \}$ for $k \in S^{L}$. Then, using Markov property and (i) of Lemma \ref{lemma:aperiodic_irreducible_and_positive_recurrent}, it follows that
\begin{align*}
    f_{kk}^{L} &=\sum_{n=1}^{\infty} \bP_{\delta_{k}^{L}}(T_{k,1}^{L} \geq n) =\sum_{n=1}^{\infty} \bP_{\delta_{k}^{L}}( \cap_{j=1}^{n-1} \{Z_{\nu_{j}} \neq k\}) \\
    &=\sum_{n=1}^{\infty} p_{k i_{1}} \prod_{I_{n}^{L}(k)} p_{i_{j-1}i_{j}} \leq \sum_{n=1}^{\infty} (1-\underline{p}^{L})^{n-1} < \infty.
\end{align*}
Now by the finiteness of $S^{L}$ it follows that $\{ Z_{\nu_{j}} \}_{j=0}^{\infty}$ is uniformly ergodic. Next, as above for all $k \in S^{U}$
\begin{equation*}
    f_{kk}^{U}=\sum_{n=1}^{\infty} \bP_{\delta_{k}^{U}}(T_{k,1}^{U}-1 \geq n) \leq \sum_{n=1}^{\infty} (1-\underline{p}_{k}^{U})^{n-1} < \infty.
\end{equation*}
To complete the proof of uniform ergodicity of $\{ Z_{\tau_{j}} \}_{j=1}^{\infty}$, we will verify the Doeblin's condition for one-step transition: that is, for a probability distribution $q=\{ q_{k} \}_{k \in S^{U}}$ and every set $A \subset S^{U}$ satisfying $\sum_{k \in A} q_{k} > \epsilon$
\begin{equation*}
    \inf_{l \in S^{U}} ( \sum_{k \in A} p_{lk}^{U}) > \delta.
\end{equation*}
Now, taking $q_{k} \coloneqq (\underline{p}_{k}^{U})/(\sum_{k \in S} \underline{p}_{k}^{U})$, it follows from Lemma \ref{lemma:aperiodic_irreducible_and_positive_recurrent} (ii) that
\begin{equation*}
    \inf_{l \in S^{U}} ( \sum_{k \in A} p_{lk}^{U} ) \geq (\sum_{k \in S} \underline{p}_{k}^{U}) \sum_{k \in A} q_{k}.
\end{equation*}
Choosing $\delta=(\sum_{k \in S} \underline{p}_{k}^{U}) \epsilon$ uniform ergodicity of $\{ Z_{\tau_{j}} \}_{j=1}^{\infty}$ follows.
\end{proof}

\begin{Rk} \label{remark:stationary_distribution}
An immediate consequence of the above theorem is that $\{ Z_{\nu_{j}} \}_{j=0}^{\infty}$ possesses a proper stationary distribution $\pi^{L}=\{ \pi_{k}^{L} \}_{k \in S^{L}}$, where $\pi_{k}^{L} \coloneqq 1/f_{kk}^{L}>0$ and satisfies $\pi_{k}^{L}=\sum_{i \in S}  \pi_{i}^{L} p_{ik}^{L}$ for all $k \in S^{L}$. Furthermore, $\lim_{j \to \infty} \sup_{l \in S^{L}} \norm{p_{l}^{L}(j)- \pi^{L}} = 0$ where $\norm{\cdot}$ denotes the total variation norm. Furthermore, under a finite second moment hypothesis the central limit theorem holds for functions of $Z_{\nu_{j}}$. Similar comment also holds for $\{ Z_{\tau_{j}} \}_{j=1}^{\infty}$ with $L$ replaced by $U$. 
\end{Rk}

\begin{Rk} \label{remark:link_between_stationary_distributions}
It is worth noticing that the stationary distributions $\pi^{L}$ and $\pi^{U}$ are connected using $\pi_{i}^{L}=\sum_{l \in S^{U}} \bP_{\delta_{l}^{U}}(Z_{\nu_{1}}=i) \pi_{l}^{U}$ for all $i \in S^{L}$ since by time homogeneity (Lemma \ref{lemma:stationarity}) $\bP(Z_{\nu_{j+1}}=i | Z_{\tau_{j+1}}=l)=\bP_{\delta_{l}^{U}}(Z_{\nu_{1}}=i)$. Now, taking the limit as $j \to \infty$ in
\begin{equation*}
    \bP(Z_{\nu_{j+1}}=i) = \sum_{l \in S^{U}} \bP(Z_{\nu_{j+1}}=i | Z_{\tau_{j+1}}=l) \bP(Z_{\tau_{j+1}}=l),
\end{equation*}
the above expression follows. Similarly, $\pi_{k}^{U}=\sum_{l \in S^{L}} \bP_{\delta_{l}^{L}}(Z_{\tau_{1}}=k) \pi_{l}^{L}$ for all $k \in S^{U}$.
\end{Rk}

Since the state space $S^{L}$ is finite $\pi^{L}$ has moments of all orders. Proposition \ref{proposition:finite_expectation_stationary_distribution} in Appendix \ref{subsection:finiteness_of_mean_stationary_distribution} shows that $\pi^{U}$ has a finite first moment $\overline{\pi}^{U} \coloneqq \sum_{k \in S^{U}} k \pi_{k}^{U}$.

\section{Regenerative property of crossing times} \label{section:regenerative_property_of_crossing_times}

In this section, we establish the law of large numbers and central limit theorem for the length of the supercritical and subcritical regimes $\{ \Delta_{j}^{U} \}_{j=1}^{\infty}$ and $\{ \Delta_{j}^{L} \}_{j=1}^{\infty}$. To this end, we will show that $\{ \Delta_{j}^{U} \}_{j=1}^{\infty}$ and $\{ \Delta_{j}^{L} \}_{j=1}^{\infty}$ are regenerative over the times $\{ T_{i,l}^{L} \}_{l=0}^{\infty}$ and $\{ T_{i,l}^{U} \}_{l=1}^{\infty}$, respectively. In our analysis we will also encounter the random variables $\overline{\Delta}_{j}^{U} \coloneqq \Delta_{j}^{U}+\Delta_{j}^{L}$ and $\overline{\Delta}_{j}^{L} \coloneqq \Delta_{j}^{L}+\Delta_{j+1}^{U}$. For $l \geq 1$ and $i \in S^{L}$ let $B_{i,l}^{L} \coloneqq (K_{i,l}^{L}, \bD_{i,l}^{L}, \overline{\bD}_{i,l}^{L})$, where
\begin{equation*}
K_{i,l}^{L} \coloneqq T_{i,l}^{L}-T_{i,l-1}^{L}, \bD_{i,l}^{L} \coloneqq ( \Delta_{T_{i,l-1}^{L}+1}^{U}, \dots, \Delta_{T_{i,l}^{L}}^{U}), \text{ and } \overline{\bD}_{i,l}^{L} \coloneqq ( \overline{\Delta}_{T_{i,l-1}^{L}+1}^{U}, \dots, \overline{\Delta}_{T_{i,l}^{L}}^{U}).
\end{equation*}
The triple $B_{i,l}^{L}$ consists of the random time $K_{i,l}^{L}$ required by $\{ Z_{\nu_{j}} \}_{j=0}^{\infty}$ to return for the $l^{\text{th}}$ time to state $i$, the lengths of all supercritical regimes $\Delta_{j}^{U}$ between the $(l-1)^{\text{th}}$ return and the $l^{\text{th}}$ return, and the lengths $\overline{\Delta}_{j}^{U}$ of both regimes in the same time interval. Similarly, for $l \geq 1$ and $i \in S^{U}$ we let $B_{i,l}^{U} \coloneqq (K_{i,l}^{U}, \bD_{i,l}^{U}, \overline{\bD}_{i,l}^{U})$, where
\begin{equation*}
  K_{i,l}^{U} \coloneqq T_{i,l}^{U}-T_{i,l-1}^{U}, \bD_{i,l}^{U} \coloneqq ( \Delta_{T_{i,l-1}^{U}}^{L}, \dots, \Delta_{T_{i,l}^{U}-1}^{L}), \text{ and } \overline{\bD}_{i,l}^{U} \coloneqq ( \overline{\Delta}_{T_{i,l-1}^{U}}^{L}, \dots, \overline{\Delta}_{T_{i,l}^{U}-1}^{L}).
\end{equation*}
The proof of the following lemma is included in Appendix \ref{subsection:proofs_of_lemmas}.

\begin{lem} \label{lemma:regenerative_property_for_crossing_times}
Assume \ref{H1}-\ref{H4}. (i) If \ref{H5} also holds and $Z_{\nu_{0}}=i \in S^{L}$, then $\{ B_{i,l}^{L} \}_{l=1}^{\infty}$ are i.i.d. (ii) If \ref{H6} (or \ref{H7}) holds and $Z_{\tau_{1}}=i \in S^{U}$, then $\{ B_{i,l}^{U} \}_{l=1}^{\infty}$ are i.i.d.
\end{lem}
The proof of the following lemma, which is required in the proof of the Theorem \ref{theorem:law_of_large_numbers_and_central_limit_theorem_for_crossing_times}, is also included in Appendix \ref{subsection:proofs_of_lemmas}. We need the following additional notations: $\overline{S}_{n}^{U} \coloneqq \sum_{j=1}^{n} \overline{\Delta}_{j}^{U}$, $\overline{S}_{n}^{L} \coloneqq \sum_{j=1}^{n} \overline{\Delta}_{j}^{L}$,
\begin{align*}
    &\overline{\sigma}^{2,U} \coloneqq \bV_{\pi^{L}}[\overline{\Delta}_{1}^{U}] + 2\sum_{j=1}^{\infty} \bC_{\pi^{L}}[\overline{\Delta}_{1}^{U}, \overline{\Delta}_{j+1}^{U}], \\
    &\overline{\sigma}^{2,L} \coloneqq \bV_{\pi^{U}}[\overline{\Delta}_{1}^{L}] +2 \sum_{j=1}^{\infty} \bC_{\pi^{U}}[\overline{\Delta}_{1}^{L}, \overline{\Delta}_{j+1}^{L}],  \\
    &\mathbb{C}^{U} \coloneqq \sum_{j=0}^{\infty} \bC_{\pi^{L}}[ \Delta_{1}^{U}, \overline{\Delta}_{j+1}^{U}] + \sum_{j=1}^{\infty} \bC_{\pi^{L}}[ \overline{\Delta}_{1}^{U}, \Delta_{j+1}^{U}], \text{ and} \\
    &\mathbb{C}^{L} \coloneqq \sum_{j=0}^{\infty} \bC_{\pi^{U}}[\Delta_{1}^{L}, \overline{\Delta}_{j+1}^{L}] + \sum_{j=1}^{\infty} \bC_{\pi^{U}}[\overline{\Delta}_{1}^{L}, \Delta_{j+1}^{L}].
\end{align*}

\begin{lem} \label{lemma:variance_of_SnL_and_SnU}
Under the assumptions of Theorem \ref{theorem:law_of_large_numbers_and_central_limit_theorem_for_crossing_times}, for all $i \in S^{L}$ the following hold:
\begin{align*}
  &\text{(i) } \bE_{\delta_{i}^{L}}[S_{T_{i,1}^{L}}^{U}] = (\pi_{i}^{L})^{-1} \mu^{U} \text{ and } \bE_{\delta_{i}^{L}}[\overline{S}_{T_{i,1}^{L}}^{U}] = (\pi_{i}^{L})^{-1} (\mu^{U}+\mu^{L}), \\
  &\text{(ii) } \bV_{\delta_{i}^{L}}[S_{T_{i,1}^{L}}^{U}] = (\pi_{i}^{L})^{-1} \sigma^{2,U} \text{ and } \bV_{\delta_{i}^{L}}[\overline{S}_{T_{i,1}^{L}}^{U}] = (\pi_{i}^{L})^{-1} \overline{\sigma}^{2,U}; \text{ and} \\
  &\text{(iii) } \bC_{\delta_{i}^{L}}[S_{T_{i,1}^{L}}^{U},\overline{S}_{T_{i,1}^{L}}^{U}] = (\pi_{i}^{L})^{-1} \mathbb{C}^{U}.
\end{align*}
The above statements also hold with $U$ replaced by $L$.
\end{lem}
Proposition \ref{proposition:mean_and_variance_are_finite_and_variance_is_positive} in Appendix \ref{subsection:finiteness_of_means_and_variances} shows that $\sigma^{2,T}$ and $\overline{\sigma}^{2,T}$ are positive and finite, and $\abs{\mathbb{C}^{T}}<\infty$. We are now ready to prove Theorem \ref{theorem:law_of_large_numbers_and_central_limit_theorem_for_crossing_times}. The proof relies on decomposing $S_{n}^{U}$ and $S_{n}^{L}$ into i.i.d.\ cycles using Lemma \ref{lemma:regenerative_property_for_crossing_times}. Specifically, conditionally on $Z_{\nu_{0}}=i \in S^{L}$ (resp.\ $Z_{\tau_{1}}=i \in S^{U}$), the random variables $\{ S_{T_{i,l}^{L}}^{U}-S_{T_{i,l-1}^{L}}^{U} \}_{l=1}^{\infty}$ (resp.\ $\{ S_{T_{i,l}^{U}-1}^{L}-S_{T_{i,l-1}^{U}-1}^{L} \}_{l=1}^{\infty}$) are i.i.d.

\begin{proof}[Proof of Theorem \ref{theorem:law_of_large_numbers_and_central_limit_theorem_for_crossing_times}]
  We begin by proving (i). For $i  \in S^{L}$ and $n \in \mathbb{N}$, let $N_{i}^{L}(n) \coloneqq \sum_{l=1}^{\infty} \mathbf{I}_{ \{T_{i,l}^{L} \leq n\}}$ be the number of times $T_{i,l}^{L}$ is in $\{ 0,1,\dots,n \}$. Conditionally on $Z_{\nu_{0}}=i$, notice that $N_{i}^{L}(n)$ is a renewal process (recall that $T_{i,0}^{L}=0$). We recall that $K_{i,l}^{L} = T_{i,l}^{L}-T_{i,l-1}^{L}$ and let $K_{i,n}^{*,L} \coloneqq n-T_{i,N_{i}^{L}(n)}$, $R_{i,l}^{L} \coloneqq S_{T_{i,l}^{L}}^{U}-S_{T_{i,l-1}^{L}}^{U}$, and $R_{i,n}^{*,L} \coloneqq S_{n}^{U}-S_{T_{i,N_{i}^{L}(n)}^{L}}^{U}$.  Using the decomposition
\begin{equation} \label{decomposition_SnL}
    \frac{1}{n} S_{n}^{U} = \frac{N_{i}^{L}(n)}{n} \biggl( \frac{1}{N_{i}^{L}(n)} \sum_{l=1}^{N_{i}^{L}(n)} R_{i,l}^{L} + \frac{1}{N_{i}^{L}(n)} R_{i,n}^{*,L} \biggr),
\end{equation}
$\{ R_{i,l}^{L} \}_{l=1}^{\infty}$ are i.i.d.\ and $\lim_{n \to \infty} N_{i}^{L}(n) = \infty$ a.s., we obtain using the law of large numbers for random sums and Lemma \ref{lemma:variance_of_SnL_and_SnU} (i) that
\begin{equation} \label{law_of_large_numbers}
  \lim_{n \to \infty} \frac{1}{N_{i}^{L}(n)} \sum_{l=1}^{N_{i}^{L}(n)} R_{i,l}^{L} = \bE_{\delta_{i}^{L}}[S_{T_{i,1}^{L}}^{U}]=(\pi_{i}^{L})^{-1} \mu^{U} \text{ a.s.}
\end{equation}
Also, 
\begin{equation*}
\limsup_{n \to \infty} \frac{1}{N_{i}^{L}(n)} R_{i,n}^{*,L} \leq \lim_{n \to \infty} \frac{1}{N_{i}^{L}(n)} R_{i,N_{i}^{L}(n)+1}^{L} =0 \text{ a.s.}
\end{equation*}
Finally, using the key renewal theorem (Corollary 2.11 of \citet{Serfozo-2009}) and Remark \ref{remark:stationary_distribution}
\begin{equation} \label{frequency_of_return_to_i}
  \lim_{n \to \infty} \frac{N_{i}^{L}(n)}{n}=\frac{1}{\bE_{\delta_{i}^{L}}[T_{i,1}^{L}]}=\pi_{i}^{L} \text{ a.s.}
\end{equation}
Using \eqref{law_of_large_numbers} and \eqref{frequency_of_return_to_i} in \eqref{decomposition_SnL}, we obtain the SLLN for $S_{n}^{U}$. Turning to the central limit theorem, we let $\underline{R}_{i,l}^{L} \coloneqq R_{i,l}^{L} - \mu^{U} K_{i,l}^{L}$ and $\underline{R}_{i,n}^{*,L} \coloneqq R_{i,n}^{*,L} - \mu^{U} K_{i,n}^{*,L}$. Conditionally on $Z_{\nu_{0}}=i$, using the decomposition in \eqref{decomposition_SnL} and centering, we obtain
\begin{equation*}
   \frac{1}{\sqrt{n}}(S_{n}^{U}-n\mu^{U}) = \sqrt{\frac{N_{i}^{L}(n)}{n}} \biggl( \frac{1}{\sqrt{N_{i}^{L}(n)}} \sum_{l=1}^{N_{i}^{L}(n)} \underline{R}_{i,l}^{L} + \frac{1}{\sqrt{N_{i}^{L}(n)}} \underline{R}_{i,n}^{*,L} \biggr),
\end{equation*}
where $\{ \underline{R}_{i,l}^{L} \}_{l=1}^{\infty}$ are i.i.d.\ with mean $0$ and variance which using Lemma \ref{lemma:variance_of_SnL_and_SnU} is
\begin{equation} \label{variance_central_limit_theorem_for_crossing_times}
  \bV_{\delta_{i}^{L}}[S_{T_{i,1}^{L}}^{U}-\mu^{U}T_{i,1}^{L}] = (\pi_{i}^{L})^{-1} \sigma^{2,U}.
\end{equation}
Finally, using the central limit theorem for i.i.d.\ random sums and \eqref{frequency_of_return_to_i}, it follows that
\begin{equation*}
  \sqrt{\frac{N_{i}^{L}(n)}{n}} \biggl( \frac{1}{\sqrt{N_{i}^{L}(n)}} \sum_{l=1}^{N_{i}^{L}(n)} \underline{R}_{i,l}^{L} \biggr) \xrightarrow[n \to \infty]{d} N(0,\sigma^{2,U}).
\end{equation*}
To complete the proof notice that
\begin{equation*}
   \abs{ \frac{1}{\sqrt{N_{i}^{L}(n)}} \underline{R}_{i,n}^{*,L}} \leq \frac{1}{\sqrt{N_{i}^{L}(n)}} \abs{\underline{R}_{i,N_{i}^{L}(n)+1}^{L}} \xrightarrow[n \to \infty]{p} 0.
\end{equation*}
The proof for $S_{n}^{L}$ is similar.
\end{proof}

When studying the proportion of time the process spends in supercritical and subcritical regimes we will need the above theorem with $n$ replaced by a random time $\tilde{N}(n)$.

\begin{Rk} \label{remark:law_of_large_numbers_and_central_limit_theorem_for_crossing_times}
Theorem \ref{theorem:law_of_large_numbers_and_central_limit_theorem_for_crossing_times} holds if $n$ is replaced by a random time $\tilde{N}(n)$, where $\lim_{n \to \infty} \tilde{N}(n)=\infty$ a.s.
\end{Rk}

\section{Proportion of time spent in supercritical and subcritical regimes} \label{section:proportion_of_time_spent_in_supercritical_and_subcritical_regimes}

We recall that $\chi_{n}^U =\mathbf{I}_{\cup_{j=1}^{\infty} [\nu_{j-1},\tau_{j})}(n)$ is $1$ if the  process is in the supercritical regime and $0$ otherwise and similarly $\chi_{n}^L=1-\chi_{n}^U$. Also $\theta_{n}^{U}=\frac{1}{n} C_{n}^U$ is the proportion of time the process spends in the supercritical regime up to time $n-1$. $\theta_{n}^{L}$ is defined similarly. The limit theorems for $\theta_{n}^{U}$ and $\theta_{n}^{L}$ will invoke the i.i.d.\ blocks developed in Section \ref{section:regenerative_property_of_crossing_times}. Let $\bm{S}_{n}^{U} \coloneqq (S_{n}^{U}, \overline{S}_{n}^{U})^{\top}$, $\bm{\mu}^{U} \coloneqq (\mu^{U},\mu^{U}+\mu^{L})^{\top}$, $\bm{\mu}^{L} \coloneqq (\mu^{L},\mu^{U}+\mu^{L})^{\top}$, and
\begin{equation*}
    \Sigma^{U} \coloneqq \begin{pmatrix}
    \sigma^{2,U} & \mathbb{C}^{U} \\
    \mathbb{C}^{U} & \overline{\sigma}^{2,U}
    \end{pmatrix}, \text{ and }
    \Sigma^{L} \coloneqq \begin{pmatrix}
    \sigma^{2,L} & \mathbb{C}^{L} \\
    \mathbb{C}^{L} & \overline{\sigma}^{2,L}
    \end{pmatrix}.
\end{equation*}
We note that while $S_{n}^{U}$ represents the length of the first $n$ supercritical regimes $\overline{S}_{n}^{U}$ is the total time taken for the process to complete the first $n$ cycles.

\begin{lem} \label{lemma:multivariate_central_limit_theorem_for_crossing_times}
Under the conditions of Theorem \ref{theorem:law_of_large_numbers_and_central_limit_theorem_for_proportions}, $\frac{1}{\sqrt{n}} (\bm{S}_{n}^{U}-n\bm{\mu}^{U}) \xrightarrow[n \to \infty]{d} N(\bm{0},\Sigma^{U})$, and $\frac{1}{\sqrt{n}} (\bm{S}_{n}^{L}-n\bm{\mu}^{L}) \xrightarrow[n \to \infty]{d} N(\bm{0},\Sigma^{L})$.
\end{lem}
\begin{proof}[Proof of Lemma \ref{lemma:multivariate_central_limit_theorem_for_crossing_times}]
  The proof is similar to that of Theorem \ref{theorem:law_of_large_numbers_and_central_limit_theorem_for_crossing_times}. We let $\bm{R}_{i,l}^{L} \coloneqq \bm{S}_{T_{i,l}^{L}}^{U}-\bm{S}_{T_{i,l-1}^{L}}^{U}$, $\bm{R}_{i,n}^{*,L} \coloneqq \bm{S}_{n}^{U}-\bm{S}_{T_{i,N_{i}^{L}(n)}^{L}}^{U}$, $\underline{\bm{R}}_{i,l}^{L} \coloneqq \bm{R}_{i,l}^{L} - K_{i,l}^{L} \bm{\mu}^{U}$ and $\underline{\bm{R}}_{i,n}^{*,L} \coloneqq \bm{R}_{i,n}^{*,L} - K_{i,n}^{*,L} \bm{\mu}^{U}$. Conditionally on $Z_{\nu_{0}}=i$, we write
\begin{equation*}
   \frac{1}{\sqrt{n}} (\bm{S}_{n}^{U}-n\bm{\mu}^{U}) = \sqrt{\frac{N_{i}^{L}(n)}{n}} \biggl( \frac{1}{\sqrt{N_{i}^{L}(n)}} \sum_{l=1}^{N_{i}^{L}(n)} \underline{\bm{R}}_{i,l}^{L} + \frac{1}{\sqrt{N_{i}^{L}(n)}} \underline{\bm{R}}_{i,n}^{*,L} \biggr).
\end{equation*}
Now, by Lemma \ref{lemma:regenerative_property_for_crossing_times} and Lemma \ref{lemma:variance_of_SnL_and_SnU}, $\{ \underline{\bm{R}}_{i,l}^{L} \}_{l=1}^{\infty}$ are i.i.d.\ with mean $\bm{0}=(0,0)^{\top}$ and covariance matrix $(\pi_{i}^{L})^{-1} \Sigma^{U}$. Using the key renewal theorem we conclude that
\begin{equation*}
   \sqrt{\frac{N_{i}^{L}(n)}{n}} \biggl( \frac{1}{\sqrt{N_{i}^{L}(n)}} \sum_{l=1}^{N_{i}^{L}(n)} \underline{\bm{R}}_{i,l}^{L} \biggr)  \xrightarrow[n \to \infty]{d} N(\bm{0},\Sigma^{U})
\end{equation*}
   and
\begin{equation*}
  \frac{1}{\sqrt{N_{i}^{L}(n)}} \abs{\underline{\bm{R}}_{i,n}^{*,L}} \xrightarrow[n \to \infty]{p} 0.
\end{equation*}
The proof of (ii) is similar.
\end{proof}
\begin{Rk} \label{remark:central_limit_theorem_for_crossing_times}
  Lemma \ref{lemma:multivariate_central_limit_theorem_for_crossing_times} holds also with $n$ replaced by a random time $\tilde{N}(n)$ such that $\lim_{n \to \infty} \tilde{N}(n)=\infty$ a.s.
\end{Rk}
The next lemma concerns the number of crossings of upper and lower thresholds, namely, $\tilde{N}^{U}(n) \coloneqq \sup \{ j \geq 0 : \tau_{j} \leq n \}$ and $\tilde{N}^{L}(n) \coloneqq \sup \{ j \geq 0 : \nu_{j} \leq n \}$ where $n \in \mathbb{N}_{0}$.

\begin{lem} \label{lemma:law_of_large_numbers_for_proportion_of_times}
Under the conditions of Theorem \ref{theorem:law_of_large_numbers_and_central_limit_theorem_for_proportions},  \\
  (i) $\lim_{n \to \infty} \frac{\tilde{N}^{U}(n)}{n+1} = \frac{1}{\mu^{U}+\mu^{L}}$ and $\lim_{n \to \infty} \frac{\tilde{N}^{L}(n)}{n+1} = \frac{1}{\mu^{U}+\mu^{L}}$ a.s. \\
  (ii) $\lim_{n \to \infty} \frac{C_{n+1}^{U}}{\tilde{N}^{U}(n)} = \mu^{U}$ and $\lim_{n \to \infty} \frac{C_{n+1}^{L}}{\tilde{N}^{L}(n)} = \mu^{L}$ a.s. \\
\end{lem}
\begin{proof}[Proof Lemma \ref{lemma:law_of_large_numbers_for_proportion_of_times}]
We begin by proving (i). We recall that $\tau_{0}=-1$ and $\tau_{j} < \nu_{j} < \tau_{j+1}$ a.s.\ for all $j \geq 0$ yielding that
\begin{equation*}
  \tilde{N}^{L}(n) \leq \tilde{N}^{U}(n) \leq \tilde{N}^{L}(n)+1.
\end{equation*}
Since $\tau_{j}$ and $\nu_{j}$ are finite almost surely, we obtain that $\lim_{n \to \infty} \tilde{N}^{U}(n)=\infty$ and $\lim_{n \to \infty} \tilde{N}^{L}(n)=\infty$ a.s. (i) follows if we show that $\lim_{n \to \infty} \frac{\tilde{N}^{L}(n)}{n+1} = \frac{1}{\mu^{U}+\mu^{L}}$ a.s. To this end, we notice that $\nu_{\tilde{N}^{L}(n)} \leq n \leq \nu_{\tilde{N}^{L}(n)+1}$ and for $n \geq \nu_{1}$
\begin{equation*}
  \frac{\nu_{\tilde{N}^{L}(n)}}{\tilde{N}^{L}(n)} \leq \frac{n}{\tilde{N}^{L}(n)} \leq \frac{\nu_{\tilde{N}^{L}(n)+1}}{\tilde{N}^{L}(n)+1} \frac{\tilde{N}^{L}(n)+1}{\tilde{N}^{L}(n)}
\end{equation*}
Clearly, $\lim_{n \to \infty} \frac{\tilde{N}^{L}(n)+1}{\tilde{N}^{L}(n)} = 1$ a.s. Remark \ref{remark:law_of_large_numbers_and_central_limit_theorem_for_crossing_times} with $\tilde{N}(n)=\tilde{N}^{U}(n)$ yields that
\begin{equation*}
  \lim_{n \to \infty} \frac{\nu_{\tilde{N}^{L}(n)}}{\tilde{N}^{L}(n)}=\lim_{n \to \infty} \frac{1}{\tilde{N}^{L}(n)} \sum_{j=1}^{\tilde{N}^{L}(n)} \overline{\Delta}_{j}^{U} = \mu^{U}+\mu^{L} \text{ a.s.}
\end{equation*}
Thus, we obtain
\begin{equation} \label{inequality_for_renewal_processes}
  \lim_{n \to \infty} \frac{\tilde{N}^{L}(n)}{n+1} = \frac{1}{\mu^{U}+\mu^{L}} \text{ a.s.}
\end{equation}
Turning to (ii), we notice that
\begin{equation*}
  C_{n+1}^{U} = \sum_{j=0}^{n} \chi_{j}^{U} = \sum_{j=1}^{\tilde{N}^{U}(n)} \Delta_{j}^{U}+\sum_{l=\nu_{\tilde{N}^{U}(n)}}^{n} 1,
\end{equation*}
where $\sum_{l=r}^{n}=0$ for $r>n$. Remark \ref{remark:law_of_large_numbers_and_central_limit_theorem_for_crossing_times} with $\tilde{N}(n)=\tilde{N}^{U}(n)$ yield that
\begin{align*}
  &\lim_{n \to \infty} \frac{1}{\tilde{N}^{U}(n)} \sum_{j=1}^{\tilde{N}^{U}(n)} \Delta_{j}^{U} = \mu^{U} \text{ a.s., and } \\
  &\limsup_{n \to \infty} \frac{1}{\tilde{N}^{U}(n)} \sum_{l=\nu_{\tilde{N}^{U}(n)}}^{n} 1 \leq \lim_{n \to \infty} \frac{1}{\tilde{N}^{U}(n)} \Delta_{\tilde{N}^{U}(n)+1}^{U} = 0 \text{ a.s.}
\end{align*}
Thus, we obtain that $\lim_{n \to \infty} \frac{C_{n+1}^{U}}{\tilde{N}^{U}(n)} = \mu^{U}$ a.s. Similarly, $\lim_{n \to \infty} \frac{C_{n+1}^{L}}{\tilde{N}^{L}(n)} = \mu^{L}$ a.s.
\end{proof}

Our next result is concerned with the joint distribution of the last time the process is in a specific regime and the proportion of time the process spends in that regime under the assumptions of Theorem \ref{theorem:law_of_large_numbers_and_central_limit_theorem_for_proportions}. Let
\begin{equation*}
     \overline{\Sigma}^{U} \coloneqq \begin{pmatrix}
    \sigma^{2,U}(\mu^{U}+\mu^{L}) & -\frac{\mathbb{C}^{U}}{\mu^{U}+\mu^{L}} \\
    -\frac{\mathbb{C}^{U}}{\mu^{U}+\mu^{L}} & \frac{\overline{\sigma}^{2,U}}{(\mu^{U}+\mu^{L})^{3}}
    \end{pmatrix} \text{ and } 
    \overline{\Sigma}^{L} \coloneqq \begin{pmatrix}
    \sigma^{2,L}(\mu^{U}+\mu^{L}) & -\frac{\mathbb{C}^{L}}{\mu^{U}+\mu^{L}} \\
    -\frac{\mathbb{C}^{L}}{\mu^{U}+\mu^{L}} & \frac{\overline{\sigma}^{2,L}}{(\mu^{U}+\mu^{L})^{3}}
    \end{pmatrix}.
\end{equation*}

\begin{lem} \label{lemma:central_limit_theorem_for_frequences}
Under the conditions of Theorem \ref{theorem:law_of_large_numbers_and_central_limit_theorem_for_proportions}, 
\begin{align*}
\sqrt{n+1}
\begin{pmatrix}
    \frac{C_{n+1}^{U}}{\tilde{N}^{U}(n)} - \mu^{U} \\
    \frac{\tilde{N}^{U}(n)}{n+1}-\frac{1}{\mu^{U}+\mu^{L}}
\end{pmatrix} &\xrightarrow[n \to \infty]{d} N(\bm{0},\overline{\Sigma}^{U}) \text{ and } 
\sqrt{n+1}
\begin{pmatrix}
    \frac{C_{n+1}^{L}}{\tilde{N}^{L}(n)} - \mu^{L} \\
    \frac{\tilde{N}^{L}(n)}{n+1}-\frac{1}{\mu^{U}+\mu^{L}}
\end{pmatrix} &\xrightarrow[n \to \infty]{d} N(\bm{0},\overline{\Sigma}^{L}).
\end{align*}
\end{lem}
\begin{proof}[Proof of Lemma \ref{lemma:central_limit_theorem_for_frequences}]
We only prove the statement for $C_{n+1}^{U}$ and $\tilde{N}^{U}(n)$ since the other case is similar. We write
\begin{equation*}
\sqrt{\tilde{N}^{U}(n)} 
\begin{pmatrix}
    \frac{C_{n+1}^{U}}{\tilde{N}^{U}(n)} - \mu^{U} \\
    \frac{n+1}{\tilde{N}^{U}(n)}-(\mu^{U}+\mu^{L})
\end{pmatrix}
=\frac{1}{\sqrt{\tilde{N}^{U}(n)}} \sum_{j=1}^{\tilde{N}^{U}(n)}
\begin{pmatrix}
    \Delta_{j}^{U}-\mu^{U} \\
    \overline{\Delta}_{j}^{U}-(\mu^{U}+\mu^{L}) 
\end{pmatrix} + \frac{1}{\sqrt{\tilde{N}^{U}(n)}} \sum_{l=\nu_{\tilde{N}^{U}(n)}}^{n} \begin{pmatrix} 1 \\ 1 \end{pmatrix}
\end{equation*}
and using Lemma \ref{lemma:multivariate_central_limit_theorem_for_crossing_times} and Remark \ref{remark:central_limit_theorem_for_crossing_times} with $\tilde{N}(n)=\tilde{N}^{U}(n)$ we obtain that
\begin{equation*}
\sqrt{\tilde{N}^{U}(n)} 
\begin{pmatrix}
    \frac{C_{n+1}^{U}}{\tilde{N}^{U}(n)} - \mu^{U} \\
    \frac{n+1}{\tilde{N}^{U}(n)}-(\mu^{U}+\mu^{L})
\end{pmatrix}  \xrightarrow[n \to \infty]{d}  N(\bm{0},\Sigma^{U}).
\end{equation*}
Next, we apply the delta method with $g:\mathbb{R}^{2} \to \mathbb{R}^{2}$ given by $g(x,y)=(x,1/y)$ and obtain that
\begin{equation*}
\sqrt{\tilde{N}^{U}(n)} 
\begin{pmatrix}
    \frac{C_{n+1}^{U}}{\tilde{N}^{U}(n)} - \mu^{U} \\
    \frac{\tilde{N}^{U}(n)}{n+1}-\frac{1}{\mu^{U}+\mu^{L}}
\end{pmatrix}  \xrightarrow[n \to \infty]{d}  N(\bm{0},\Sigma_{2}^{U}),
\end{equation*}
where 
\begin{equation*}
    \Sigma_{2}^{U}=J_{g}(\bm{\mu}^{U}) \Sigma^{U} J_{g}(\bm{\mu}^{U})^{\top} =\begin{pmatrix}
    \sigma^{2,U} & -\frac{\mathbb{C}^{U}}{(\mu^{U}+\mu^{L})^{2}} \\
    -\frac{\mathbb{C}^{U}}{(\mu^{U}+\mu^{L})^{2}} & \frac{\overline{\sigma}^{2,U}}{(\mu^{U}+\mu^{L})^{4}}
    \end{pmatrix}
\end{equation*}
and $J_{g}(\cdot)$ is the Jacobian matrix of $g(\cdot)$. Using Lemma \ref{lemma:law_of_large_numbers_for_proportion_of_times} (i), we obtain that
\begin{equation*}
\sqrt{n+1} 
\begin{pmatrix}
    \frac{C_{n+1}^{U}}{\tilde{N}^{U}(n)} - \mu^{U} \\
    \frac{\tilde{N}^{U}(n)}{n+1}-\frac{1}{\mu^{U}+\mu^{L}}
\end{pmatrix}  \xrightarrow[n \to \infty]{d}  N(\bm{0},\overline{\Sigma}^{U}).
\end{equation*}
\end{proof}

We are now ready to prove Theorem \ref{theorem:law_of_large_numbers_and_central_limit_theorem_for_proportions}. Recall that $\theta_{n}^{U} = \frac{C_{n}^{U}}{n}$, $\theta_{n}^{L} = \frac{C_{n}^{L}}{n}$, $\theta^{U} = \frac{\mu^{U}}{\mu^{U}+\mu^{L}}$, and $\theta^{L} = \frac{\mu^{L}}{\mu^{L}+\mu^{U}}$; and let $\theta^{k,U}$ and $\theta^{k,L}$ be the $k^{\text{th}}$ power of $\theta^{U}$ and $\theta^{L}$, respectively.

\begin{proof}[Proof of Theorem \ref{theorem:law_of_large_numbers_and_central_limit_theorem_for_proportions}]
Almost sure convergence of $\theta_{n}^{T}$ follows from Lemma \ref{lemma:law_of_large_numbers_for_proportion_of_times} upon noticing that $\frac{C_{n+1}^{T}}{n+1}=(\frac{C_{n+1}^{T}}{\tilde{N}^{T}(n)} ) ( \frac{\tilde{N}^{T}(n)}{n+1})$. Using Lemma \ref{lemma:central_limit_theorem_for_frequences} and the decomposition
\begin{equation*}
    \sqrt{n+1} \biggl( \frac{C_{n+1}^{T}}{n+1} - \frac{\mu^{T}}{\mu^{U}+\mu^{L}} \biggr) = \frac{\tilde{N}^{T}(n)}{n+1} \cdot \sqrt{n+1} \biggl(\frac{C_{n+1}^{T}}{\tilde{N}^{T}(n)} - \mu^{T} \biggr) + \mu^{T} \cdot \sqrt{n+1} \biggl(\frac{\tilde{N}^{T}(n)}{n+1}-\frac{1}{\mu^{U}+\mu^{L}} \biggr),
\end{equation*}
it follows that $\sqrt{n+1} (\theta_{n+1}^{T} - \theta^{T} )$ is asymptotically normal with mean zero and variance
\begin{equation} \label{definition:eta}
\eta^{2,T} \coloneqq \frac{1}{\mu^{T}} (\sigma^{2,T}\theta^{T}-2\mathbb{C}^{T} \theta^{2,T}+\overline{\sigma}^{2,T} \theta^{3,T}).
\end{equation}
\end{proof}

\begin{Cor} \label{corollary:frequence_central_limit_theorem_for_crossing_times}
Under the conditions of Theorem \ref{theorem:law_of_large_numbers_and_central_limit_theorem_for_crossing_times}, for $T \in \{L,U\}$
\begin{equation*}
    \sqrt{\tilde{N}^{T}(n)} \biggl(\frac{C_{n+1}^{T}}{\tilde{N}^{T}(n)}- \mu^{T} \biggr) \xrightarrow[n \to \infty]{d} N(0,\sigma^{2,T}).
\end{equation*}
\end{Cor}

\begin{proof}[Proof of Corollary \ref{corollary:frequence_central_limit_theorem_for_crossing_times}]
We only prove the case $T=U$. We write
\begin{equation*}
  \sqrt{\tilde{N}^{U}(n)} \biggl(\frac{C_{n+1}^{U}}{\tilde{N}^{U}(n)}- \mu^{U}\biggr) = \frac{1}{\sqrt{\tilde{N}^{U}(n)}} \sum_{j=1}^{\tilde{N}^{U}(n)} (\Delta_{j}^{U} - \mu^{U}) + \frac{1}{\sqrt{\tilde{N}^{U}(n)}} \sum_{l=\nu_{\tilde{N}^{U}(n)}}^{n} 1.
\end{equation*}
Taking the limit in the above equation and using Remark \ref{remark:law_of_large_numbers_and_central_limit_theorem_for_crossing_times} the result follows.
\end{proof}

\section{Estimating the mean of the offspring distribution}
\label{section:estimating_the_mean_of_the_offspring_distribution}

We recall that $\tilde{\chi}_{n}^{U} = \chi_{n}^{U} \mathbf{I}_{\{Z_{n} \geq 1\}}$, $\tilde{C}_{n}^{U} = \sum_{j=1}^{n} \tilde{\chi}_{j-1}^{U}$ and set for the subcritical regime $\tilde{\chi}_{n}^{L} \coloneqq \chi_{n}^{L}$ and $\tilde{C}_{n}^{L} \coloneqq C_{n}^{L}$. We also recall the offspring mean estimate of the BPRET $\{Z_{n} \}_{n=0}^{\infty}$ in the supercritical and subcritical regimes are given by
\begin{equation*}
    M_{n}^{U} = \frac{1}{\tilde{C}_{n}^{U}} \sum_{j=1}^n \frac{Z_{j}-I_{j-1}^{U}}{Z_{j-1}} \tilde{\chi}_{j-1}^U \text{ and } M_{n}^L = \frac{1}{\tilde{C}_{n}^{L}} \sum_{j=1}^n \frac{Z_{j}}{Z_{j-1}} \tilde{\chi}_{j-1}^L.
\end{equation*}
 The decomposition
\begin{equation} \label{decomposition_of_offspring_mean}
    M_{n}^{T} = M^{T} + \frac{1}{\tilde{C}_{n}^T} (M_{n,1}^T + M_{n,2}^T)
\end{equation}
will be used in the proof of Theorem \ref{theorem:law_of_large_numnbers_and_central_limit_theorem_mean_estimation} and involves the martingale structure of $M_{n,i}^{T} \coloneqq \sum_{j=1}^{n}  D_{j,i}^{T}$, where
\begin{equation} \label{martingale_diffrerence_in_decomposition_of_offspring_mean}
    D_{j,1}^{T} \coloneqq (\overline{P}_{j-1}^{T}-M^{T}) \tilde{\chi}_{j-1}^T \text{ and } D_{j,2}^{T} \coloneqq \frac{\tilde{\chi}_{j-1}^{T}}{Z_{j-1}} \sum_{i=1}^{Z_{j-1}} ( \xi_{j-1,i}^{T} - \overline{P}_{j-1}^{T}).
\end{equation}
Specifically, let $\mathcal{G}_{n}$ be the $\sigma$-algebra generated by the random environments $\{ \Pi_{j}^{T} \}_{j=0}^{n}$, $\mathcal{H}_{n,1}$ the $\sigma$-algebra generated by $\mathcal{F}_{n}$ and $\mathcal{G}_{n-1}$; and $\mathcal{H}_{n,2}$ the $\sigma$-algebra generated by $\mathcal{F}_{n}$, $\mathcal{G}_{n-1}$, and the offspring distributions $\{ \xi_{j,i}^{T} \}_{i=0}^{\infty}$, $j=0,1,\dots,n-1$. Hence, $Z_{n}$, $\tilde{\chi}_{n}^{T}$, and $\Pi_{n-1}^{T}$ are $\mathcal{H}_{n,1}$-measurable, whereas $\Pi_{n}^{T}$ is not $\mathcal{H}_{n,1}$-measurable. We also denote by $\tilde{\mathcal{H}}_{n,1}$ the $\sigma$-algebra generated by $\mathcal{F}_{n-1}$ and $\mathcal{G}_{n-1}$ and the $\sigma$-algebra $\tilde{\mathcal{H}}_{n,2}$ generated by $\mathcal{F}_{n-1}$, $\mathcal{G}_{n-1}$, and $\{ \xi_{j,i}^{T} \}_{i=0}^{\infty}$, $j=0,1,\dots,n-1$. Hence, $Z_{n-1}$, $\tilde{\chi}_{n-1}^{T}$, and $\Pi_{n-1}^{T}$ are all $\tilde{\mathcal{H}}_{n,1}$-measurable but not $\tilde{\mathcal{H}}_{n-1,1}$-measurable. We establish in Proposition \ref{proposition:moments_m_n} in Appendix \ref{subsection:martingale_structure_of_MniT} that
\begin{equation*}
    \{ (M_{n,1}^{T}, \mathcal{H}_{n,i}) \}_{n=1}^{\infty} \text{ and } \{ (M_{n,2}^{T}, \mathcal{H}_{n,2}) \}_{n=1}^{\infty}
\end{equation*}
are mean zero martingale sequences. Additionally, $\bE[(M_{n,1}^{T})^{2}]=V_{1}^{T} \bE[\tilde{C}_{n}^{T}]$ and $\bE[(M_{n,2}^{T})^{2}]= V_{2}^{T} \bE[ \tilde{A}_{n}^{T}]$, where $\tilde{A}_{n}^{U} \coloneqq \sum_{j=1}^n \frac{\tilde{\chi}_{j-1}^U}{Z_{j-1}}$ is the sum of $\frac{1}{Z_{j}}$ over supercritical time steps up to time $n-1$ discarding times in which $Z_{j}$ is zero and $\tilde{A}_{n}^{L} \coloneqq A_{n}^{L} \coloneqq \sum_{j=1}^n \frac{\chi_{j-1}^L}{Z_{j-1}}$ is the sum of $\frac{1}{Z_{j}}$ over subcritical time steps up to time $n-1$. Proposition \ref{proposition:moments_m_n} contains other two martingales involving the terms $D_{j,1}^{T}$ and $D_{j,2}^{T}$ in \eqref{martingale_diffrerence_in_decomposition_of_offspring_mean} and related moment bounds, which will be used in the proof of Theorem \ref{theorem:law_of_large_numnbers_and_central_limit_theorem_mean_estimation}. As a first step, we derive the limit of the variances $\bE[(M_{n,1}^{T})^{2}]$ and $\bE[(M_{n,2}^{T})^{2}]$ when rescaled by $n$. By Proposition \ref{proposition:moments_m_n}, this entails studying the limit behavior of the quantities $\frac{1}{n} \tilde{C}_{n}^{T}$ and $\frac{1}{n} \tilde{A}_{n}^{T}$. To this end, we build i.i.d.\ blocks as in Section \ref{section:regenerative_property_of_crossing_times}. For $l \geq 1$ and $i \in S^{L}$ let $\overline{B}_{i,l}^{L} \coloneqq (K_{i,l}^{L}, \tilde{\bD}_{i,l}^{L}, \tilde{\bG}_{i,l}^{L})$, where
\begin{equation*}
  K_{i,l}^{L} = T_{i,l}^{L}-T_{i,l-1}^{L}, \tilde{\bD}_{i,l}^{L} \coloneqq ( \tilde{\Delta}_{T_{i,l-1}^{L}+1}^{U}, \dots, \tilde{\Delta}_{T_{i,l}^{L}}^{U}), \text{ and } \tilde{\bG}_{i,l}^{L} \coloneqq ( \tilde{\Gamma}_{T_{i,l-1}^{L}+1}^{U}, \dots, \tilde{\Gamma}_{T_{i,l}^{L}}^{U}),
\end{equation*}
$\tilde{\Delta}_{j+1}^{U} \coloneqq \sum_{k=\nu_{j}+1}^{\tau_{j+1}} \tilde{\chi}_{k-1}^{U}$, and $\tilde{\Gamma}_{j+1}^{U} \coloneqq \sum_{k=\nu_{j}+1}^{\tau_{j+1}} \frac{\tilde{\chi}_{k-1}^{U}}{Z_{k-1}}$. The triple $B_{i,l}^{L}$ consists of the random time $K_{i,l}^{L}$ required by $\{ Z_{\nu_{j}} \}_{j=0}^{\infty}$ to return for the $l^{\text{th}}$ time to state $i$, the lengths of all supercritical regimes $\tilde{\Delta}_{j}^{U}$ between the $(l-1)^{\text{th}}$ return and the $l^{\text{th}}$ return, and the sum of $Z_{j}$ inverse over supercritical regimes, disregarding the times when the process hits zero. Similarly, for $l \geq 1$ and $i \in S^{U}$ we let $\overline{B}_{i,l}^{U} \coloneqq (K_{i,l}^{U}, \bD_{i,l}^{U}, \bG_{i,l}^{U})$, where
\begin{equation*}
  K_{i,l}^{U} = T_{i,l}^{U}-T_{i,l-1}^{U}, \bD_{i,l}^{U} = ( \Delta_{T_{i,l-1}^{U}}^{L}, \dots, \Delta_{T_{i,l}^{U}-1}^{L}), \text{ and } \bG_{i,l}^{U} \coloneqq ( \Gamma_{T_{i,l-1}^{U}}^{L}, \dots, \Gamma_{T_{i,l}^{U}-1}^{L}),
\end{equation*}
and $\Gamma_{j}^{L} \coloneqq \sum_{k=\tau_{j}+1}^{\nu_{j}} \frac{\chi_{k-1}^{L}}{Z_{k-1}}$. Notice that, since $\tilde{C}_{n}^{L}=C_{n}^{L}$, Theorem \ref{theorem:law_of_large_numbers_and_central_limit_theorem_for_proportions} already yields that $\lim_{n \to \infty} \frac{C_{n}^{L}}{n} = \frac{\mu^{L}}{\mu^{U}+\mu^{L}}$. We need the following slight modification of Lemma \ref{lemma:regenerative_property_for_crossing_times}, whose proof is similar and hence omitted.
\begin{lem} \label{lemma:regenerative_property_for_averaged_tilde_crossing_times}
Assume \ref{H1}-\ref{H4}. (i) If \ref{H5} holds and $Z_{\nu_{0}}=i \in S^{L}$, then $\{ \overline{B}_{i,l}^{L} \}_{l=1}^{\infty}$ are i.i.d. (ii) If \ref{H6} (or \ref{H7}) holds and $Z_{\tau_{1}}=i \in S^{U}$, then $\{ \overline{B}_{i,l}^{U} \}_{l=1}^{\infty}$ are i.i.d.
\end{lem}

\begin{Prop} \label{proposition:limit_of_averages}
  Suppose that \ref{H1}-\ref{H6} (or \ref{H7}) hold and $\mu^{U},\mu^{L}<\infty$. Then \\
  (i) $\lim_{n \to \infty} \frac{\tilde{C}_{n}^{U}}{\tilde{N}^{L}(n)} = \tilde{\mu}^{U}$ and $\lim_{n \to \infty} \frac{\tilde{C}_{n}^{U}}{n} = \frac{\tilde{\mu}^{U}}{\mu^{U}+\mu^{L}}$ a.s. \\
  (ii) $\lim_{n \to \infty} \frac{\tilde{A}_{n}^{U}}{\tilde{N}^{L}(n)} = \tilde{A}^{U}$ and $\lim_{n \to \infty} \frac{\tilde{A}_{n}^{U}}{n} = \frac{\tilde{A}^{U}}{\mu^{U}+\mu^{L}}$ a.s. \\
  (iii) $\lim_{n \to \infty} \frac{A_{n}^{L}}{\tilde{N}^{U}(n)} = A^{L}$ and $\lim_{n \to \infty} \frac{A_{n}^{L}}{n} = \frac{A^{L}}{\mu^{U}+\mu^{L}}$ a.s.
\end{Prop}
Since $\frac{\tilde{C}_{n}^{T}}{n}$ and $\frac{\tilde{A}_{n}^{T}}{n}$ are non-negative and bounded by one, Proposition \ref{proposition:limit_of_averages} implies convergence in mean of these quantities.
\begin{proof}[Proof of Proposition \ref{proposition:limit_of_averages}]
By Lemma \ref{lemma:law_of_large_numbers_for_proportion_of_times} (i) it is enough to show the first part of the statements (i)-(iii). Since the proof of the other cases is similar we only prove (i). We recall that for $i  \in S^{L}$ and $j \in \mathbb{N}$, $N_{i}^{L}(j) = \sum_{l=1}^{\infty} \mathbf{I}_{ \{T_{i,l}^{L} \leq j\}}$ is the number of times $T_{i,l}^{L}$ is in $\{ 0,1,\dots,j \}$ and define $\overline{N}_{i}^{L}(n) \coloneqq N_{i}^{L}(\tilde{N}^{L}(n))$, $\tilde{D}_{i,l}^{L} \coloneqq \tilde{C}_{\nu_{T_{i,l}^{L}}}^{U}-\tilde{C}_{\nu_{T_{i,l-1}^{L}}}^{U}$, and $\tilde{D}_{i,n}^{*,L} \coloneqq \tilde{C}_{n}^{U}-\tilde{C}_{\nu_{T_{i,\overline{N}_{i}^{L}(n)}^{L}}}^{U}$. Conditionally on $Z_{\nu_{0}}=i$, $T_{i,0}^{L}=0$ and we write
\begin{equation*}
    \frac{\tilde{C}_{n}^{U}}{\tilde{N}^{L}(n)} = \frac{\overline{N}_{i}^{L}(n)}{\tilde{N}^{L}(n)} \biggl( \frac{1}{\overline{N}_{i}^{L}(n)} \sum_{l=1}^{\overline{N}_{i}^{L}(n)} \tilde{D}_{i,l}^{L} + \frac{1}{\overline{N}_{i}^{L}(n)} \tilde{D}_{i,n}^{*,L} \biggr).
\end{equation*}
Lemma\ref{lemma:regenerative_property_for_averaged_tilde_crossing_times} implies that $\{ \tilde{D}_{i,l}^{L} \}_{l=1}^{\infty}$ are i.i.d.\ with expectation that using Proposition 1.69 of \citet{Serfozo-2009} is given by $\bE_{\delta_{i}^{L}}[\tilde{C}_{\nu_{T_{i,1}^{L}}}^{U}] = (\pi_{i}^{L})^{-1} \tilde{\mu}^{U}$. Since $\lim_{n \to \infty} \overline{N}_{i}^{L}(n) = \infty$ a.s., we obtain that
\begin{equation*}
  \lim_{n \to \infty} \frac{1}{\overline{N}_{i}^{L}(n)} \sum_{l=1}^{\overline{N}_{i}^{L}(n)} \tilde{D}_{i,l}^{L} = (\pi_{i}^{L})^{-1} \tilde{\mu}^{U} \text{ a.s.\ and } \lim_{n \to \infty} \frac{1}{\overline{N}_{i}^{L}(n)} \tilde{D}_{i,n}^{*,L} =0 \text{ a.s.,}
\end{equation*}
since $\tilde{D}_{i,n}^{*,L} \leq \tilde{D}_{i,\overline{N}_{i}^{L}(n)+1}^{L}$. Finally, it holds that $\lim_{n \to \infty} \frac{\overline{N}_{i}^{L}(n)}{\tilde{N}^{L}(n)} = \pi_{i}^{L}$ a.s.
\end{proof}

We next establish that, when rescaled by their standard deviations, the terms $M_{n,i}^{T}$, where $i=1,2$ and $T \in \{L,U\}$, are jointly asymptotically normal. To this end, let 
\begin{equation*}
    \overline{\bM}_{n}^{T} \coloneqq \biggl( \frac{M_{n,1}^{T}}{\sqrt{\bE[(M_{n,1}^{T})^{2}]}}, \frac{M_{n,2}^{T}}{\sqrt{\bE[(M_{n,2}^{T})^{2}]}} \biggr)^{\top}\text{ and }\overline{\bM}_{n} \coloneqq ( (\overline{\bM}_{n}^{U})^{\top}, (\overline{\bM}_{n}^{L})^{\top} )^{\top}.
\end{equation*}

\begin{lem} \label{lemma:central_limit_theorem_mean_estimation}
Under the assumption of Theorem \ref{theorem:law_of_large_numnbers_and_central_limit_theorem_mean_estimation} (ii), $\overline{\bM}_{n} \xrightarrow[n \to \infty]{d} N(\bm{0},I)$.
\end{lem}
\begin{proof}[Proof of Lemma \ref{lemma:central_limit_theorem_mean_estimation}]
By Cram\'er–Wold theorem (see Theorem 29.4 of \citet{Billingsley-2013}), it is enough to show that for $t_{i}^{T} \in \mathbb{R}$, where $i=1,2$ and $T \in \{L,U\}$,
\begin{equation} \label{central_limit_theorem_mean_estimation}
    \sum_{T \in \{L,U\}} \sum_{i=1}^{2} t_{i}^{T} \frac{M_{n,i}^{T}}{\sqrt{\bE[(M_{n,i}^{T})^{2}]}} \xrightarrow[n \to \infty]{d} N \biggl(0,\sum_{T \in \{L,U\}}\sum_{i=1}^{2} (t_{i}^{T})^{2} \biggr).
\end{equation}
Using Proposition \ref{proposition:moments_m_n}, we see that
\begin{equation*}
    \biggl\{ \biggl( \sum_{T \in \{L,U\}} \sum_{i=1}^{2} t_{i}^{T} M_{n,i}^{T}, \mathcal{H}_{n,2} \biggr) \biggr\}_{n=1}^{\infty}
\end{equation*}
is a mean zero martingale sequence. In particular, 
\begin{equation*}
    \biggl\{ \biggl( \sum_{T \in \{L,U\}} \sum_{i=1}^{2} t_{i}^{T} \frac{M_{j,i}^{T}}{\sqrt{\bE[(M_{n,i}^{T})^{2}]}}, \mathcal{H}_{j,2} \biggr) \biggr\}_{j=1}^{n}
\end{equation*}
is a mean zero martingale array. We will apply Theorem 3.2 of \citet{Hall-1980} with $k_{n}=n$, $X_{nl}=\sum_{T \in \{L,U\}} \sum_{i=1}^{2} t_{i}^{T} \frac{D_{l,i}^{T}}{\sqrt{\bE[(M_{n,i}^{T})^{2}]}}$, $S_{nj}=\sum_{l=1}^{j} X_{nl}=\sum_{T \in \{L,U\}} \sum_{i=1}^{2} t_{i}^{T} \frac{M_{j,i}^{T}}{\sqrt{\bE[(M_{n,i}^{T})^{2}]}}$, $\mathcal{F}_{nj}=\mathcal{H}_{j,2}$, and $B^{2}=\sum_{T \in \{L,U\}}\sum_{i=1}^{2} (t_{i}^{T})^{2}$; and obtain \eqref{central_limit_theorem_mean_estimation}. To this end, we need to verify the following conditions: (i) $\bE[(S_{nj})^{2}] < \infty$, (ii) $\max_{l=1,\dots,n} \abs{X_{nl}} \xrightarrow[n \to \infty]{p} 0$, (iii) $\sum_{l=1}^{n} X_{nl}^{2} \xrightarrow[n \to \infty]{p} B^{2}$, and (iv) $\sup_{n \in \mathbb{N}} \bE[\max_{l=1,\dots,n} X_{nl}^{2}] < \infty$. Using Proposition \ref{proposition:moments_m_n} (iv) $\bE[(M_{j,i_{1}}^{T_{1}})(M_{j,i_{2}}^{T_{2}})]=0$ if either $T_{1} \neq T_{2}$ or $i_{1} \neq i_{2}$ and since $\bE[(M_{j,i}^{T})^{2}]$ are non-decreasing in $j$, we obtain that
\begin{equation*}
    \bE \biggl[ \biggl( \sum_{T \in \{L,U\}} \sum_{i=1}^{2} t_{i}^{T} \frac{M_{j,i}^{T}}{\sqrt{\bE[(M_{n,i}^{T})^{2}]}} \biggr)^{2} \biggr] = \sum_{T \in \{L,U\}} \sum_{i=1}^{2} (t_{i}^{T})^{2} \frac{\bE[(M_{j,i}^{T})^{2}]}{\bE[(M_{n,i}^{T})^{2}]} \leq \sum_{T \in \{L,U\}} \sum_{i=1}^{2} (t_{i}^{T})^{2} < \infty
\end{equation*}
yielding Condition (i). Using again that $\bE[(D_{l,i_{1}}^{T_{1}})(D_{l,i_{2}}^{T_{2}})]=0$ if either $T_{1} \neq T_{2}$ or $i_{1} \neq i_{2}$ and $\bE[(M_{n,i}^{T})^{2}]=\sum_{l=1}^{n} \bE[ (D_{l,i}^{T})^{2} ]$, we obtain that
\begin{align*}
    \bE \biggl[ \max_{l=1,\dots,n} \biggl( \sum_{T \in \{L,U\}} \sum_{i=1}^{2} t_{i}^{T} \frac{D_{l,i}^{T}}{\sqrt{\bE[(M_{n,i}^{T})^{2}]}} \biggr)^{2} \biggr] &\leq \bE \biggl[ \sum_{l=1}^{n} \biggl( \sum_{T \in \{L,U\}} \sum_{i=1}^{2} t_{i}^{T} \frac{D_{l,i}^{T}}{\sqrt{\bE[(M_{n,i}^{T})^{2}]}} \biggr)^{2} \biggr] \\
    &= \sum_{T \in \{L,U\}} \sum_{i=1}^{2} (t_{i}^{T})^{2} 
\end{align*}
yielding Condition (iv). Turning to Condition (ii), assuming w.lo.g.\ that $t_{i}^{T} \neq 0$ and using that 
\begin{equation*}
    \biggl( \sum_{T \in \{L,U\}} \sum_{i=1}^{2} t_{i}^{T} \frac{D_{l,i}^{T}}{\sqrt{\bE[(M_{n,i}^{T})^{2}]}} \biggr)^{2} \leq 4 \sum_{T \in \{L,U\}} \sum_{i=1}^{2} (t_{i}^{T})^{2} \frac{(D_{l,i}^{T})^{2}}{\sqrt{\bE[(M_{n,i}^{T})^{2}]}},
\end{equation*}
we obtain that for all $\epsilon>0$
\begin{align*}
  \bP\biggl( \max_{l=1,\dots,n} \biggl\lvert \sum_{T \in \{L,U\}} \sum_{i=1}^{2} t_{i}^{T} \frac{D_{l,i}^{T}}{\sqrt{\bE[(M_{n,i}^{T})^{2}]}} \biggr\rvert \geq \epsilon \biggr) &\leq \sum_{l=1}^{n} \bP\biggl( \biggl( \sum_{T \in \{L,U\}} \sum_{i=1}^{2} t_{i}^{T} \frac{D_{l,i}^{T}}{\sqrt{\bE[(M_{n,i}^{T})^{2}]}} \biggr)^{2} \geq \epsilon^{2} \biggr) \\
  &\leq \sum_{T \in \{L,U\}} \sum_{i=1}^{2} \sum_{l=1}^{n} \bP \biggl( (D_{l,i}^{T})^{2} \geq \biggl( \frac{\epsilon}{4 t_{i}^{T}} \biggr)^{2} \bE[(M_{n,i}^{T})^{2}] \biggr).
\end{align*}
For $i=1$, we use that $\tilde{\chi}_{l-1}^{T} \in \{0,1\}$ and obtain that
\begin{align*}
  &\sum_{l=1}^{n} \bP \biggl( (D_{l,1}^{T})^{2} \geq \biggl( \frac{\epsilon}{4t_{1}^{T}} \biggr)^{2} \bE[(M_{n,1}^{T})^{2}] \biggr) \\
 \leq& \sum_{l=1}^{n} \bP \biggl( (D_{l,1}^{T})^{2} \geq \biggl( \frac{\epsilon}{4t_{1}^{T}} \biggr)^{2} \bE[(M_{n,1}^{T})^{2}] | \tilde{\chi}_{l-1}^{T}=1 \biggr) \\
  =& \frac{n}{\bE[(M_{n,1}^{T})^{2}]} \bE[(M_{n,1}^{T})^{2}]\bP \biggl( (\overline{P}_{0}^{T}-M^{T})^{2} \geq \biggl( \frac{\epsilon}{4t_{1}^{T}} \biggr)^{2} \bE[(M_{n,1}^{T})^{2}] \biggr).
\end{align*}
It follows from Proposition \ref{proposition:moments_m_n} (i) and Proposition \ref{proposition:limit_of_averages} (i) that
\begin{equation} \label{limsup_M_1}
  \lim_{n \to \infty} \frac{n}{\bE[(M_{n,1}^{T})^{2}]} = \frac{1}{\tilde{\theta}^{T} V_{1}^{T}} < \infty,
\end{equation}
where for $T=L$ $\tilde{\theta}^{L} \coloneqq \theta^{L}$, and since $V_{1}^{T}<\infty$
\begin{equation*}
  \lim_{n \to \infty} \bE[(M_{n,1}^{T})^{2}] \bP \biggl( (\overline{P}_{0}^{T}-M^{T})^{2} \geq \biggl( \frac{\epsilon}{4t_{1}^{T}} \biggr)^{2} \bE[(M_{n,1}^{T})^{2}] \biggr)=0
\end{equation*}
yielding that
\begin{equation*}
    \lim_{n \to \infty} \sum_{l=1}^{n} \bP \biggl( (D_{l,1}^{T})^{2} \geq \biggl( \frac{\epsilon}{4t_{1}^{T}} \biggr)^{2} \bE[(M_{n,1}^{T})^{2}] \biggr) =0.
\end{equation*}
For $i=2$, we use that if $\tilde{\chi}_{l-1}^{T}=1$ then $Z_{l-1} \geq 1$ and obtain that 
\begin{align*}
  &\sum_{l=1}^{n} \bP \biggl( (D_{l,2}^{T})^{2} \geq \biggl( \frac{\epsilon}{4t_{2}^{T}} \biggr)^{2} \bE[(M_{n,2}^{T})^{2}] \biggr) \\
 \leq& n \sup_{z \in \mathbb{N}} \bP \biggl( \biggl( \frac{1}{z} \sum_{i=1}^{z} (\xi_{0,i}^{T}-\overline{P}_{0}^{T}) \biggr)^{2} \geq \biggl( \frac{\epsilon}{4t_{2}^{T}} \biggr)^{2} \bE[(M_{n,2}^{T})^{2}] \biggr).
\end{align*}
Next, using Proposition \ref{proposition:moments_m_n} (ii) and Proposition \ref{proposition:limit_of_averages} (ii)-(iii), we have that
\begin{equation} \label{limsup_M_2}
  \lim_{n \to \infty} \frac{n}{\bE[(M_{n,2}^{T})^{2}]} \leq \frac{\mu^{U}+\mu^{L}}{V_{2}^{T} \tilde{A}^{T}}< \infty,
\end{equation}
where $\tilde{A}^{L} \coloneqq A^{L}$. Since by Jensen's inequality
\begin{equation*}
  \bE \biggl[ \biggl| \frac{1}{z} \sum_{i=1}^{z} (\xi_{0,i}^{T}-\overline{P}_{0}^{T}) \biggr|^{2+\delta} \biggr] \leq \bE[\Lambda_{0,2}^{T,2+\delta}] < \infty,
\end{equation*}
using Markov inequality, we obtain that
\begin{equation*}
    \sum_{n=0}^{\infty} \sup_{z \in \mathbb{N}} \bP \biggl( \biggl( \frac{1}{z} \sum_{i=1}^{z} (\xi_{0,i}^{T}-\overline{P}_{0}^{T}) \biggr)^{2} \geq \biggl( \frac{\epsilon}{4t_{2}^{T}} \biggr)^{2} \bE[(M_{n,2}^{T})^{2}] \biggr) < \infty,
\end{equation*}
which yields that
\begin{equation*}
 \lim_{n \to \infty} n \sup_{z \in \mathbb{N}} \bP \biggl( \biggl( \frac{1}{z} \sum_{i=1}^{z} (\xi_{0,i}^{T}-\overline{P}_{0}^{T}) \biggr)^{2} \geq \biggl( \frac{\epsilon}{4t_{2}^{T}} \biggr)^{2} \bE[(M_{n,2}^{T})^{2}] \biggr)=0.
\end{equation*}
For (iii), we decompose
\begin{equation*}
    \sum_{l=1}^{n} \biggl( \sum_{T \in \{L,U\}} \sum_{i=1}^{2} t_{i}^{T} \frac{D_{l,i}^{T}}{\sqrt{\bE[(M_{n,i}^{T})^{2}]}} \biggr)^{2} - \sum_{T \in \{L,U\}} \sum_{i=1}^{2} (t_{i}^{T})^{2}
\end{equation*}
as
\begin{equation} \label{convergence_of_variance}
    \sum_{T \in \{L,U\}} \sum_{i=1}^{2} (t_{i}^{T})^{2} \frac{\sum_{l=1}^{n} (D_{l,i}^{T})^{2}-\bE[(M_{n,i}^{T})^{2}]}{\bE[(M_{n,i}^{T})^{2}]} + \sum_{T_{1} \neq T_{2} \text{ or } i_{1} \neq i_{2}} t_{i_{1}}^{T_{1}} t_{i_{2}}^{T_{2}} \frac{\sum_{l=1}^{n} (D_{l,i_{1}}^{T_{1}}) (D_{l,i_{2}}^{T_{2}})}{\sqrt{\bE[(M_{n,i_{1}}^{T_{1}})^{2}] \bE[(M_{n,i_{2}}^{T_{2}})^{2}]}}
\end{equation}
and show that each of the above terms converges to zero with probability one. Since $\bE[(M_{n,i}^{T})^{2}]=\sum_{l=1}^{n} \bE[(D_{l,i}^{T})^{2}]$, we use Proposition \ref{proposition:moments_m_n} (iii) and obtain that $\{(\sum_{l=1}^{n} ((D_{l,i}^{T})^{2}-\bE[(D_{l,i}^{T})^{2}]), \tilde{\mathcal{H}}_{n,i})\}_{n=1}^{\infty}$ are mean zero martingale sequences and for $s = 1+\delta/2$
\begin{equation*}
    \bE \biggl[ \biggr \lvert (D_{l,i}^{T})^{2}-\bE[(D_{l,i}^{T})^{2}] \biggr \rvert^{s} | \tilde{\mathcal{H}}_{l-1,i} \biggr] \leq 2^{s} \bE[\Lambda_{0,i}^{T,2s}] \bE[ \tilde{\chi}_{l-1}^{T}] <\infty.
\end{equation*}
We use \eqref{limsup_M_1} and \eqref{limsup_M_2}, and apply Theorem 2.18 of \citet{Hall-1980} with $S_{n}=\sum_{l=1}^{n} ((D_{l,i}^{T})^{2}-\bE[(D_{l,i}^{T})^{2}])$, $X_{l}=(D_{l,i}^{T})^{2}-\bE[(D_{l,i}^{T})^{2}]$, $\mathcal{F}_{n} = \tilde{\mathcal{H}}_{n,i}$, $U_{n}=\bE[(M_{n,i}^{T})^{2}]$, and $p=s$, where $i=1,2$, and obtain the convergence of the first term in \eqref{convergence_of_variance}. For the second term we proceed similarly. Specifically, using Proposition \ref{proposition:moments_m_n} (iv) and Cauchy-Schwartz inequality we obtain that $\{ (\sum_{l=1}^{n} (D_{l,i_{1}}^{T_{1}})(D_{l,i_{2}}^{T_{2}}), \tilde{\mathcal{H}}_{n,2}) \}_{n=0}^{\infty}$ is a mean zero martingale sequence and for $s = 1+\delta/2$
\begin{equation*}
    \bE[\abs{(D_{l,i_{1}}^{T_{1}})(D_{l,i_{2}}^{T_{2}})}^{s} | \tilde{\mathcal{H}}_{l-1,2}] \leq \bE[\Lambda_{0,i_{1}}^{T_{1},s} \Lambda_{0,i_{2}}^{T_{2},s}] \leq  \sqrt{\bE[\Lambda_{0,i_{1}}^{T_{1},2+\delta}] \bE[\Lambda_{0,i_{2}}^{T_{2},2+\delta}]} < \infty.
\end{equation*}
Finally, we apply Theorem 2.18 of \citet{Hall-1980} with $S_{n}=\sum_{l=1}^{n} (D_{l,i_{1}}^{T_{1}})(D_{l,i_{2}}^{T_{2}})$, $X_{l}=(D_{l,i_{1}}^{T_{1}})(D_{l,i_{2}}^{T_{2}})$, $\mathcal{F}_{n} = \tilde{\mathcal{H}}_{n,2}$, $U_{n}=\sqrt{\bE[(M_{n,i_{1}}^{T_{1}})^{2}] \bE[(M_{n,i_{2}}^{T_{2}})^{2}]}$, and $p=s$; and obtain convergence of the second term in \eqref{convergence_of_variance}.
\end{proof}

\noindent We are now ready to prove the main result of the section.

\begin{proof}[Proof of Theorem \ref{theorem:law_of_large_numnbers_and_central_limit_theorem_mean_estimation}]
Using Proposition \ref{proposition:moments_m_n} (i)-(ii) and $\tilde{\chi}_{j-1}^{T} \leq 1$, we obtain that for $i=1,2$ $\{ (M_{n,i}^{T}, \mathcal{H}_{n,i}) \}_{n=1}^{\infty}$ are martingales and
\begin{equation*}
    \sum_{j=1}^{\infty} \frac{1}{j^{s}} \bE[ \abs{D_{j,i}^{T}}^{s} | \mathcal{H}_{j-1,i} ] \leq \bE[\Lambda_{0,i}^{T,s}] \sum_{j=1}^{\infty} \frac{1}{j^{s}} < \infty.
\end{equation*}
 We apply Theorem 2.18 of \citet{Hall-1980} with $S_{n}=M_{n,i}^{T}$, $X_{j}=D_{j,i}^{T}$, where $i=1,2$ and $T \in \{ L, U \}$, $U_{n}=n$, $p=s$, and $\mathcal{F}_{n} = \mathcal{H}_{n,1}$ for $i=1$ and $\mathcal{F}_{n} = \mathcal{H}_{n,2}$ for $i=2$ and obtain that $\lim_{n \to \infty} \frac{1}{n} M_{n,i}^T = 0$ a.s. From this, Theorem \ref{theorem:law_of_large_numbers_and_central_limit_theorem_for_proportions}, and Proposition \ref{proposition:limit_of_averages} (i), we obtain that $\lim_{n \to \infty} \frac{1}{\tilde{C}_{n}^T} M_{n,i}^T = 0$ a.s. Using \eqref{decomposition_of_offspring_mean} we conclude that $\lim_{n \to \infty} M_{n}^{T}=M^{T}$ a.s. Turning to the central limit theorem,  Lemma \ref{lemma:central_limit_theorem_mean_estimation} (iii) yields that
\begin{equation*}
\overline{\bM}_{n} = \biggl( \frac{M_{n,1}^{U}}{\sqrt{\bE[(M_{n,1}^{U})^{2}]}}, \frac{M_{n,2}^{U}}{\sqrt{\bE[(M_{n,2}^{U})^{2}]}}, \frac{M_{n,1}^{L}}{\sqrt{\bE[(M_{n,1}^{L})^{2}]}}, \frac{M_{n,2}^{L}}{\sqrt{\bE[(M_{n,2}^{L})^{2}]}} \biggr)^{\top}
\end{equation*}
is asymptotically normal with mean zero and identity covariance matrix. Let $D_{n}^{2}$ be the $4 \times 4$ diagonal matrix 
\begin{equation*}
    D_{n}^{2} \coloneqq \diag \biggl( \frac{n\bE[(M_{n,1}^{U})^{2}]}{(\tilde{C}_{n}^U)^{2}}, \frac{n\bE[(M_{n,2}^{U})^{2}]}{(\tilde{C}_{n}^U)^{2}}, \frac{n\bE[(M_{n,1}^{L})^{2}]}{(\tilde{C}_{n}^L)^{2}}, \frac{n\bE[(M_{n,2}^{L})^{2}]}{(\tilde{C}_{n}^L)^{2}} \biggr).
\end{equation*}
By Proposition \ref{proposition:moments_m_n} (i)-(ii) and Proposition \ref{proposition:limit_of_averages}, $D_{n} \overline{\bM}_{n}$ is asymptotically normal with mean zero and covariance matrix
\begin{equation*}
    \tilde{D}^{2} \coloneqq \diag \biggl(\frac{V_{1}^{U}}{\tilde{\theta}^{U}}, \frac{\tilde{A}^{U} V_{2}^{U}}{\tilde{\theta}^{U} \tilde{\mu}^{U}}, \frac{V_{1}^{L}}{\theta^{L}}, \frac{A^{L}V_{2}^{L}}{\theta^{L} \mu^{L}} \biggr).
\end{equation*}
Using the continuous mapping theorem, it follows that
\begin{equation*}
    \sqrt{n} (\bM_{n}-\bM) \xrightarrow[n \to \infty]{d} N(\bm{0},\Sigma).
\end{equation*}
\end{proof}

\section{Discussion and concluding remarks} \label{section:discussion_and_concluding_remarks}

In this paper we developed BPRE with Thresholds to describe periods of growth and decrease in the population size arising in several applications including COVID dynamics. Even though the model is non-Markov, we identify Markov subsequences and use them to understand the length of time the process spends in the supercritical and subcritical regimes. Furthermore, using the regeneration technique we also study the rate of growth (or decline) of the process in the supercritical (subcritical) regime. It is possible to start the process using the subcritical BPRE and then move to the supercritical regime; this introduces only minor changes and the qualitative results remain the same. Finally, we note that without incorporating immigration in the supercritical regime the process will become extinct with probability one and hence the cyclical path behavior may not be observed.

An interesting question concerns the choice of strongly subcritical BPRE for the subcritical regime. It is folklore that the generation sizes of moderately and weakly subcritical processes can increase for long periods of time and in that case the time to cross the lower threshold will have a heavier tail. This could lead to lack of identifiability of supercritical and subcritical regimes. Similar issues arise when a subcritical BPRE is replaced by a critical BPRE or when immigration is allowed in both regimes. Since a subcritical BPRE with immigration converges in distribution to a proper limit law \citep{Roitershtein-2007}, we may fail to observe a clear period of decrease. The path properties of these alternatives could be useful for modeling other dynamics observed (see \citet{Klebaner-1993,Iannelli-2014}). Mathematical issues arising from these alternatives would involve different techniques than used in this paper. We end this section with a brief discussion concerning the moment conditions in Theorem \ref{theorem:law_of_large_numnbers_and_central_limit_theorem_mean_estimation}. It is possible to reduce the conditions $\bE[ \Lambda_{0,i}^{T,2+\delta}] < \infty$ to finite second moment hypothesis. This requires an extension of Lemma \ref{lemma:regenerative_property_for_crossing_times} to joint independence of blocks in $B_{i,l}^{L}$, $B_{i,l}^{U}$, offspring random variables, environments, immigration over cycles. The proof will need the Markov property of the pair $\{ (Z_{\nu_{j-1}},Z_{\tau_{j}}) \}_{j=1}^{\infty}$  and its uniform ergodicity. Also, the joint Markov property will also yield joint central limit theorem for the length and proportion of time spent in the supercritical and subcritical regimes. The proof is similar to that of Theorem \ref{theorem:uniform_ergodicity} and Lemma \ref{lemma:regenerative_property_for_crossing_times} but is more cumbersome with an increased notational burden. The numerical experiments suggest that the estimators of the mean parameters of the supercritical and subcritical regime are not affected by the choice of various distributions. A thorough statistical analysis of the robustness of the estimators and analysis of the datasets are beyond the scope of this paper and will be investigated elsewhere.

\medspace



\medspace

\appendix

\section{Auxiliary results} \label{section:auxiliary_results}

This section contains detailed descriptions and proofs of auxiliary results used in the paper. We begin with a detailed description of probability space for BPRET.

\subsection{Probability space} \label{subsection:probability_space}

In this subsection we describe in detail the random variables used to define BPRET as well as the underlying probability space. The thresholds $\{ (U_{j},L_{j})\}_{j=1}^{\infty}$ are i.i.d.\ random vectors with support $S_{B}^{U} \times S_{B}^{L}$, where $S_{B}^{U} = \mathbb{N} \cap [L_U+1,\infty)$, $S_{B}^{L} = \mathbb{N} \cap [L_{0}, L_{U}]$, and $1 \leq L_{0} \leq L_{U}$ are fixed integers, defined on the probability space $(\Omega_{B},\mathcal{F}_{B},\bP_{B})$. Next, $\Pi^{L} = \{ \Pi_{n}^{L} \}_{n=0}^{\infty}$ and $\Pi^{U} = \{ \Pi_{n}^{U} \}_{n=0}^{\infty}$ are subcritical and supercritical environmental sequences that are defined on probability spaces $(\Omega_{E^{L}}, \mathcal{F}_{E^{L}}, \bP_{E^{L}})$ and $(\Omega_{E^{U}}, \mathcal{F}_{E^{U}}, \bP_{E^{U}})$. Specifically, $\Pi_{n}^{U}=(P_{n}^{U},Q_{n}^{U})$ and $\Pi_{n}^{L}=P_{n}^{L}$, where $P_{n}^{U}=\{ P_{n,r}^{U} \}_{r=0}^{\infty}$, $P_{n}^{L}=\{ P_{n,r}^{L} \}_{r=0}^{\infty}$, and $Q_{n}^{U}=\{ Q_{n,r}^{U} \}_{r=0}^{\infty}$ are probability distributions in $\mathcal{P}$. Let $(\Omega_{U},\mathcal{F}_{U},\bP_{U})$ and $(\Omega_{L},\mathcal{F}_{L},\bP_{L})$ denote probability spaces corresponding to supercritical BPRE with immigration and subcritical BPRE. Hence, the environment sequence $\Pi^{U} = \{ \Pi_{n}^{U} \}_{n=0}^{\infty}$, the offspring sequence $\{ \xi_{n,i}^{U} \}_{i=1}^{\infty}$, and the immigration sequence $\{ I_{n}^{U} \}_{n=0}^{\infty}$ are random variables on $(\Omega_{U},\mathcal{F}_{U},\bP_{U})$. Similarly, $\Pi^{L} = \{ \Pi_{n}^{L} \}_{n=0}^{\infty}$ and $\{ \xi_{n,i}^{L} \}_{i=1}^{\infty}$, $n \geq 0$, are random variables on $(\Omega_{L},\mathcal{F}_{L},\bP_{L})$. We point out here that the probability spaces $(\Omega_{U},\mathcal{F}_{U},\bP_{U})$ and $(\Omega_{E^{U}},\mathcal{F}_{E^{U}},\bP_{E^{U}})$ are linked; that is, for all integrable functions $H: \Omega_{U} \to \mathbb{R}$
\begin{equation*}
     \int H(z,\Pi^{U}) d\bP_{U}(z,\Pi^{U}) = \int \int H(z,\Pi^{U}) d\bP_{U}(z | \Pi^{U}) d\bP_{E^{U}}(\Pi^{U}).
\end{equation*}
Similar comments also holds with $U$ replaced by $L$ in the above. All the above described random variables are defined on the probability space $(\Omega,\mathcal{F},\bP) = (\Omega_{B} \times \Omega_{U} \times \Omega_{L}, \mathcal{F}_{B} \otimes \mathcal{F}_{U} \otimes \mathcal{F}_{L}, \bP_{B} \times \bP_{U} \times \bP_{L})$.

\subsection{Time homogeneity of $\{Z_{\nu_{j}}\}_{j=0}^{\infty}$ and $\{ Z_{\tau_{j}} \}_{j=1}^{\infty}$} \label{subsection:time_homogeneity_of_Znu_and_Ztau}

\begin{lem} \label{lemma:stationarity}
  Assume \ref{H1} and \ref{H2prime}. For all $i \in S^{L}$, $k \in S^{U}$, and $j \in \mathbb{N}_{0}$ the following holds:
\begin{align*}
  \text{(i) } &\bP(Z_{\tau_{j+1}} = k | Z_{\nu_{j}}=i, \nu_{j}<\infty) = \bP_{\delta_{i}^{L}}(Z_{\tau_{1}} = k) \text{ and} \\
  &\bP(Z_{\tau_{j+1}} = k | \tau_{j+1}<\infty, Z_{\nu_{j}}=i) = \bP_{\delta_{i}^{L}}(Z_{\tau_{1}} = k | \tau_{1}< \infty), \text{ and} \\
  \text{(ii) } &\bP(Z_{\nu_{j+1}}=i | Z_{\tau_{j+1}}=k, \tau_{j+1} < \infty)=\bP_{\delta_{k}^{U}}(Z_{\nu_{1}}=i | \tau_{1} < \infty) \text{ and } \\
  &\bP(Z_{\nu_{j+1}}=i | \nu_{j+1} < \infty, Z_{\tau_{j+1}}=k)=\bP_{\delta_{k}^{U}}(Z_{\nu_{1}}=i | \nu_{1} < \infty).
\end{align*}
If additionally \ref{H3} holds, then (iii) $\tau_{j}$ and $\nu_{j}$ are finite almost surely,
\begin{align*}
  &\bP(Z_{\tau_{j+1}} = k | Z_{\nu_{j}}=i) = \bP_{\delta_{i}^{L}}(Z_{\tau_{1}} = k) \text{ and } \\
  &\bP(Z_{\nu_{j+1}}=i | Z_{\tau_{j+1}}=k)=\bP_{\delta_{k}^{U}}(Z_{\nu_{1}}=i).
\end{align*}
\end{lem}

\begin{proof}[Proof of Lemma \ref{lemma:stationarity}]
We only prove (i) and (iii). Since $\bP(Z_{\tau_{j+1}} = k | Z_{\nu_{j}}=i, \nu_{j}<\infty)$ is equal to
\begin{equation*}
    \sum_{s=1}^{\infty} \sum_{u=L_{U}+1}^{\infty} \bP(Z_{\nu_{j}+s}=k, Z_{\nu_{j}+s-1}<u, \dots, Z_{\nu_{j}+1}<u | Z_{\nu_{j}}=i, \nu_{j}<\infty ),
\end{equation*}
it is enough to show that for all $s \geq 1$ and $u \geq L_{U}+1$
\begin{equation} \label{stationarity_Znutau}
\begin{aligned}
    &\bP(Z_{\nu_{j}+s}=k, Z_{\nu_{j}+s-1}<u, \dots, Z_{\nu_{j}+1}<u | Z_{\nu_{j}}=i, \nu_{j}<\infty ) \\ =&\bP_{\delta_{i}^{L}}(Z_{s}=k, Z_{s-1}<u, \dots, Z_{1}<u).
\end{aligned}
\end{equation}
To this end, we condition on $\Pi_{\nu_{j}+l}^{U}=(p_{l}^{U},q_{l}^{U})$ and $\Pi_{l}^{U}=(p_{l}^{U},q_{l}^{U})$, where $l=0,1,\dots,s-1$. Since, given $\Pi_{\nu_{j}+l}^{U}=(p_{l}^{U},q_{l}^{U})$ and $\Pi_{l}^{U}=(p_{l}^{U},q_{l}^{U})$, both the sequences $\{ \xi_{\nu_{j}+l,i}^{U} \}_{i=1}^{\infty}$, $\{ \xi_{l,i}^{U} \}_{i=1}^{\infty}$ and the random variables $I_{\nu_{j}+l}^{U}$, $I_{l}^{U}$ are i.i.d., we obtain from \eqref{process_in_supercritical_regime} that
\begin{align*}
    &\bP(Z_{\nu_{j}+s}=k, Z_{\nu_{j}+s-1}<u, \dots, Z_{\nu_{j}+1}<u | Z_{\nu_{j}}=i, \nu_{j}<\infty, \Pi_{\nu_{j}+l}^{U}=(p_{l}^{U},q_{l}^{U}), l=0,1,\dots,s-1 ) \\
    =&\bP_{\delta_{i}^{L}}(Z_{s}=k, Z_{s-1}<u, \dots, Z_{1}<u | \Pi_{l}^{U}=(p_{l}^{U},q_{l}^{U}), l=0,1,\dots,s-1).
\end{align*}
By taking expectation w.r.t.\ $\Pi^{U}=\{ \Pi_{n}^{U} \}_{n=0}^{\infty}$ and using that $\Pi_{n}^{U}$ are i.i.d., we obtain \eqref{stationarity_Znutau}. Next, we notice that
\begin{equation*}
    \bP(Z_{\tau_{j+1}} = k | \tau_{j+1}<\infty, Z_{\nu_{j}}=i) = \frac{\bP(Z_{\tau_{j+1}} = k | Z_{\nu_{j}}=i, \nu_{j}<\infty)}{\bP(\tau_{j+1}<\infty | Z_{\nu_{j}}=i, \nu_{j}<\infty)}, \text{ where}
\end{equation*}
\begin{equation*}
    \bP(\tau_{j+1}<\infty | Z_{\nu_{j}}=i, \nu_{j}<\infty) = \sum_{k=L_{U}+1}^{\infty} \bP(Z_{\tau_{j+1}}=k | Z_{\nu_{j}}=i, \nu_{j}<\infty)
\end{equation*}
is positive because $M^{U} >1$. It follows from part (i) that $\bP(Z_{\tau_{j+1}} = k | \tau_{j+1}<\infty, Z_{\nu_{j}}=i)=\bP_{\delta_{i}^{L}}(Z_{\tau_{1}} = k | \tau_{1}<\infty)$. Finally, (iii) follows from (i) and (ii) using \eqref{tau_finite} and \eqref{nu_finite}.
\end{proof}

\subsection{Finiteness of $\overline{\pi}^{U}$} \label{subsection:finiteness_of_mean_stationary_distribution}

We show that the stationary distribution $\pi^{U}$ of the Markov chain $\{ Z_{\tau_{j}} \}_{j=1}^{\infty}$ has a finite first moment $\overline{\pi}^{U}$.

\begin{Prop} \label{proposition:finite_expectation_stationary_distribution}
  Under \ref{H1}-\ref{H4}, \ref{H6} (or \ref{H7}), and \ref{H9}, $\overline{\pi}^{U} < \infty$.
\end{Prop}

\begin{proof}[Proof of Proposition \ref{proposition:finite_expectation_stationary_distribution}]
Using that $\pi^{U}=\{ \pi_{k}^{U} \}_{k \in S^{U}}$ is the stationary distribution of the Markov chain $\{ Z_{\tau_{j}}\}_{j=1}^{\infty}$, we write for all $k \in S^{U}$ $\pi_{k}^{U}=\bP_{\pi^{U}}(Z_{\tau_{2}}=k) = \bE[ \bP_{\pi^{U}}(Z_{\tau_{2}}=k | U_{2})]$. Next, we notice that
\begin{equation*}
  \bP_{\pi^{U}}(Z_{\tau_{2}}=k | U_{2}) = \sum_{n=3}^{\infty} \bP_{\pi^{U}}(Z_{\tau_{2}}=k | \tau_{2}=n, U_{2}) \bP_{\pi^{U}}(\tau_{2}=n | U_{2}).
\end{equation*}
Now, using that the event $\{\tau_{2}=n\}$ is same as $\{Z_{n} \geq U_{2}\} \cap \cap_{k=\nu_{1}+1}^{n-1} \{Z_{k} < U_{2}\} \cap \{\nu_{1} \leq n-1\}$, the RHS of the above inequality is bounded above by
\begin{equation*}
\max_{i=0,1,\dots,U_{2}-1} \sum_{n=3}^{\infty} \bP(Z_{n}=k | Z_{n} \geq U_{2}, U_{2}, Z_{n-1}=i, \cap_{k=\nu_{1}+1}^{n-2} \{Z_{k}<U_{2}\}, \nu_{1} \leq n-1) \bP_{\pi^{U}}(\tau_{2}=n | U_{2}).
\end{equation*}
Since BPRE is a time-homogeneous Markov chain, it follows that
\begin{equation*}
    \bP(Z_{n}=k | Z_{n} \geq U_{2}, U_{2}, Z_{n-1}=i, \cap_{k=\nu_{1}+1}^{n-2} \{Z_{k}<U_{2}\}, \nu_{1} \leq n-1) = \bP_{\delta_{i}^{L}}(Z_{1}=k | Z_{1} \geq U_{2}, U_{2}),
\end{equation*}
where we also use the fact that the process starts in the supercritical regime. Now, using $U_{1}$ and $U_{2}$ are i.i.d., it follows that
\begin{equation*}
  \pi_{k}^{U} \leq \bE[\max_{i=0,1,\dots,U_{1}-1} \bP_{\delta_{i}^{L}}(Z_{1}=k | Z_{1} \geq U_{1}, U_{1})].
\end{equation*}
Since 
\begin{equation*}
  \bP_{\delta_{i}^{L}}(Z_{1}=k | Z_{1} \geq U_{1}, U_{1}) \leq \frac{\bP_{\delta_{i}^{L}}(Z_{1}=k)}{\bP_{\delta_{i}^{L}}(Z_{1} \geq U_{1} | U_{1})},
\end{equation*}
using Fubini-Tonelli theorem, we obtain that
\begin{equation*}
  \overline{\pi}^{U} \leq \bE \biggl[\max_{i=0,1,\dots,U_{1}-1} \frac{\sum_{k \in S^{U}} k \bP_{\delta_{i}^{L}}(Z_{1}=k)}{\bP_{\delta_{i}^{L}}(Z_{1} \geq U_{1} | U_{1})} \biggr] \leq \bE \biggl[ \max_{i=0,1,\dots,U_{1}-1} \frac{\bE_{\delta_{i}^{L}}[Z_{1}]}{\bP_{\delta_{i}^{L}}(Z_{1} \geq U_{1} | U_{1})} \biggr].
\end{equation*}
Now, for all $i=0,1,\dots,U_{1}-1$, we have that
\begin{equation*}
  \bP_{\delta_{i}^{L}}(Z_{1} \geq U_{1} | U_{1}) \geq \bP(I_{0}^{U} \geq U_{1} | U_{1}).
\end{equation*}
Finally, using the Assumptions \ref{H2}, \ref{H3}, and \ref{H9}, we conclude that
\begin{equation*}
  \overline{\pi}^{U} \leq \bE \biggl[ \frac{(U_{1}-1) M^{U} + N^{U}}{\bP(I_{0}^{U} \geq U_{1} | U_{1})} \biggr] \leq \bE \biggl[ \frac{U_{1}}{\bP(I_{0}^{U} \geq U_{1} | U_{1})} \biggr] \max(M^{U}, N^{U}) < \infty.
\end{equation*}
\end{proof}

\subsection{Proofs of Lemma \ref{lemma:regenerative_property_for_crossing_times} and Lemma \ref{lemma:variance_of_SnL_and_SnU}} \label{subsection:proofs_of_lemmas}

\begin{proof}[Proof of Lemma \ref{lemma:regenerative_property_for_crossing_times}]
  We begin by proving (i). It is sufficient to show that for $n \in \mathbb{N}_{0}$ and $k \in \mathbb{N}$
\begin{equation} \label{regenerative_property_for_crossing_times}
    \bP_{\delta_{i}^{L}}( B_{i,n+1}^{L}=(k, \bd^{L}, \bd^{L}+\bd^{U}) | B_{i,n}^{L},\dots, B_{i,1}^{L}) = \bP_{\delta_{i}^{L}}( B_{i,1}^{L}=(k, \bd^{L}, \bd^{L}+\bd^{U})),
\end{equation}
where $\bd^{L}=(d_{1}^{L},\dots,d_{k}^{L})$, $\bd^{U}=(d_{1}^{U},\dots,d_{k}^{U})$, and $d_{j}^{L},d_{j}^{U} \in \mathbb{N}$. For simplicity set $X_{j} \coloneqq Z_{\nu_{j}}$. We recall that $\overline{\Delta}_{j}^{U}=\Delta_{j}^{U}+\Delta_{j}^{L}$ and notice
\begin{align*}
    &\bP_{\delta_{i}^{L}}( B_{i,n+1}^{L}=(k, \bd^{L}, \bd^{L}+\bd^{U}) | B_{i,n}^{L},\dots, B_{i,1}^{L}) \\
    =& \bP_{\delta_{i}^{L}}( K_{i,n+1}^{L}=k, \cap_{j=1}^{k} \{\Delta_{T_{i,n}+j}^{U}=d_{j}^{U}, \Delta_{T_{i,n}+j}^{L}=d_{j}^{L}\} | X_{T_{i,n}^{L}}=i, B_{i,n}^{L},\dots, B_{i,1}^{L}) \\
    =& \bP_{\delta_{i}^{L}}( X_{T_{i,n}+k}=i, \cap_{j=1}^{k-1} \{X_{T_{i,n}+j} \neq i\}, \cap_{j=1}^{k} \{\Delta_{T_{i,n}+j}^{U}=d_{j}^{U}, \Delta_{T_{i,n}+j}^{L}=d_{j}^{L}\} | X_{T_{i,n}^{L}}=i, B_{i,n}^{L},\dots, B_{i,1}^{L}).
\end{align*}
We now compute the last term of the above equation. Specifically, by proceeding as in the proof of Lemma \ref{lemma:stationarity} (involving conditioning on the environments), we obtain that for $n, k \in \mathbb{N}$, $x_{j} \in S^{L}$, and $x_{0}=i$
\begin{align*}
    &\bP_{\delta_{i}^{L}}(\cap_{j=1}^{k} \{X_{T_{i,n}^{L}+j}=x_{j}, \Delta_{T_{i,n}^{L}+j}^{U}=d_{j}^{U}, \Delta_{T_{i,n}^{L}+j}^{L}=d_{j}^{L}\} | X_{T_{i,n}^{L}}=i, B_{i,n}^{L},\dots, B_{i,1}^{L}) \\
    =& \prod_{j=1}^{k} \bP(X_{T_{i,n}^{L}+j}=x_{j}, \Delta_{T_{i,n}^{L}+j}^{U}=d_{j}^{U}, \Delta_{T_{i,n}^{L}+j}^{L}=d_{j}^{L} | X_{T_{i,n}^{L}+j-1}=x_{j-1}) \\
    =&\prod_{j=1}^{k} \bP(X_{j}=x_{j}, \Delta_{j}^{U}=d_{j}^{U}, \Delta_{j}^{L}=d_{j}^{L} | X_{j-1}=x_{j-1}) \\
    =&\bP( \cap_{j=1}^{k} \{X_{j}=x_{j}, \Delta_{j}^{U}=d_{j}^{U}, \Delta_{j}^{L}=d_{j}^{L}\} | X_{0}=i).
\end{align*}
Now, by summing over $x_{k} \in \{ i \}$ and $x_{j} \in S^{L} \setminus \{ i \}$, we obtain that
\begin{align*}
    &\bP_{\delta_{i}^{L}}( X_{T_{i,n}+k}=i, \cap_{j=1}^{k-1} \{X_{T_{i,n}+j} \neq i\}, \cap_{j=1}^{k} \{\Delta_{T_{i,n}+j}^{U}=d_{j}^{U}, \Delta_{T_{i,n}+j}^{L}=d_{j}^{L}\} | X_{T_{i,n}^{L}}=i, B_{i,n}^{L},\dots, B_{i,1}^{L}) \\
    =&\bP_{\delta_{i}^{L}}(X_{k}=i, \cap_{j=1}^{k-1} \{X_{j} \neq i \}, \cap_{j=1}^{k} \{\Delta_{j}^{U}=d_{j}^{U}, \Delta_{j}^{L}=d_{j}^{L}\}).
\end{align*}
The last term in the above is
\begin{equation*}
    \bP_{\delta_{i}^{L}}( T_{i,1}^{L}=k, \cap_{j=1}^{k} \{ \Delta_{j}^{U}=d_{j}^{U}, \Delta_{j}^{L}=d_{j}^{L}\}) = \bP_{\delta_{i}^{L}}( B_{i,1}^{L}=(k, \bd^{L}, \bd^{L}+\bd^{U})).
\end{equation*}
We thus obtain \eqref{regenerative_property_for_crossing_times}. The proof of (ii) is similar.
\end{proof}

\begin{proof}[Proof of Lemma \ref{lemma:variance_of_SnL_and_SnU}]
The first part of (i) follows from Proposition 1.69 of \citet{Serfozo-2009} with $X_{j}=Z_{\nu_{j}}$, $\pi=\pi^{L}=\{ \pi_{i}^{L} \}_{i \in S^{L}}$, and $V_{j}=\Delta_{j+1}^{U}$. For the second part of (i) we use the above proposition with $V_{j}=\overline{\Delta}_{j+1}^{U}$ and obtain that
\begin{equation*}
    \bE_{\delta_{i}^{L}}[\overline{S}_{T_{i,1}^{L}}^{U}] = (\pi_{i}^{L})^{-1} \bE_{\pi^{L}}[\overline{\Delta}_{1}^{U}] = (\pi_{i}^{L})^{-1} (\bE_{\pi^{L}}[\Delta_{1}^{L}] + \mu^{U}).
\end{equation*}
Remark \eqref{remark:link_between_stationary_distributions} yields that
\begin{equation*}
    \bE_{\pi^{L}}[\Delta_{1}^{L}] = \sum_{k \in S^{L}} \sum_{l \in S^{U}} \bE_{\delta_{l}^{U}}[\Delta_{1}^{L}] P_{\delta_{k}^{L}}(Z_{\tau_{1}}=l] \pi_{k}^{L} = \mu^{L}.
\end{equation*}
We now prove the fist part of (ii). Since, conditionally on $Z_{\nu_{0}}=i$, $\Delta_{1}^{U}$ and $\Delta_{T_{i,1}^{L}+1}^{U}$ have the same distribution, using (i) we have that
\begin{equation*}
  \bV_{\delta_{i}^{L}}[S_{T_{i,1}^{L}}^{U}] =\bE_{\delta_{i}^{L}} \biggl[ \biggl( \sum_{j=1}^{T_{i,1}^{L}}(\Delta_{j+1}^{U}-\mu^{U} \biggr)^{2} \biggr] =\bE_{\delta_{i}^{L}} \biggl[\sum_{j=1}^{T_{i,1}^{L}}(\Delta_{j+1}^{U}-\mu^{U})^{2} \biggr] + 2 C_{i}^{U},
\end{equation*}  
\begin{equation*}
   \text{where }  C_{i}^{U} \coloneqq \bE_{\delta_{i}^{L}} \biggl[\sum_{j=1}^{T_{i,1}^{L}}(\Delta_{j+1}^{U}-\mu^{U}) \sum_{l=j+1}^{T_{i,1}^{L}}(\Delta_{l+1}^{U}-\mu^{U}) \biggr].
\end{equation*}
Next, we apply Proposition 1.69 of \citet{Serfozo-2009} with $X_{j}=Z_{\nu_{j}}$, $\pi=\pi^{L}=\{ \pi_{i}^{L} \}_{i \in S^{L}}$ and $V_{j}=(\Delta_{j+1}^{U}-\mu^{U})^{2}$ and obtain that
\begin{equation*}
  \bE_{\delta_{i}^{L}} \biggl[\sum_{j=1}^{T_{i,1}^{L}}(\Delta_{j+1}^{U}-\mu^{U})^{2} \biggr] = (\pi_{i}^{L})^{-1} \bE_{\pi^{L}}[ (\Delta_{1}^{U}-\mu^{U})^{2}].
\end{equation*}
Then, we compute
\begin{align*}
 C_{i}^{U}=& \bE_{\delta_{i}^{L}}\biggl[\sum_{j=1}^{\infty} \mathbf{I}_{\{T_{i,1}^{L} \geq j\}} \bE_{\delta_{i}^{L}} \biggl[(\Delta_{j+1}^{U}-\mu^{U}) \sum_{l=j+1}^{T_{i,1}^{L}}(\Delta_{l+1}^{U}-\mu^{U}) | T_{i,1}^{L} \geq j \biggr] \biggr] \\
  =& \bE_{\delta_{i}^{L}} \biggl[\sum_{j=1}^{\infty} \mathbf{I}_{\{T_{i,1}^{L} \geq j\}} \sum_{k \in S^{L}} \mathbf{I}_{\{Z_{\nu_{j}}=k\}} g^{L}(k) \biggr] \\
  =& \sum_{k \in S^{L}} g^{L}(k) \bE_{\delta_{i}^{L}} \biggl[\sum_{j=1}^{T_{i,1}^{L}} \mathbf{I}_{\{Z_{\nu_{j}}=k\}} \biggr],
\end{align*}
where $g^{L} : S^{L} \to \mathbb{R}$ is given by
\begin{align*}
  g^{L}(k) &= \bE_{\delta_{i}^{L}} \biggl[(\Delta_{j+1}^{U}-\mu^{U}) \sum_{l=j+1}^{T_{i,1}^{L}}(\Delta_{l+1}^{U}-\mu^{U}) | T_{i,1}^{L} \geq j, Z_{\nu_{j}}=k \biggr] \\
  &= \sum_{l=j+1}^{\infty} \bE_{\delta_{i}^{L}} \biggl[(\Delta_{j+1}^{U}-\mu^{U}) (\Delta_{l+1}^{U}-\mu^{U}) \mathbf{I}_{\{T_{i,1}^{L} \geq l\}} | T_{i,1}^{L} \geq j, Z_{\nu_{j}}=k \biggr].
\end{align*}
Using Theorem 1.54 of \citet{Serfozo-2009}, we obtain $\bE_{\delta_{i}^{L}} [\sum_{j=1}^{T_{i,1}^{L}} \mathbf{I}_{\{Z_{\nu_{j}}=k\}}] = (\pi_{i}^{L})^{-1} \pi_{k}^{L}$, which yields
\begin{equation*}
  C_{i}^{U}= (\pi_{i}^{L})^{-1} \sum_{k \in S^{L}} g^{L}(k) \pi_{k}^{L}.
\end{equation*}
Now, using Lemma \ref{lemma:regenerative_property_for_crossing_times}, we see that, conditionally on $j \leq T_{i,1}^{L} < l$, $(\Delta_{l+1}^{U}-\mu^{U})$ is independent of $(\Delta_{j+1}^{U}-\mu^{U})$. If $Z_{\nu_{j}} \sim \pi^{L}$, then using stationarity (see Remark \ref{remark:stationary_distribution})
\begin{equation*}
    \bE[(\Delta_{l+1}^{U}-\mu^{U}) | j \leq T_{i,1}^{L}<l,Z_{\nu_{j}} \sim \pi^{L}]=\bE[(\Delta_{l+1}^{U}-\mu^{U}) | Z_{\nu_{l}} \sim \pi^{L}]=0.
\end{equation*}
Therefore,
\begin{align*}
    &\sum_{l=j+1}^{\infty}\sum_{k \in S^{L}} \pi_{k}^{L} \bE_{\delta_{i}^{L}} \biggl[(\Delta_{j+1}^{U}-\mu^{U}) (\Delta_{l+1}^{U}-\mu^{U}) \mathbf{I}_{\{T_{i,1}^{L}<l\}} | T_{i,1}^{L} \geq j, Z_{\nu_{j}}=k \biggr] \\
    =&\sum_{l=j+1}^{\infty} \bE_{\delta_{i}^{L}}\biggl[(\Delta_{j+1}^{U}-\mu^{U}) \bE[(\Delta_{l+1}^{U}-\mu^{U}) | j \leq T_{i,1}^{L}<l,Z_{\nu_{j}} \sim \pi^{L}] \mathbf{I}_{\{T_{i,1}^{L}<l\}} | T_{i,1}^{L} \geq j, Z_{\nu_{j}} \sim \pi^{L} \biggr]=0.
\end{align*}
Adding the above to $\sum_{k \in S^{L}} g^{L}(k) \pi_{k}^{L}$ we conclude that
\begin{align*}
    \sum_{k \in S^{L}} g^{L}(k) \pi_{k}^{L} &= \sum_{l=j+1}^{\infty} \sum_{k \in S^{L}} \pi_{k}^{L} \bE_{\delta_{i}^{L}} \biggl[(\Delta_{j+1}^{U}-\mu^{U}) (\Delta_{l+1}^{U}-\mu^{U}) | T_{i,1}^{L} \geq j, Z_{\nu_{j}}=k\biggr] \\
    &= \sum_{l=1}^{\infty} \sum_{k \in S^{L}} \pi_{k}^{L} \bE_{\delta_{k}^{L}} \biggl[(\Delta_{1}^{U}-\mu^{U}) (\Delta_{l+1}^{U}-\mu^{U}) \biggr] \\
    &=\sum_{l=1}^{\infty} \bC_{\pi^{L}}[\Delta_{1}^{U}, \Delta_{l+1}^{U}].
\end{align*}
The second part of (ii) and (iii) are obtained similarly.
\end{proof}

\subsection{Finiteness of $\mu^{T}$, $\sigma^{2,T}$ and $\overline{\sigma}^{2,T}$} \label{subsection:finiteness_of_means_and_variances}

We establish positivity and finiteness of $\sigma^{2,T}$ and $\overline{\sigma}^{2,T}$ where $T \in \{L,U\}$. Lemma \ref{lemma:bound_on_covariance} below is used to control the covariance terms in $\sigma^{2,T}$ and $\overline{\sigma}^{2,T}$. We recall that uniform ergodicity of the Markov chains $\{ Z_{\nu_{j}} \}_{j=0}^{\infty}$ and $\{ Z_{\tau_{j}} \}_{j=1}^{\infty}$ is equivalent to the existence of constants $C_{T} \geq 0$ and $\rho_{T} \in (0,1)$ such that $\sup_{l \in S_{T}} \norm{p_{l}^{T}(j)- \pi^{T}} \leq C_{T} \rho_{T}^{j}$.
\begin{lem} \label{lemma:bound_on_covariance}
Assume \ref{H1}-\ref{H4}. The following holds: \\
  (i) If \ref{H5} holds and $w_{i} \in \mathbb{R}$,  $i \in S^{L}$, then
\begin{equation*}
  \sum_{j=1}^{\infty} \biggl\lvert \sum_{i,k \in S^{L}} w_{k} w_{i} \pi_{k}^{L} p_{ki}^{L}(j) - \biggl( \sum_{k \in S^{L}} w_{k} \pi_{k}^{L} \biggr) \biggl( \sum_{i \in S^{L}} w_{i} \pi_{i}^{L} \biggr) \biggr\rvert \leq 2 C_{L}^{1/2} \frac{\rho_{L}^{1/2}}{1-\rho_{L}^{1/2}} \biggl( \sum_{k \in S^{L}} w_{k}^{2} \pi_{k}^{L} \biggr).
\end{equation*}
  (ii) If \ref{H6} (or \ref{H7}) holds and $w_{k} \in \mathbb{R}$,  $k \in S^{U}$, then 
\begin{equation*}
  \sum_{j=1}^{\infty} \biggl\lvert \sum_{i,k \in S^{U}} w_{k} w_{i} \pi_{k}^{U} p_{ki}^{U}(j) - \biggl( \sum_{k \in S^{U}} w_{k} \pi_{k}^{U} \biggr) \biggl( \sum_{i \in S^{U}} w_{i} \pi_{i}^{U} \biggr) \biggr\rvert \leq 2 C_{U}^{1/2} \frac{\rho_{U}^{1/2}}{1-\rho_{U}^{1/2}} \biggl( \sum_{k \in S^{U}} w_{k}^{2} \pi_{k}^{U} \biggr).
\end{equation*}
\end{lem}
The proof of the above lemma can be constructed along the lines of Theorem 17.2.3 of \citet{Ibragimov-1971} with $p=q=1/2$ and it involves a repeated use of Cauchy-Schwarz inequality and stationarity in Remark \ref{remark:stationary_distribution}. 

\begin{proof}[Proof of Lemma \ref{lemma:bound_on_covariance}]
Since the proof of (ii) is similar we only prove (i). We proceed along the lines of the proof of Theorem 17.2.3 of \citet{Ibragimov-1971}. Using Cauchy-Schwarz inequality, we have that
\begin{align*}
  &\biggl\lvert \sum_{i,k \in S^{L}} w_{k} w_{i} \pi_{k}^{L} p_{ki}^{L}(j) - \biggl( \sum_{k \in S^{L}} w_{k} \pi_{k}^{L} \biggr) \biggl( \sum_{i \in S^{L}} w_{i} \pi_{i}^{L} \biggr) \biggr\rvert \\
  =&\biggl\lvert \sum_{k \in S^{L}} w_{k} (\pi_{k}^{L})^{1/2} \sum_{i \in S^{L}} w_{i} (p_{ki}^{L}(j)-\pi_{i}^{L}) (\pi_{k}^{L})^{1/2} \biggr\rvert \\
  \leq& \biggl( \sum_{k \in S^{L}} (w_{k})^{2} \pi_{k}^{L} \biggr)^{1/2} \biggl( \sum_{k \in S^{L}} \pi_{k}^{L} \biggl( \sum_{i \in S^{L}} w_{i} (p_{ki}^{L}(j)-\pi_{i}^{L}) \biggr)^{2} \biggr)^{1/2}.
\end{align*}
Using Cauchy-Schwarz inequality again, we obtain that
\begin{align*}
  \biggl( \sum_{i \in S^{L}} w_{i} (p_{ki}^{L}(j)-\pi_{i}^{L}) \biggr)^{2} &\leq \biggl( \sum_{i \in S^{L}} \abs{w_{i}} (p_{ki}^{L}(j)+\pi_{i}^{L})^{1/2} \abs{p_{ki}^{L}(j)-\pi_{i}^{L}}^{1/2} \biggr)^{2} \\
  &\leq \biggl( \sum_{i \in S^{L}} (w_{i})^{2} (p_{ki}^{L}(j)+\pi_{i}^{L}) \biggr) \biggl( \sum_{i \in S^{L}} \abs{p_{ki}^{L}(j)-\pi_{i}^{L}} \biggr).
\end{align*}
Since $\sum_{k \in S^{L}} p_{ki}^{L}(j) \pi_{k}^{L}=\pi_{i}^{L}$ by Remark \ref{remark:stationary_distribution}, we deduce that
\begin{align*}
  \biggl( \sum_{k \in S^{L}} \pi_{k}^{L} \biggl( \sum_{i \in S^{L}} w_{i} (p_{ki}^{L}(j)-\pi_{i}^{L}) \biggr)^{2} \biggr)^{1/2} &\leq \biggl( \sum_{k \in S^{L}} \pi_{k}^{L} \biggl( \sum_{i \in S^{L}} (w_{i})^{2} (p_{ki}^{L}(j)+\pi_{i}^{L}) \biggr) \biggl( \sum_{i \in S^{L}} \abs{p_{ki}^{L}(j)-\pi_{i}^{L}} \biggr) \biggr)^{1/2} \\
  &\leq \biggl( 2 \sum_{i \in S^{L}} (w_{i})^{2} \pi_{i}^{L} \biggr)^{1/2} \sup_{k \in S^{L}} \biggl( \sum_{i \in S^{L}} \abs{p_{ki}^{L}(j)-\pi_{i}^{L}} \biggr)^{1/2}.
\end{align*}
Using that $\sup_{l \in S^{L}} \norm{p_{l}^{L}(j)- \pi^{L}} \leq C_{L} \rho_{L}^{j}$, we obtain that
\begin{align*}
  \sup_{k \in S^{L}} \biggl( \sum_{i \in S^{L}} \abs{p_{ki}^{L}(j)-\pi_{i}^{L}} \biggr) &\leq \sup_{k \in S^{L}} \biggl( \sum_{i \in S^{L} : p_{ki}^{L}(j)-\pi_{i}^{L} > 0} (p_{ki}^{L}(j)-\pi_{i}^{L}) \biggr) + \sup_{k \in S^{L}} \biggl( \sum_{i \in S^{L} : p_{ki}^{L}(j)-\pi_{i}^{L} < 0} (\pi_{i}^{L}-p_{ki}^{L}(j)) \biggr) \\
  &\leq \sup_{k \in S^{L}} \biggl( \sum_{i \in S^{L} : p_{ki}^{L}(j)-\pi_{i}^{L} > 0} p_{ki}^{L}(j) - \sum_{i \in S^{L} : p_{ki}^{L}(j)-\pi_{i}^{L} > 0} \pi_{i}^{L}) \biggr) \\
  &+ \sup_{k \in S^{L}} \biggl( \sum_{i \in S^{L} : p_{ki}^{L}(j)-\pi_{i}^{L} < 0} \pi_{i}^{L}- \sum_{i \in S^{L} : p_{ki}^{L}(j)-\pi_{i}^{L} < 0} p_{ki}^{L}(j) \biggr) \\
  &\leq 2C_{L}\rho_{L}^{j}.
\end{align*}
We have thus shown that
\begin{equation*}
  \biggl\lvert \sum_{i,k \in S^{L}} w_{k} w_{i} \pi_{k}^{L} p_{ki}^{L}(j) - \biggl( \sum_{k \in S^{L}} w_{k} \pi_{k}^{L} \biggr) \biggl( \sum_{i \in S^{L}} w_{i} \pi_{i}^{L} \biggr) \biggr\rvert \leq 2 C_{L}^{1/2} \rho_{L}^{j/2} \biggl( \sum_{k \in S^{L}} w_{k}^{2} \pi_{k}^{L} \biggr),
\end{equation*}
which yields that
\begin{equation*}
  \sum_{j=1}^{\infty} \biggl\lvert \sum_{i,k \in S^{L}} w_{k} w_{i} \pi_{k}^{L} p_{ki}^{L}(j) - \biggl( \sum_{k \in S^{L}} w_{k} \pi_{k}^{L} \biggr) \biggl( \sum_{i \in S^{L}} w_{i} \pi_{i}^{L} \biggr) \biggr\rvert \leq 2 C_{L}^{1/2} \frac{\rho_{L}^{1/2}}{1-\rho_{L}^{1/2}} \biggl( \sum_{k \in S^{L}} w_{k}^{2} \pi_{k}^{L} \biggr).
\end{equation*}
\end{proof}

We are now ready to study the finiteness of means and variances $\mu^{T}$, $\sigma^{2,T}$, and $\overline{\sigma}^{2,T}$, where $T \in \{L,U\}$.
\begin{Prop} \label{proposition:mean_and_variance_are_finite_and_variance_is_positive}
Assume \ref{H1}-\ref{H4}. (i) Also, if \ref{H5} and \ref{H8} hold, then $\mu^{U}< \infty$ and $0<\sigma^{2,U}<\infty$. Next, (ii) if \ref{H6} (or \ref{H7}) and \ref{H9} hold, then $\mu^{L}<\infty$ and $0<\sigma^{2,L}<\infty$. (iii) Additionally, under the assumptions in (i) and (ii) $0<\overline{\sigma}^{2,U},\overline{\sigma}^{2,L} <\infty$.
\end{Prop}
It is easy to see that Proposition \ref{proposition:mean_and_variance_are_finite_and_variance_is_positive} implies that $\abs{\mathbb{C}^{U}}$ and $\abs{\mathbb{C}^{L}}$ are finite. 

\begin{proof}[Proof of Proposition \ref{proposition:mean_and_variance_are_finite_and_variance_is_positive}]
We begin by proving (i). For all $i \in S^{L}$ it holds that
\begin{equation*}
  \bE_{\delta_{i}^{L}}[\tau_{1}] = \sum_{n=1}^{\infty} \bP_{\delta_{i}^{L}}(\tau_{1} \geq n) \leq \sum_{n=0}^{\infty} \bP_{\delta_{i}^{L}}(\tilde{Z}_{n} < U_{1}),
\end{equation*}
where $\{ \tilde{Z}_{n} \}_{n=0}^{\infty}$ is a supercritical BPRE with immigration having environmental sequence $\Pi^{U} = \{ \Pi_{n}^{U} \}_{n=0}^{\infty}$ and, conditionally on $\Pi_{n}^{U}$, offspring distributions $\{ \xi_{n,i}^{U} \}_{i=0}^{\infty}$ and immigration distribution $I_{n}^{U}$. Using $\lim_{n \to \infty} \tilde{Z}_{n}=\infty$ a.s.\ and $\bE[U_{1}]< \infty$, we see that $\lim_{n \to \infty} \tilde{Z}_{n} \bP_{\delta_{i}^{L}}(\tilde{Z}_{n} < U_{1} | \tilde{Z}_{n})=0$ a.s. Since $\lim_{n \to \infty} \frac{\tilde{Z}_{n}}{(M^{L})^{n}}>0$, we obtain $\lim_{n \to \infty} (M^{L})^{n} \bP_{\delta_{i}^{L}}(\tilde{Z}_{n} < U_{1} | \tilde{Z}_{n})=0$ a.s., which yields $\lim_{n \to \infty} (M^{L})^{n} \bP_{\delta_{i}^{L}}(\tilde{Z}_{n} < U_{1})=0$. Therefore, there exists $\tilde{C}$ such that
\begin{equation} \label{upper_bound_on_1_over_Z_n}
    \bP_{\delta_{i}^{L}}(\tilde{Z}_{n} < U_{1}) \leq  \tilde{C} (M^{L})^{n}.
\end{equation}
Now, using $S^{L}$ is finite and $\bE[U_{1}]<\infty$ it follows that
\begin{equation*}
  \mu^{U} = \sum_{i \in S^{L}} \bE_{\delta_{i}^{L}}[\tau_{1}] \pi_{i}^{L} \leq (\max_{i \in S^{L}} C_{i}) \frac{\bE[U_{1}]}{1-\gamma} < \infty.
\end{equation*}
Turning to the finiteness of $\sigma^{2,U}$, replacing $n$ by $\lfloor \sqrt{n} \rfloor$ in \eqref{upper_bound_on_1_over_Z_n}, one obtains that
\begin{equation*}
  \bE_{\delta_{k}^{L}}[\tau_{1}^{2}] = \sum_{n=1}^{\infty} \bP_{\delta_{k}^{L}}(\tau_{1} \geq \sqrt{n}) \leq \sum_{n=0}^{\infty} \bP_{\delta_{k}^{L}}(\tilde{Z}_{\lfloor \sqrt{n} \rfloor} < U_{1}) < \infty.
\end{equation*}
This together with the finiteness of $\mu^{U}$ yields $\bV_{\pi^{L}}[ \Delta_{1}^{U}] < \infty$. Turning to the covariance terms in $\sigma^{2,U}$, we apply Lemma \ref{lemma:bound_on_covariance} (i) with $w_{i} = \bE_{\delta_{i}^{L}}[\tau_{1} - \mu^{U}]$ and using $\sum_{i \in S^{L}} w_{i} \pi_{i}^{L}=0$, we obtain that
\begin{align*}
  \sum_{j=1}^{\infty} \abs{\bC_{\pi^{L}}[ \Delta_{1}^{U}, \Delta_{j+1}^{U} ]} =& \sum_{j=1}^{\infty} \abs{\sum_{k \in S^{L}} \bE_{\delta_{k}^{L}}[\tau_{1}-\mu^{U}] \pi_{k}^{L} \sum_{i \in S^{L}} \bE_{\delta_{i}^{L}}[ \tau_{1}-\mu^{U}] p_{ki}^{L}(j)} \\
  \leq& 2 C_{L}^{1/2} \biggl( \frac{\rho_{L}^{1/2}}{1-\rho_{L}^{1/2}} \biggr) \biggl( \sum_{k \in S^{L}} ( \bE_{\delta_{k}^{L}}[\tau_{1}-\mu^{U}] )^{2} \pi_{k}^{L} \biggr) < \infty.
\end{align*}
We conclude that $\sigma^{2,U}$ is finite. For $i \in S^{U}$, it holds that
\begin{equation*}
  \bE_{\delta_{i}^{U}}[\Delta_{1}^{L}] = \sum_{n=0}^{\infty} \bP_{\delta_{i}^{U}}(\Delta_{1}^{L}>n) \leq i \sum_{n=0}^{\infty} (M^{L})^{n} = \frac{i}{1-M^{L}},
\end{equation*}
where the inequality follows from the upper bound on the extinction time of the process in the subcritical regime. Hence, using Proposition \ref{proposition:finite_expectation_stationary_distribution} it follows that $\mu^{L} \leq \frac{\overline{\pi}^{U}}{1-M^{L}} < \infty$.

Next, we show that $\sigma^{2,L}$ is finite. As before, we obtain that for all $i \in S^{U}$
\begin{equation} \label{expectation_nu1_square}
  \bE_{\delta_{i}^{U}}[ (\Delta_{1}^{L})^{2}] \leq \sum_{n=0}^{\infty} \bP_{\delta_{i}^{U}}(\Delta_{1}^{L} > \lfloor \sqrt{n} \rfloor] \leq i \sum_{n=0}^{\infty} (M^{L})^{\lfloor \sqrt{n} \rfloor}
\end{equation}
yielding that
\begin{equation} \label{expectation_nu1_square_is_finite}
  \bE_{\pi^{U}}[ (\Delta_{1}^{L})^{2} ] \leq \overline{\pi}^{U} \sum_{n=0}^{\infty} (M^{L})^{\lfloor \sqrt{n} \rfloor} < \infty.
\end{equation}
This and the finiteness of $\mu^{L}<\infty$ implies that $\bV_{\pi^{L}}[ \Delta_{1}^{L}]<\infty$. Turning to covariances we apply Lemma \ref{lemma:bound_on_covariance} (ii) with $w_{i} = \bE_{\delta_{i}^{U}}[\Delta_{1}^{L} - \mu^{L}]$ and using $\sum_{i \in S^{U}} w_{i} \pi_{i}^{U}=0$, we obtain that
\begin{equation*}
  \sum_{j=1}^{\infty} \abs{\bC_{\pi^{U}}[ \Delta_{1}^{L}, \Delta_{j+1}^{L}]} \leq 2 C_{U}^{1/2} \biggl( \frac{\rho_{U}^{1/2}}{1-\rho_{U}^{1/2}} \biggr) \biggl( \sum_{k \in S^{U}} (\bE_{\delta_{k}^{U}}[\Delta_{1}^{L}-\mu^{L}] )^{2} \pi_{k}^{U} \biggr) < \infty
\end{equation*}
yielding the finiteness of $\sigma^{2,L}$. Turning to (iii), we compute
\begin{equation*}
    \bV_{\pi^{L}}[ \overline{\Delta}_{1}^{U} ] \leq 2 ( \bV_{\pi^{L}}[ \Delta_{1}^{U} ] + \bV_{\pi^{L}}[ \Delta_{1}^{L}] ),
\end{equation*}
where using Part (i) $\bV_{\pi^{L}}[ \Delta_{1}^{U} ] < \infty$ and using Remark \ref{remark:link_between_stationary_distributions}
\begin{equation*}
    \bV_{\pi^{L}}[ \Delta_{1}^{L}] = \sum_{k \in S^{U}} \sum_{i \in S^{L}} \bV_{\delta_{k}^{U}}[ \Delta_{1}^{L}] \bP_{\delta_{i}^{L}}(Z_{\tau_{1}}=k) \pi_{i}^{L} = \bV_{\pi^{U}}[ \Delta_{1}^{L} ] < \infty.
\end{equation*}
Turning to the covariance, we again apply Lemma \ref{lemma:bound_on_covariance} (i) with $w_{i} = \bE_{\delta_{i}^{L}}[\overline{\Delta}_{1}^{U} - (\mu^{U}+\mu^{L})]$ and conclude that also
\begin{equation*}
  \sum_{j=1}^{\infty} \bC_{\pi^{L}}[ \overline{\Delta}_{1}^{U}, \overline{\Delta}_{j+1}^{U}]
  \leq 2 C_{L}^{1/2} \biggl( \frac{\rho_{L}^{1/2}}{1-\rho_{L}^{1/2}} \biggr) \biggl( \sum_{k \in S^{L}} ( \bE_{\delta_{k}^{L}}[\overline{\Delta}_{1}^{U} - (\mu^{U}+\mu^{L})] )^{2} \pi_{k}^{L} \biggr)
\end{equation*}
and the finiteness of the RHS yields $\overline{\sigma}^{2,U} < \infty$. The proof of $\overline{\sigma}^{2,L} < \infty$ is similar.

We finally establish that $\sigma^{2,U}$, $\sigma^{2,L}$, $\overline{\sigma}^{2,U}$, and $\overline{\sigma}^{2,L}$ are positive. We first show that, conditionally on $Z_{0} \sim \delta_{i}^{L}$ and $Z_{\tau_{1}} \sim \delta_{k}^{U}$, $\Delta_{1}^{U}$ and $\Delta_{1}^{L}$ are non-degenerate. To this end, suppose by contradiction that
\begin{equation*}
  1 = \bP_{\delta_{i}^{L}}(\Delta_{1}^{U}=\mu^{U}) = \sum_{u=L_{U}+1}^{\infty} \bP_{\delta_{i}^{L}}(Z_{\mu^{U}} \geq u, Z_{\mu^{U}-1} < u, \dots, Z_{1} < u) \bP(U_{1}=u).
\end{equation*}
Since $U_{1}$ has support $S_{B}^{U}$, we obtain that $\bP_{\delta_{i}^{L}}(Z_{\mu^{U}} \geq u, Z_{\mu^{U}-1} < u, \dots, Z_{1} < u)=1$ for all $u \in S_{B}^{U}$. In particular,
\begin{equation*}
  \bE_{\delta_{i}^{L}}[Z_{\mu^{U}-1}] < u \leq \bE_{\delta_{i}^{L}}[Z_{\mu^{U}}] = M^{U} \bE_{\delta_{i}^{L}}[Z_{\mu^{U}-1}] + N^{U}.
\end{equation*}
By taking both $u=L_{U}+1$ and $u \geq M^{U}(L_{U}+1)+N^{U}$ in the above equation we obtain that both $\bE_{\delta_{i}^{L}}[Z_{\mu^{U}-1}] < L_{U}+1$ and $\bE_{\delta_{i}^{L}}[Z_{\mu^{U}-1}] \geq L_{U}+1$. Similarly, if 
\begin{equation*}
  1 = \bP_{\delta_{k}^{U}}(\Delta_{1}^{L}=\mu^{L}) = \sum_{l=L_{0}}^{L_{U}} \bP_{\delta_{k}^{U}}(Z_{\mu^{L}+\tau_{1}} \leq l, Z_{\mu^{L}+\tau_{1}-1} > l, \dots, Z_{\tau_{1}+1} > l) \bP(L_{1}=l),
\end{equation*}
then using that $L_{1}$ has support $\mathbb{N} \cap [L_{0}, L_{U}]$, we obtain that
\begin{equation*}
  \bP_{\delta_{k}^{U}}(Z_{\mu^{L}+\tau_{1}} \leq l, Z_{\mu^{L}+\tau_{1}-1} > l, \dots, Z_{\tau_{1}+1} > l)=1 \text{ for all } L_{0} \leq l \leq L_{U}.
\end{equation*}
In particular,
\begin{equation*}
  (M^{L})^{\mu^{L}} k \leq l < (M^{L})^{(\mu^{L}-1)} k.
\end{equation*}
By taking both $l=L_{0}$ and $l=L_{U}$ we obtain that $L_{0} > M^{L} L_{U}$, which contradicts \ref{H4}. We deduce that $\sum_{j=0}^{T_{i,1}^{L}-1}(\Delta_{j+1}^{U}-\mu^{U})$ is non-degenerate and, similarly,
\begin{equation*}
  \sum_{j=0}^{T_{i,1}^{U}-1} (\Delta_{j+1}^{L}-\mu^{L}), \sum_{j=0}^{T_{i,1}^{L}-1}(\overline{\Delta}_{j+1}^{U}-(\mu^{U}+\mu^{L})), \text{ and } \sum_{j=0}^{T_{i,1}^{L}-1}(\overline{\Delta}_{j+1}^{L}-(\mu^{L}+\mu^{U}))
\end{equation*}
  are non-degenerate. Using Lemma \ref{lemma:variance_of_SnL_and_SnU} below, we conclude that
\begin{equation*}
  \sigma^{2,U} = \pi_{i}^{L} \bV_{\delta_{i}^{L}}[S_{T_{i,1}^{L}}^{U}-\mu^{U}T_{i,1}^{L}] =\pi_{i}^{L} \bE_{\delta_{i}^{L}} \biggl[ \biggl( \sum_{j=0}^{T_{i,1}^{L}-1}(\Delta_{j+1}^{U}-\mu^{U}) \biggr)^{2} \biggr]>0
\end{equation*}
and similarly $\sigma^{2,L}>0$, $\overline{\sigma}^{2,U} >0$, and $\overline{\sigma}^{2,L} >0$.
\end{proof}

\subsection{Martingale structure of $\bf{M_{n,i}^{T}}$} \label{subsection:martingale_structure_of_MniT}

We recall that $M_{n,i}^{T} \coloneqq \sum_{j=1}^{n}  D_{j,i}^{T}$, where
\begin{equation*}
    D_{j,1}^{T} = (\overline{P}_{j-1}^{T}-M^{T}) \tilde{\chi}_{j-1}^T \text{ and } D_{j,2}^{T} = \frac{\tilde{\chi}_{j-1}^{T}}{Z_{j-1}} \sum_{i=1}^{Z_{j-1}} ( \xi_{j-1,i}^{T} - \overline{P}_{j-1}^{T}).
\end{equation*}
Also, $\tilde{A}_{n}^{T} = \sum_{j=1}^n \frac{\tilde{\chi}_{j-1}^T}{Z_{j-1}}$ and for $s \geq 0$ $\Lambda_{j,1}^{T,s} = \abs{\overline{P}_{j}^{T}-M^{T}}^{s}$ and $\Lambda_{j,2}^{T,s} = \bE[\abs{\xi_{j,i}^{T}-\overline{P}_{j}^{T}}^{s} | \Pi_{j}^{T}]$.

\begin{Prop} \label{proposition:moments_m_n}
The following holds: \\
(i) For $i=1,2$ $\{ (M_{n,1}^{T}, \mathcal{H}_{n,i}) \}_{n=1}^{\infty}$ is a mean zero martingale sequence and for all $s \geq 0$ $\bE[\abs{D_{j,1}^{T}}^{s} | \mathcal{H}_{j-1,i} ] = \bE[\Lambda_{0,1}^{T,s}] \tilde{\chi}_{j-1}^{T}$ a.s. In particular, 
\begin{equation*}
    \bE[(D_{j,1}^{T})^{2} | \mathcal{H}_{j-1,i}]=V_{1}^{T} \tilde{\chi}_{j-1}^{T} \text{ a.s., and } \bE[(M_{n,1}^{T})^{2}]=V_{1}^{T} \bE[\tilde{C}_{n}^{T}].
\end{equation*}
(ii) $\{ (M_{n,2}^{T}, \mathcal{H}_{n,2}) \}_{n=1}^{\infty}$ is a mean zero martingale sequence satisfying 
\begin{equation*}
    \bE[(D_{j,2}^{T})^{2} | \mathcal{H}_{j-1,2}]=V_{2}^{T} \frac{\tilde{\chi}_{j-1}^T}{Z_{j-1}} \text{a.s., and } \bE[(M_{n,2}^{T})^{2}]= V_{2}^{T} \bE[ \tilde{A}_{n}^{T} ].
\end{equation*}
Additionally, for all $s \geq 1$ $\bE[\abs{D_{j,2}^{T}}^{s} | \mathcal{H}_{j-1,2}] \leq \bE[\Lambda_{0,2}^{T,s}] \tilde{\chi}_{j-1}^T$ a.s. \\
(iii) For $i=1,2$ $\{(\sum_{j=1}^{n} ((D_{j,i}^{T})^{2}-\bE[(D_{j,i}^{T})^{2}]), \tilde{\mathcal{H}}_{n,i})\}_{n=1}^{\infty}$ are mean zero martingale sequences and for all $s \geq 1$
\begin{equation*}
    \bE [ \lvert (D_{j,i}^{T})^{2}-\bE[(D_{j,i}^{T})^{2}] \rvert^{s} | \tilde{\mathcal{H}}_{j-1,i} ] \leq 2^{s} \bE[\Lambda_{0,i}^{T,2s}] \bE[ \tilde{\chi}_{j-1}^{T}].
\end{equation*}
(iv) For all $T_{1},T_{2} \in \{ L,U \}$ and $i_{1},i_{2} \in \{ 1,2\}$ such that either $T_{1} \neq T_{2}$ or $i_{1} \neq i_{2}$ it holds that $\bE[(D_{j,i_{1}}^{T_{1}})(D_{l,i_{2}}^{T_{2}})]=0$ for all $j,l =1,\dots,n$ and $\bE[(M_{n,i_{1}}^{T_{1}})(M_{n,i_{2}}^{T_{2}})]=0$.  In particular, $\{ (\sum_{j=1}^{n} (D_{j,i_{1}}^{T_{1}})(D_{j,i_{2}}^{T_{2}}), \tilde{\mathcal{H}}_{n,2}) \}_{n=0}^{\infty}$ is a mean zero martingale sequence and for all $s \geq 1$
\begin{equation*}
  \bE[\abs{(D_{j,i_{1}}^{T_{1}})(D_{j,i_{2}}^{T_{2}})}^{s} | \tilde{\mathcal{H}}_{j-1,2}] \leq \bE[\Lambda_{0,i_{1}}^{T_{1},s} \Lambda_{0,i_{2}}^{T_{2},s}] \bE[\tilde{\chi}_{j-1}^{T_{1}} \tilde{\chi}_{j-1}^{T_{2}}].
\end{equation*}
\end{Prop}

\begin{proof}[Proof of Proposition \ref{proposition:moments_m_n}]
We begin by proving (i) with $i=1$. We notice that $(M_{n,1}^{T}, \mathcal{H}_{n,i})$ is a martingale since $M_{n,1}^{T}$ is $\mathcal{H}_{n,1}$-measurable and
\begin{equation*}
    \bE[M_{n,1}^{T} | \mathcal{H}_{n-1,1}] = M_{n-1,1}^{T} + \bE[\overline{P}_{n-1}^{T}-M^{T}] \tilde{\chi}_{n-1}^{T} = M_{n-1,1}^{T}.
\end{equation*}
It follows that $\bE[M_{n,1}^{T}] = \bE[M_{1,1}^{T}] = 0$. Next, notice that for $s \geq 0$ 
\begin{equation*}
    \bE[\abs{D_{j,1}^{T}}^{s} | \mathcal{H}_{j-1,1}] = \bE[\abs{\overline{P}_{j-1}^{T}-M^{T}}^{s}] \tilde{\chi}_{j-1}^T = \bE[\Lambda_{0,1}^{T,s}] \tilde{\chi}_{j-1}^{T} \text{ a.s.}
\end{equation*}
In particular, if $s=2$ then $\bE[(D_{j,1}^{T})^{2} | \mathcal{H}_{j-1,1}]=V_{1}^{T} \tilde{\chi}_{j-1}^T$ a.s.\ and the martingale property yields that
\begin{equation*}
    \bE[(M_{n,1}^{T})^{2} = \bE \biggl[ \sum_{j=1}^{n} \bE[(D_{j,1}^{T})^{2} | \mathcal{H}_{j-1,1}] \biggr] = V_{1}^{T} \bE[\tilde{C}_{n}^{T}].
\end{equation*}
Finally, we notice that, since $D_{j,1}^{T}$ do not depend on the offspring distributions $\{ \xi_{j,i}^{T}) \}_{i=0}^{\infty}$, part (i) holds with $\mathcal{H}_{n,1}$ replaced by $\mathcal{H}_{n,2}$. We now turn to the proof of (ii). We notice that $M_{n,2}^{T}$ is $\mathcal{H}_{n,2}$-measurable and using that $\bE[ \xi_{n-1,i}^{T} - \overline{P}_{n-1}^{T}  |  \Pi_{n-1}^{T}] =0$ we obtain that
\begin{equation*}
    \bE[M_{n,2}^{T} | \mathcal{H}_{n-1,2}] = M_{n-1,2}^{T} + \frac{\tilde{\chi}_{n-1}^{T}}{Z_{n-1}} \sum_{i=1}^{Z_{n-1}} \bE[ \bE[ \xi_{n-1,i}^{T} - \overline{P}_{n-1}^{T} |  \Pi_{n-1}^{T}] ] = M_{n-1,2}^{T}
\end{equation*}
yielding the martingale property. It follows that $\bE[M_{n,2}^{T}] = \bE[M_{1,2}^{T}]=0$ as
\begin{equation*}
    \bE[M_{1,2}^{T} | \mathcal{H}_{0,2}] = \frac{\tilde{\chi}_{0}^{T}}{Z_{0}} \sum_{i=1}^{Z_{0}} \bE[ \bE[ \xi_{0,i}^{T} - \overline{P}_{0}^{T} |  \Pi_{0}^{T}]] = 0.
\end{equation*}
We now compute
\begin{equation*}
    \bE[(D_{j,2}^{T})^{2} | \mathcal{H}_{j-1,2}] = \frac{\tilde{\chi}_{j-1}^{T}}{Z_{j-1}^{2}} \bE \biggl[  \bE \biggl[  \biggl( \sum_{i=1}^{Z_{j-1}} (\xi_{j-1,i}^{T} - \overline{P}_{j-1}^{T} \biggr)^{2} | \mathcal{H}_{j-1,2},\Pi_{j-1}^{T} \biggr] | \mathcal{H}_{j-1,2} \biggr].
\end{equation*}
Using that, conditionally on the environment $\Pi_{j-1}^T$, $\{ \xi_{j-1,i}^T \}_{i=1}^{\infty}$ are i.i.d.\ with variance $\overline{\overline{P}}_{j-1}^{T}$, we obtain that
\begin{align*}
    \bE \biggl[ \biggl( \sum_{i=1}^{Z_{j-1}} \xi_{j-1,i}^T - \overline{P}_{j-1}^{T} \biggr)^2 | \mathcal{H}_{j-1,2}, \Pi_{j-1}^{T} \biggl] &= \sum_{i=1}^{Z_{j-1}} \bE \biggl[ (\xi_{j-1,i}^T - \overline{P}_{j-1}^{T})^2 | \mathcal{H}_{j-1,2}, \Pi_{j-1}^{T} \biggl]  \\
    &= Z_{j-1} \overline{\overline{P}}_{j-1}^{T}.
\end{align*}
We conclude that
\begin{equation*}
    \bE[(D_{j,2}^{T})^{2} | \mathcal{H}_{j-1,2}]=V_{2}^{T} \frac{\tilde{\chi}_{j-1}^T}{Z_{j-1}} \text{ a.s.}
\end{equation*}
and
\begin{equation*}
    \bE[(M_{n,2}^{T})^{2}] = \bE \biggl[ \sum_{j=1}^{n} \bE[(D_{j,2}^{T})^{2} | \mathcal{H}_{j-1,2}] \biggr] = V_{2}^{T} \bE[ \tilde{A}_{n}^{T} ].
\end{equation*}
Additionally, Jensen's inequality yields that for $s \geq 1$ 
\begin{equation*}
    \bE[\abs{D_{j,2}^{T}}^{s} | \mathcal{H}_{j-1,2}] \leq \bE \biggl[ \frac{\tilde{\chi}_{j-1}^{T}}{Z_{j-1}} \sum_{i=1}^{Z_{j-1}} \abs{\xi_{j-1,i}^{T} - \overline{P}_{j-1}^{T}}^{s} | \mathcal{H}_{j-1,2} \biggr] = \bE[\Lambda_{0,2}^{T,s}] \tilde{\chi}_{j-1}^T \text{ a.s.}
\end{equation*}
For (iii) we notice that $\sum_{j=1}^{n} ((D_{j,i}^{T})^{2}-\bE[(D_{j,i}^{T})^{2}])$ is $\tilde{\mathcal{H}}_{n,i}$-measurable and since $\tilde{\chi}_{n-1}^{T}$, $Z_{n-1}$, $\Pi_{n-1}^{T}$, and $\{ \xi_{n-1,i}^{T} \}_{i=0}^{\infty}$ are not $\tilde{\mathcal{H}}_{n-1,i}$-measurable we have that $\bE[(D_{n,i}^{T})^{2} | \tilde{\mathcal{H}}_{n-1,i}]=\bE[(D_{n,i}^{T})^{2}]$ and
\begin{equation*}
    \bE \biggl[ \sum_{j=1}^{n} ((D_{j,i}^{T})^{2}-\bE[(D_{j,i}^{T})^{2}]) | \tilde{\mathcal{H}}_{n-1,i} \biggr] = \sum_{j=1}^{n-1} ((D_{j,i}^{T})^{2}-\bE[(D_{j,i}^{T})^{2}]).
\end{equation*}
Again using the convexity of the function $\abs{\cdot}^{s}$ for $s \geq 1$, we get that
\begin{equation*}
    \bE[ \abs{(D_{j,i}^{T})^{2}-\bE[(D_{j,i}^{T})^{2}]}^{s} | \mathcal{H}_{j-1,i}] \leq 2^{s-1} (\bE[(D_{j,i}^{T})^{2s} ] + (\bE[(D_{j,i}^{T})^{2}])^{s} ) \leq 2^{s} \bE[(D_{j,i}^{T})^{2s} ].
\end{equation*}
If $i=1$, then by conditioning on $\tilde{\chi}_{j-1}^{T}$ we have that $\bE[(D_{j,1}^{T})^{2s} ] = \bE[\Lambda_{0,1}^{T,2s}] \bE[ \tilde{\chi}_{j-1}^{T}]$. If $i=2$, then we apply Jensen's inequality and obtain that
\begin{equation*}
    \bE[(D_{j,2}^{T})^{2s} ] \leq \bE[ \Lambda_{j-1,2}^{T,2s} \tilde{\chi}_{j-1}^{T}] = \bE[\Lambda_{0,2}^{T,2s}] \bE[ \tilde{\chi}_{j-1}^{T}].
\end{equation*}
Turning to (iv), we show that for all $T_{1},T_{2} \in \{ L,U \}$ and $i_{1},i_{2} \in \{ 1,2\}$ such that either $T_{1} \neq T_{2}$ or $i_{1} \neq i_{2}$ it holds that $\bE[(D_{j,i_{1}}^{T_{1}})(D_{l,i_{2}}^{T_{2}})]=0$ for all $j,l =1,\dots,n$. This also yields that
\begin{equation*}
    \bE[(M_{n,i_{1}}^{T_{1}})(M_{n,i_{2}}^{T_{2}})] = \sum_{j=1}^{n} \sum_{l=1}^{n} \bE[(D_{j,i_{1}}^{T_{1}})(D_{l,i_{2}}^{T_{2}})] = 0.
\end{equation*}
First, if $l=j$ and $T_{1} \neq T_{2}$, then $\bE[(D_{j,i_{1}}^{T_{1}})(D_{j,i_{2}}^{T_{2}})]=0$ because $\tilde{\chi}_{j-1}^{T_{1}} \tilde{\chi}_{j-1}^{T_{2}}=0$. Next, if $l=j$ and $i_{1} \neq i_{2}$ (say $i_{1}=1$ or $i_{2}=2$), then by conditioning on $\mathcal{H}_{j-1,1}$ and $\Pi_{j-1}^{T_{2}}$ and using that
\begin{equation*}
    \bE \biggl[\xi_{j-1,i}^{T_{2}} - \overline{P}_{j-1}^{T_{2}} | \mathcal{H}_{j-1,1}, \Pi_{j-1}^{T_{2}} \biggr] =0 \text{ a.s.}
\end{equation*}
and $\Pi_{j-1}^{T_{1}}$, $\tilde{\chi}_{j-1}^{T_{1}}$, $\tilde{\chi}_{j-1}^{T_{2}}$, and $Z_{j-1}$ are $\mathcal{H}_{j-1,1}$-measurable, we obtain that
\begin{equation*}
     \bE[(D_{j,1}^{T_{1}})(D_{j,2}^{T_{2}})] = \bE \biggl[ (\overline{P}_{j-1}^{T_{1}}-M^{T_{1}}) \tilde{\chi}_{j-1}^{T_{1}} \frac{\tilde{\chi}_{j-1}^{T_{2}}}{Z_{j-1}} \sum_{i=1}^{Z_{j-1}} \bE \biggl[\xi_{j-1,i}^{T_{2}} - \overline{P}_{j-1}^{T_{2}} | \mathcal{H}_{j-1,1}, \Pi_{j-1}^{T_{2}} \biggr] \biggr] =0.
\end{equation*}
Finally, if $l \neq j$ (say $l > j)$), then by conditioning on $\tilde{\mathcal{H}}_{l-1,2}$ and using that $D_{j,i_{1}}$ is $\tilde{\mathcal{H}}_{l-1,2}$-measurable and $\bE[(D_{l,i_{2}}) | \tilde{\mathcal{H}}_{l-1,2} ]=\bE[D_{l,i_{2}}]=0$, we obtain that
\begin{equation*}
    \bE[(D_{j,i_{1}})(D_{l,i_{2}})] = \bE[D_{j,i_{1}} \bE[D_{l,i_{2}} | \tilde{\mathcal{H}}_{l-1,2} ]] = 0.
\end{equation*}
$\{ (\sum_{j=1}^{n} (D_{j,i_{1}}^{T_{1}})(D_{j,i_{2}}^{T_{2}}), \tilde{\mathcal{H}}_{n,2}) \}_{n=0}^{\infty}$ is a mean zero martingale sequence since $\sum_{j=1}^{n} (D_{j,i_{1}}^{T_{1}})(D_{j,i_{2}}^{T_{2}})$ is $\tilde{\mathcal{H}}_{n,2}$-measurable and $\bE[(D_{j,i_{1}}^{T_{1}})(D_{j,i_{2}}^{T_{2}}) | \tilde{\mathcal{H}}_{j-1,2}]=\bE[(D_{j,i_{1}}^{T_{1}})(D_{j,i_{2}}^{T_{2}})]=0$ if either $T_{1} \neq T_{2}$ or $i_{1} \neq i_{2}$. If $T_{1} \neq T_{2}$ then both $\bE[\abs{(D_{j,i_{1}}^{T_{1}})(D_{j,i_{2}}^{T_{2}})}^{s} | \tilde{\mathcal{H}}_{j-1,2}]=0$ and $\bE[\tilde{\chi}_{j-1}^{T_{1}} \tilde{\chi}_{j-1}^{T_{2}}]=0$. Finally, if $T_{1} = T_{2}$ and $i_{1} \neq i_{2}$ (say $i_{1}=1$ and $i_{2}=2$) then by Jensen's inequality
\begin{equation*}
    \bE[\abs{(D_{j,i_{1}}^{T_{1}})(D_{j,i_{2}}^{T_{2}})}^{s} | \Pi_{j-1}^{T_{1}}, \mathcal{F}_{j-1} ] \leq \abs{(D_{j,i_{1}}^{T_{1}})}^{s} \Lambda_{j-1,i_{2}}^{T,s}  \tilde{\chi}_{j-1}^{T_{2}},
\end{equation*}
which yields that
\begin{equation*}
    \bE[\abs{(D_{j,i_{1}}^{T_{1}})(D_{j,i_{2}}^{T_{2}})}^{s} | \mathcal{F}_{j-1} ] \leq \bE[\Lambda_{0,i_{1}}^{T_{1},s} \Lambda_{0,i_{2}}^{T_{2},s}] \tilde{\chi}_{j-1}^{T_{1}} \tilde{\chi}_{j-1}^{T_{2}}
\end{equation*}
and
\begin{equation*}
\bE[\abs{(D_{j,i_{1}}^{T_{1}})(D_{j,i_{2}}^{T_{2}})}^{s} | \tilde{\mathcal{H}}_{j-1,2}] \leq \bE[\Lambda_{0,i_{1}}^{T_{1},s} \Lambda_{0,i_{2}}^{T_{2},s}] \bE[\tilde{\chi}_{j-1}^{T_{1}} \tilde{\chi}_{j-1}^{T_{2}}].
\end{equation*}
\end{proof}

\section{Numerical experiments} \label{section:numerical_experiments}

In this section we describe numerical experiments to illustrate the evolution of the process under different distributional assumptions. We also study the empirical distribution of the lengths of supercritical and subcritical regimes and illustrate how the process changes when $U_{j}$ and $L_{j}$ exhibit an increasing trend. We emphasize that these experiments illustrate the behavior of the estimates of the parameters of the BPRET when using a finite number of generations in a single synthetic dataset. In the numerical Experiments 1-4 below, we set $L_{0} = 10^{2}$, $L_{U}=10^{4}$, $n \in \{0,1,\dots,10^4\}$,
$U_{j} \sim L_{U}+\texttt{Zeta}(3)$, $L_{j} \sim \texttt{Unif}_{d}(L_{0},10 L_{0})$, and different distributions for $Z_{0}$, $I_{0}^{U} \sim Q_{0}^{U}$, $\xi_{0,1}^{U} \sim P_{0}^{U}$, and $\xi_{0,1}^{L} \sim P_{0}^{L}$ as follows:

\begin{center}
\begin{tabular}{|p{1cm}|p{3cm}|p{3cm}|p{3.2cm}|p{3.2cm}|}
\hline
        & $Z_{0}-L_{0}$ & $I_{0}^{U}$ & $\xi_{0,1}^{U}$ & $\xi_{0,1}^{L}$ \\ 
\hline
Exp.\ 1 & $\texttt{Pois}(1;L_{U}-L_{0})$ & $\texttt{Pois}(\Lambda^{I})$, \newline $\Lambda^{I} \sim \texttt{Unif}(0,10)$ & $\texttt{Pois}(\Lambda^{U})$, \newline $\Lambda^{U} \sim \texttt{Unif}(0.9,2.1)$ & $\texttt{Pois}(\Lambda^{L})$, \newline $\Lambda^{L} \sim \texttt{Unif}(0.5,1.1)$ \\ 
\hline
Exp.\ 2 & $\texttt{Pois}(1;L_{U}-L_{0})$ & $\texttt{Pois}(\Lambda^{I})$, \newline $\Lambda^{I} \sim \texttt{Gamma}(5,1)$ & $\texttt{Pois}(\Lambda^{U})$, \newline $\Lambda^{U} \sim \texttt{Gamma}(2,1)$ & $\texttt{Pois}(\Lambda^{L})$, \newline $\Lambda^{L} \sim \texttt{Gamma}(0.8,1)$ \\ 
 \hline
Exp.\ 3 & $\texttt{Nbin}(1,1;L_{U}-L_{0})$ & $\texttt{Nbin}(R^{I},O^{I})$, \newline $R^{I} \sim 1+\texttt{Pois}(1)$, $O^{I} \sim \texttt{Unif}(0,10)$ & $\texttt{Nbin}(R^{U},O^{U})$, \newline  $R^{U} \sim 1+\texttt{Pois}(1)$, $O^{U} \sim \texttt{Unif}(0.9,2.1)$ & $\texttt{Nbin}(R^{L},O^{L})$, \newline $R^{L} \sim 1+\texttt{Pois}(1)$, $O^{L} \sim \texttt{Unif}(0.5,1.1)$ \\ 
 \hline
Exp.\ 4 & $\texttt{Nbin}(1,1;L_{U}-L_{0})$ & $\texttt{Nbin}(R^{I},O^{I})$, \newline $R^{I} \sim 1+\texttt{Pois}(1)$, $O^{I} \sim \texttt{Gamma}(5,1)$ & $\texttt{Nbin}(R^{U},O^{U})$, \newline $R^{U} \sim 1+\texttt{Pois}(1)$, $O^{U} \sim \texttt{Gamma}(2,1)$ & $\texttt{Nbin}(R^{L},O^{L})$, \newline $R^{L} \sim 1+\texttt{Pois}(1)$, $O^{L} \sim \texttt{Gamma}(0.8,1)$ \\ 
 \hline
\end{tabular}
\end{center}

In the above description, we have used the notation $\texttt{Unif}(a,b)$ for the uniform distribution over the interval $(a,b)$ and $\texttt{Unif}_{d}(a,b)$ for the uniform distribution over integers between $a$ and $b$. $\texttt{Zeta}(s)$ is the zeta distribution with exponent $s>1$. $\texttt{Pois}(\lambda)$ is the Poisson distribution with parameter $\lambda$ while $\texttt{Pois}(\lambda;b)$ is the Poisson distribution truncated to values not larger than $b$. Similarly, $\texttt{Nbin}(r,o)$ is the negative binomial distribution with predefined number of successful trials $r$ and mean $o$ while $\texttt{Nbin}(r,o;b)$ is the negative binomial distribution truncated to values not larger than $b$. Finally, $\texttt{Gamma}(\alpha,\beta)$ is the Gamma distribution with shape parameter $\alpha$ and rate parameter $\beta$. In these experiments, there were between $400$ to $700$ crossing of the thresholds depending on the distributional assumptions. The results of the numerical Experiments 1-4 are shown in Figure \ref{plots_BPRET_empirical_stationary_distribution}.

We next turn our attention to construction of confidence intervals for the means in the supercritical and subcritical regimes. Values of $M^{T}=\bE[\overline{P}_{0}^{T}]$, $V_{1}^{T}=\bV[\overline{P}_{0}^{T}]$, and $V_{2}^{T}=\bE[\overline{\overline{P}}_{0}^{T}]$ in Exp.\ 1-4 can be deduced from the underlying distributions and are summarized below. The values of $V_{2}^{U}$ and $V_{2}^{L}$ in Exp.\ 3-4 are rounded to 3 decimal digits.

\begin{center}
\begin{tabular}{|p{1.2cm}|p{1.2cm}|p{1.2cm}|p{1.2cm}|p{1.2cm}|p{1.2cm}|p{1.2cm}|p{1.2cm}|}
\hline
  & $M^{U}$ & $V_{1}^{U}$ & $V_{2}^{U}$ & $M^{L}$ & $V_{1}^{L}$ & $V_{2}^{L}$ \\ 
\hline
Exp.\ 1 & $1.5$ & $0.12$ & $1.5$ & $0.8$ & $0.03$ & $0.8$  \\  
\hline
Exp.\ 2 & $2$ & $2$ & $2$ & $0.8$ & $0.8$ & $0.8$ \\ 
 \hline
Exp.\ 3 & $1.5$ & $0.12$ & $2.998$ & $0.8$ & $0.03$ & $1.224$ \\  
 \hline
Exp.\ 4 & $2$ & $2$ & $5.793$ & $0.8$ & $0.8$ & $1.710$ \\  
\hline
\end{tabular}
\end{center}

In the next table, we provide the estimators $M_{n}^{U}$, $C_{n}^{U}/\tilde{N}^{L}(n)$, $\tilde{C}_{n}^{U}/\tilde{N}^{L}(n)$, and $\tilde{A}_{n}^{U}/\tilde{N}^{L}(n)$ of $\bE[\overline{P}_{0}^{U}]$, $\mu^{U}$, $\tilde{\mu}^{U}$, and $\tilde{A}^{U}$, respectively. Notice that
\begin{equation*}
    V_{n,1}^{U} \coloneqq \frac{1}{\tilde{C}_{n}^{U}} \sum_{j=1}^{n} \biggl( \frac{Z_{j}-I_{j-1}^{U}}{Z_{j-1}} - M_{n}^{U} \biggr)^{2} \tilde{\chi}_{j-1}^{U}
   \text{ and }
   V_{n,2}^{U} \coloneqq \frac{1}{\tilde{C}_{n}^{U}} \sum_{j=1}^{n} \frac{1}{Z_{j-1}} \sum_{i=1}^{Z_{j-1}} \biggl( \xi_{j-1,i}^{U} - \frac{Z_{j}-I_{j-1}^{U}}{Z_{j-1}} \biggr)^{2} \tilde{\chi}_{j-1}^{U}
\end{equation*}
are used to estimate $V_{1}^{U}$ and $V_{2}^{U}$. Similar to the proof of Theorem \ref{theorem:law_of_large_numnbers_and_central_limit_theorem_mean_estimation} it is easy to see that $V_{n,1}^{U}$ and $V_{n,2}^{U}$ are consistent estimators of $V_{1}^{U}$ and $V_{2}^{U}$. Similar comments hold when $U$ is replaced by $L$.

\begin{center}
\begin{tabular}{|p{1.5cm}|p{1.5cm}|p{1.5cm}|p{1.5cm}|p{1.5cm}|p{1.5cm}|p{1.5cm}|p{1.5cm}|}
\hline
  & $M_{n}^{U}$ & $V_{n,1}^{U}$ & $V_{n,2}^{U}$ & $C_{n}^{U}/\tilde{N}^{L}(n)$ & $\tilde{C}_{n}^{U}/\tilde{N}^{L}(n)$ & $\tilde{A}_{n}^{U}/\tilde{N}^{L}(n)$ \\ 
\hline
Exp.\ 1 & $1.496$ & $0.122$ & $1.499$ & $9.282$ & $9.282$ & $0.010$ \\ 
\hline
Exp.\ 2 & $2.009$ & $2.051$ & $2.003$ & $10.807$ & $10.804$ & $0.105$ \\ 
 \hline
Exp.\ 3 & $1.507$ & $0.119$ & $3.016$ & $9.069$ & $9.069$ & $0.011$ \\ 
 \hline
Exp.\ 4 & $2.018$  & $2.157$ & $5.929$ & $10.899$ & $10.893$ & $0.112$ \\ 
\hline
\hline
  & $M_{n}^{L}$ & $V_{n,1}^{L}$ & $V_{n,2}^{L}$ & $C_{n}^{L}/\tilde{N}^{U}(n)$ & $\tilde{C}_{n}^{L}/\tilde{N}^{U}(n)$ & $A_{n}^{L}/\tilde{N}^{U}(n)$ \\
\hline
Exp.\ 1 & $0.799$ & $0.031$ & $0.800$ & $14.063$ & $14.063$ & $0.009$ \\
\hline
Exp.\ 2 & $0.809$ & $0.805$ & $0.809$ & $5.118$ & $5.118$ & $0.002$ \\
\hline
Exp.\ 3 & $0.799$ & $0.030$ & $1.214$ & $14.028$ & $14.028$ & $0.010$ \\
 \hline
Exp.\ 4 & $0.822$ & $0.945$  & $1.869$ & $5.136$ & $5.136$ & $0.002$ \\ 
 \hline
\end{tabular}
\end{center}
Using the above estimators in Theorem \ref{theorem:law_of_large_numnbers_and_central_limit_theorem_mean_estimation} we obtain the following confidence intervals for $M_{n}^{U}$ and $M_{n}^{L}$. We also provide confidence intervals for the estimator $M_{n}$ defined below, which \emph{does not take into account different regimes}. Specifically,
\begin{equation*}
    M_{n} \coloneqq \frac{1}{\sum_{j=1}^{n} \mathbf{I}_{\{Z_{j-1} \geq 1\}}} \sum_{j=1}^{n}   \frac{Z_{j}-I_{j-1}^{T}}{Z_{j-1}} \mathbf{I}_{\{Z_{j-1} \geq 1\}},
\end{equation*}
where $I_{j-1}^{T}$ is equal to $I_{j-1}^{U}$ if $T=U$ and $0$ otherwise.

\begin{center}
\begin{tabular}{|p{1.5cm}|p{3cm}|p{3cm}|p{3cm}|}
\hline
    & $95\%$ CI using $M_{n}^{U}$ & $95\%$ CI using $M_{n}^{L}$ & $95\%$ CI using $M_{n}$ \\
\hline
Exp.\ 1 & $(1.485,1.507)$ & $(0.795,0.804)$ & $(1.068,1.085)$ \\
\hline
Exp.\ 2 & $(1.975,2.043)$ & $(0.778,0.840)$ & $(1.596,1.651)$ \\
 \hline
Exp.\ 3 & $(1.496,1.518)$ & $(0.794,0.803)$ & $(1.068,1.085)$ \\
 \hline
Exp.\ 4 & $(1.982,2.053)$ & $(0.788,0.855)$ & $(1.607,1.664)$ \\
 \hline
\end{tabular}
\end{center}

\begin{figure}
\centering
\includegraphics[width=0.24\linewidth]{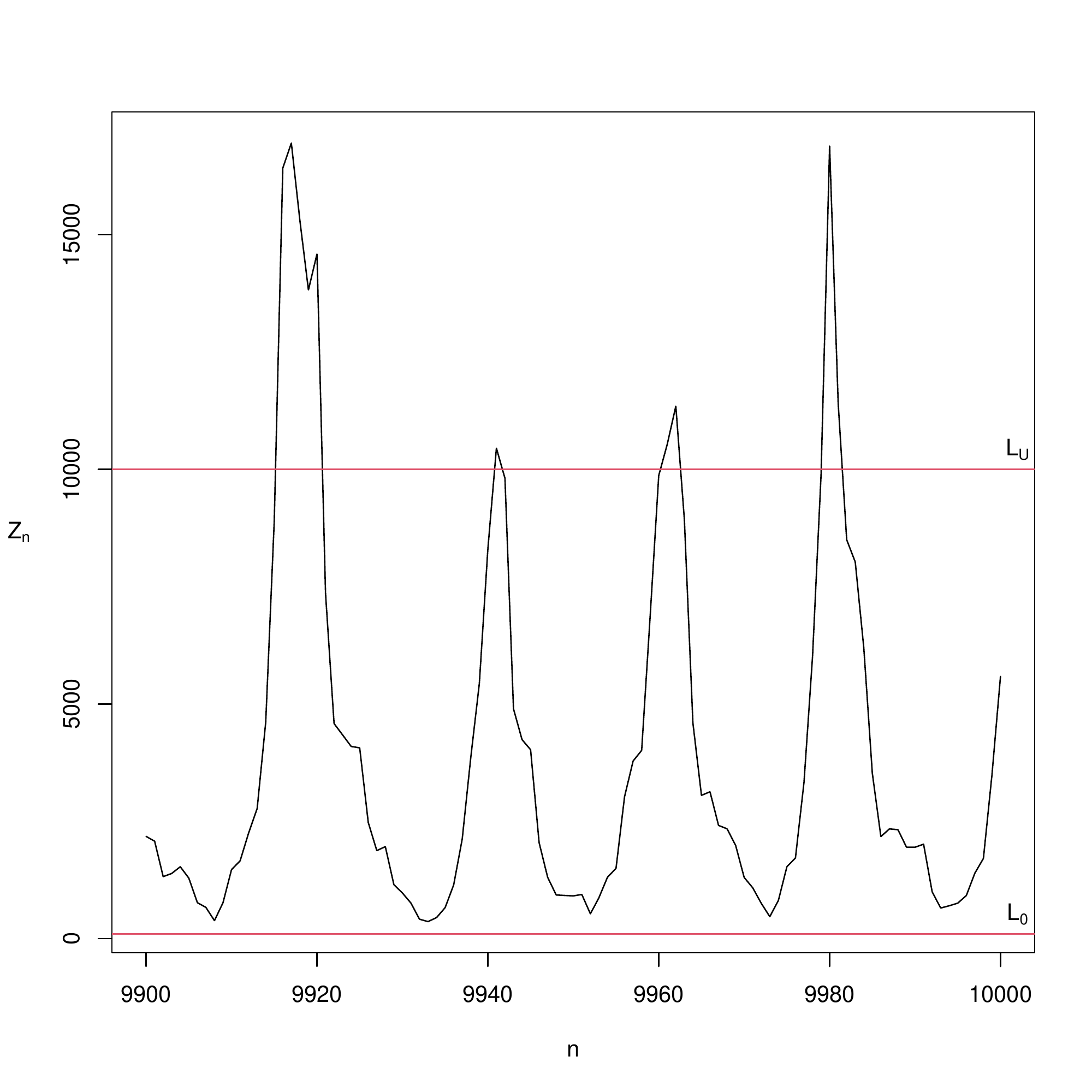}
\includegraphics[width=0.24\linewidth]{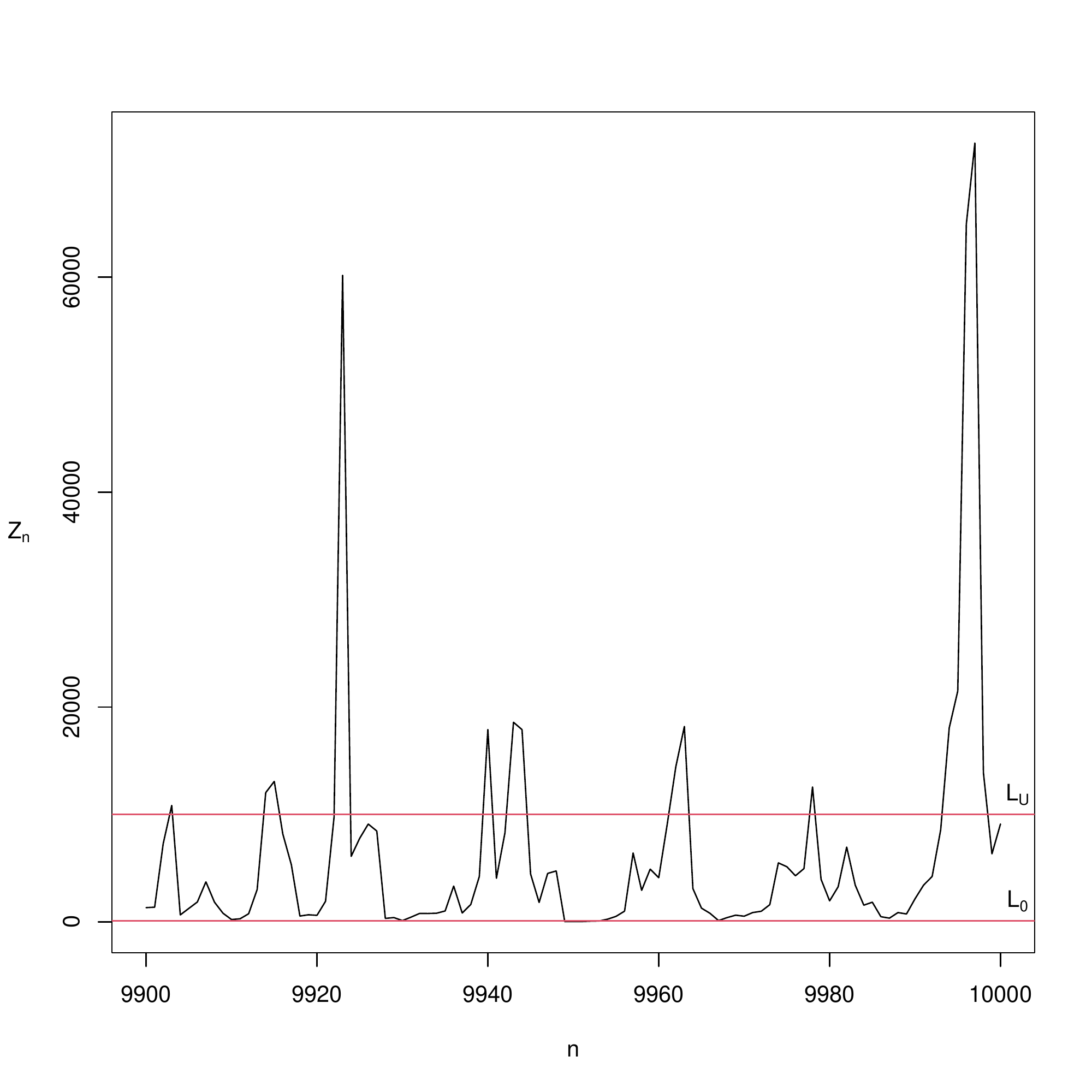}
\includegraphics[width=0.24\linewidth]{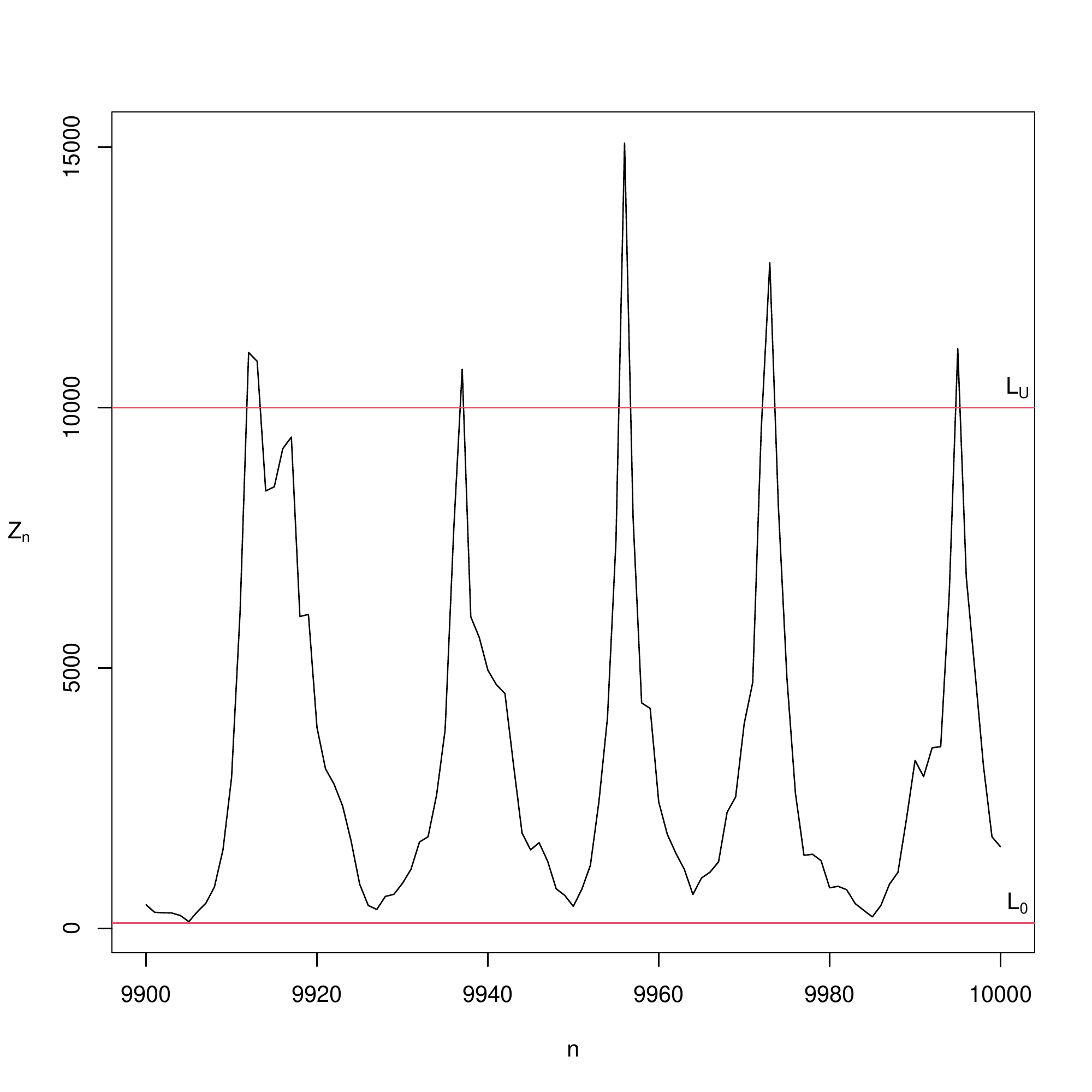}
\includegraphics[width=0.24\linewidth]{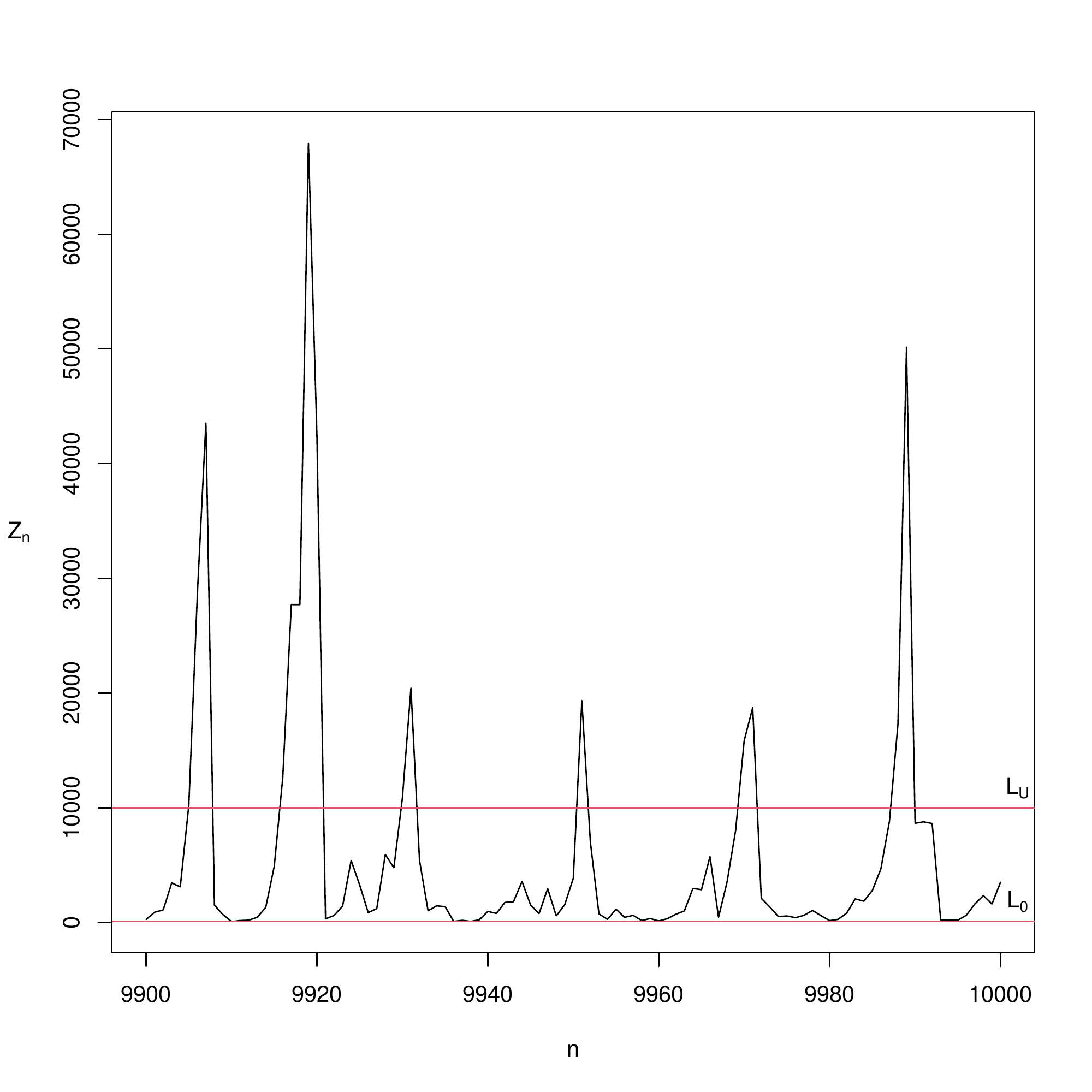}
\includegraphics[width=0.24\linewidth]{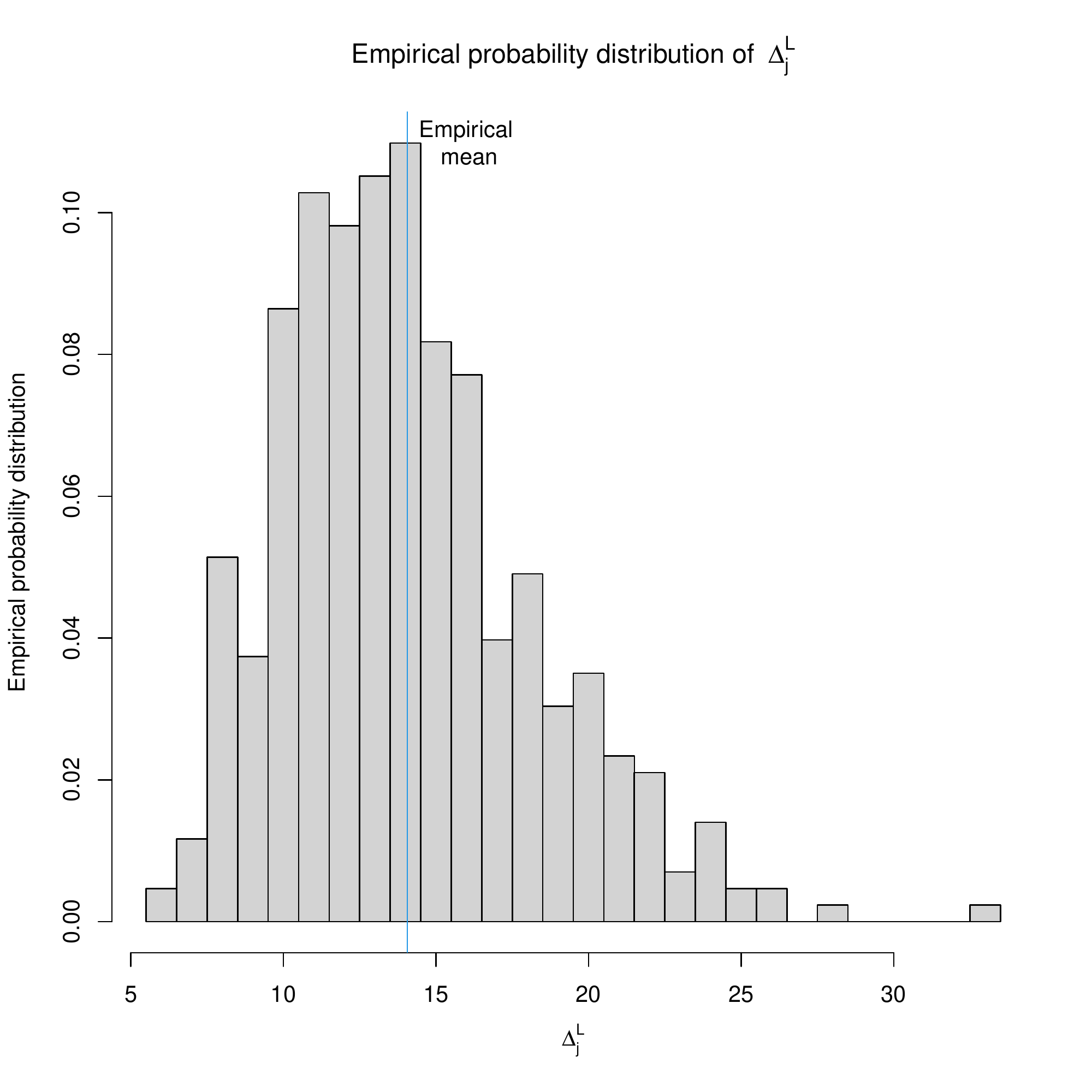}
\includegraphics[width=0.24\linewidth]{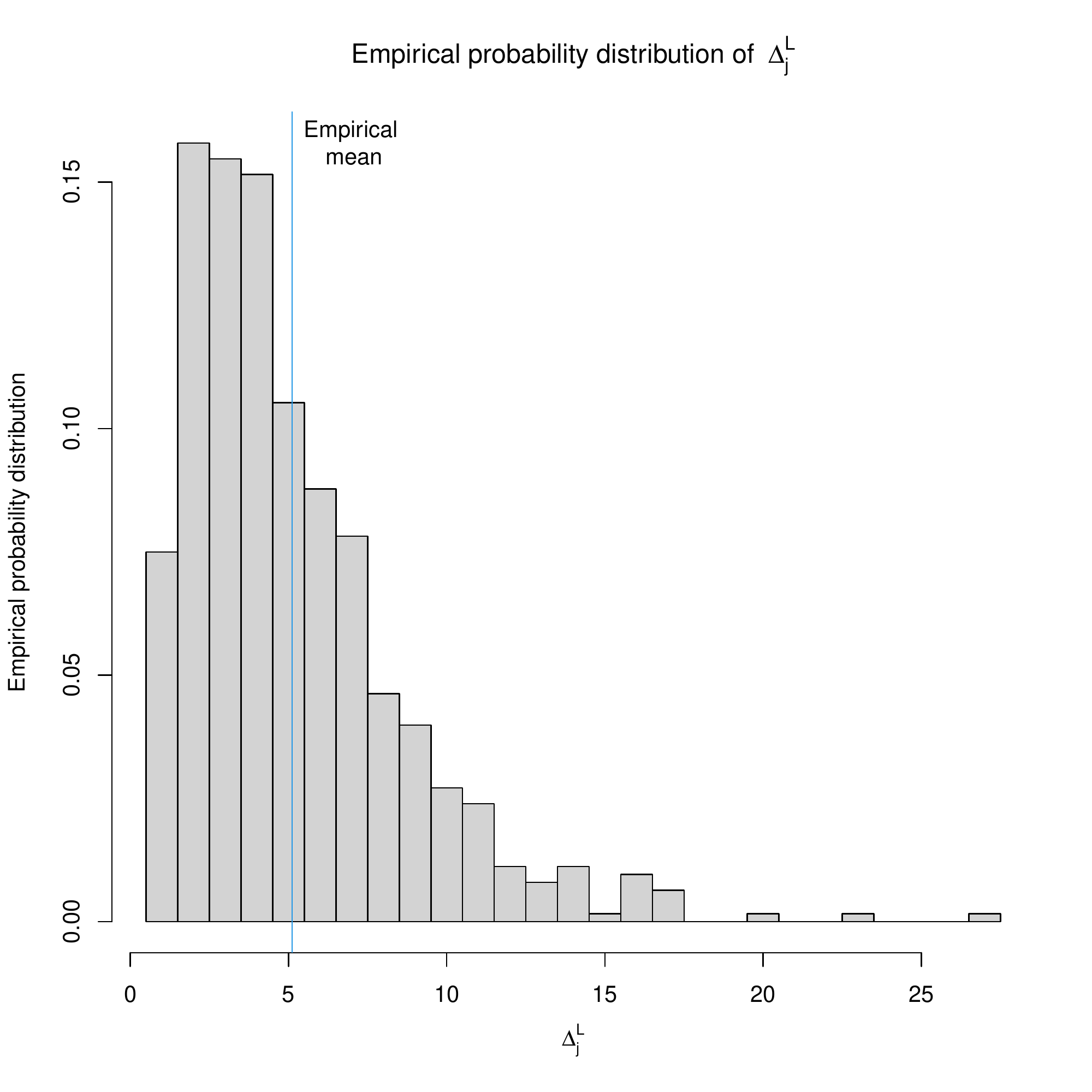}
\includegraphics[width=0.24\linewidth]{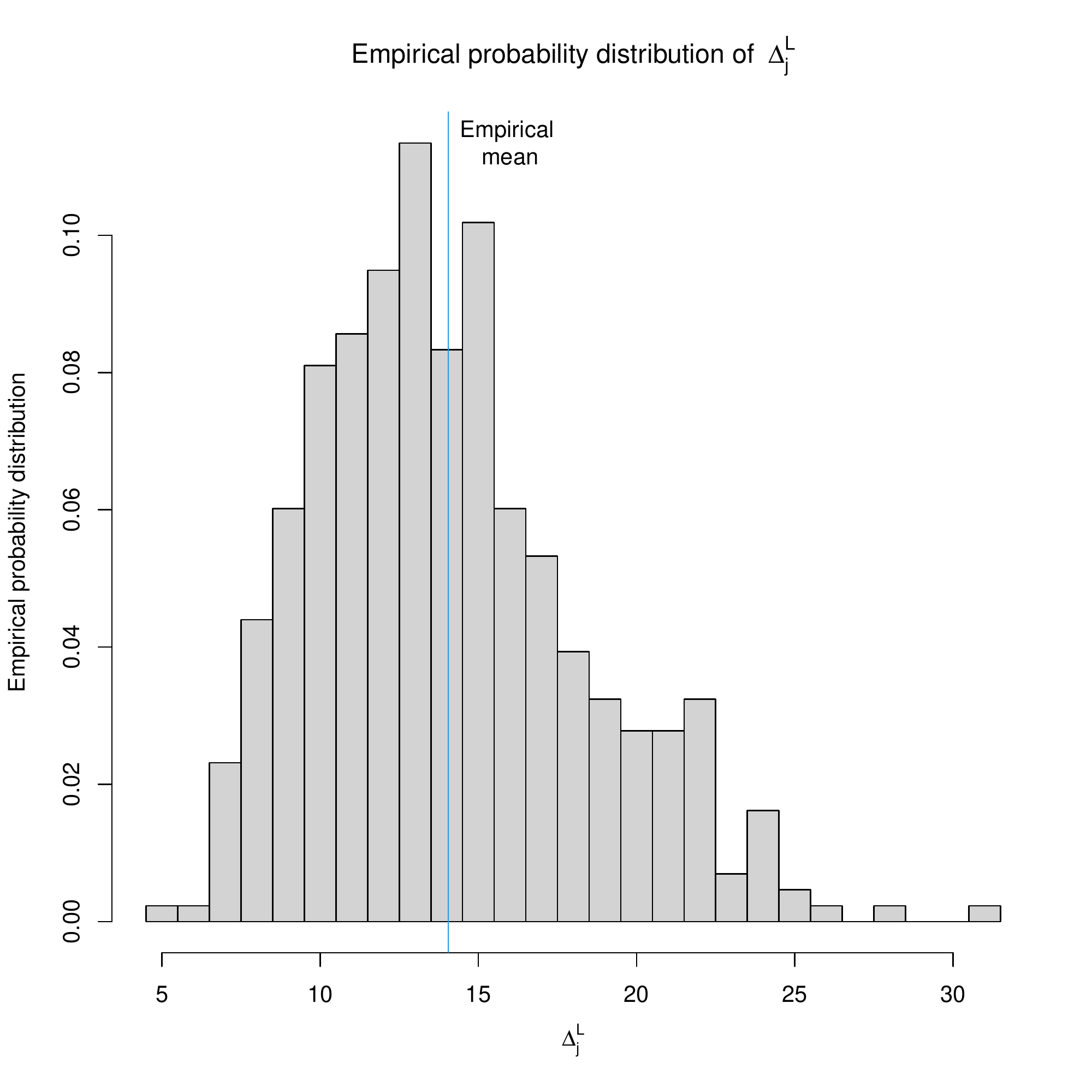}
\includegraphics[width=0.24\linewidth]{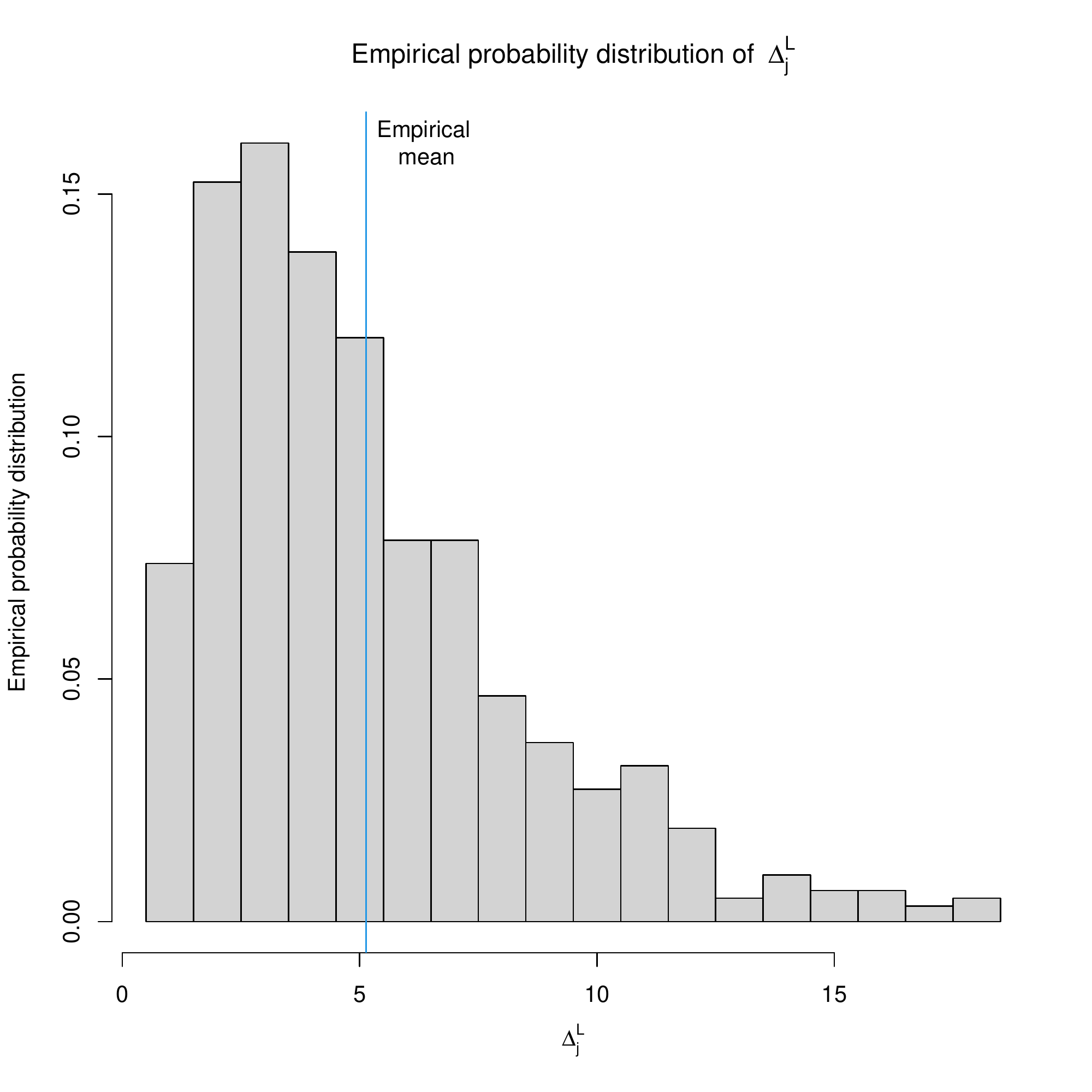}
\includegraphics[width=0.24\linewidth]{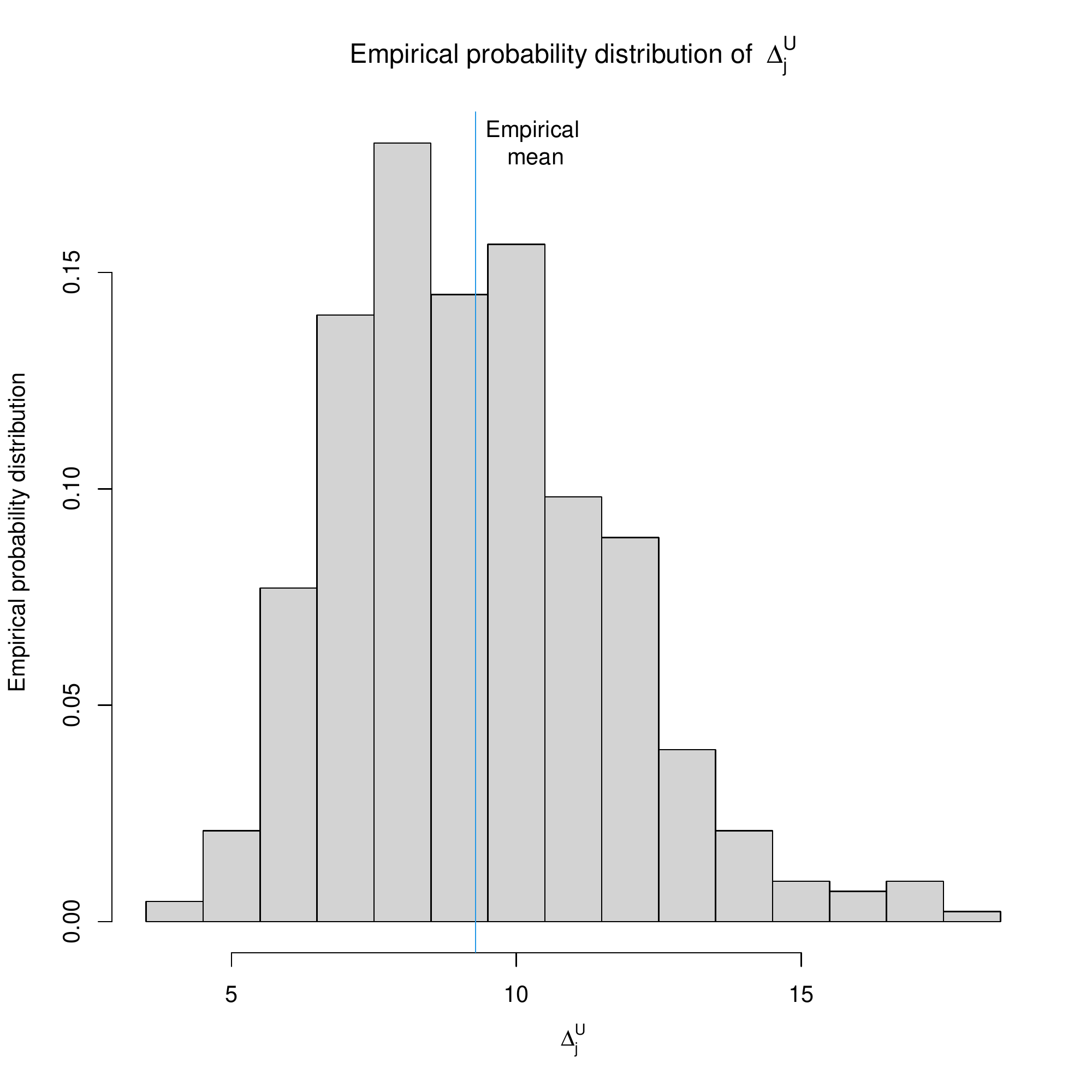}
\includegraphics[width=0.24\linewidth]{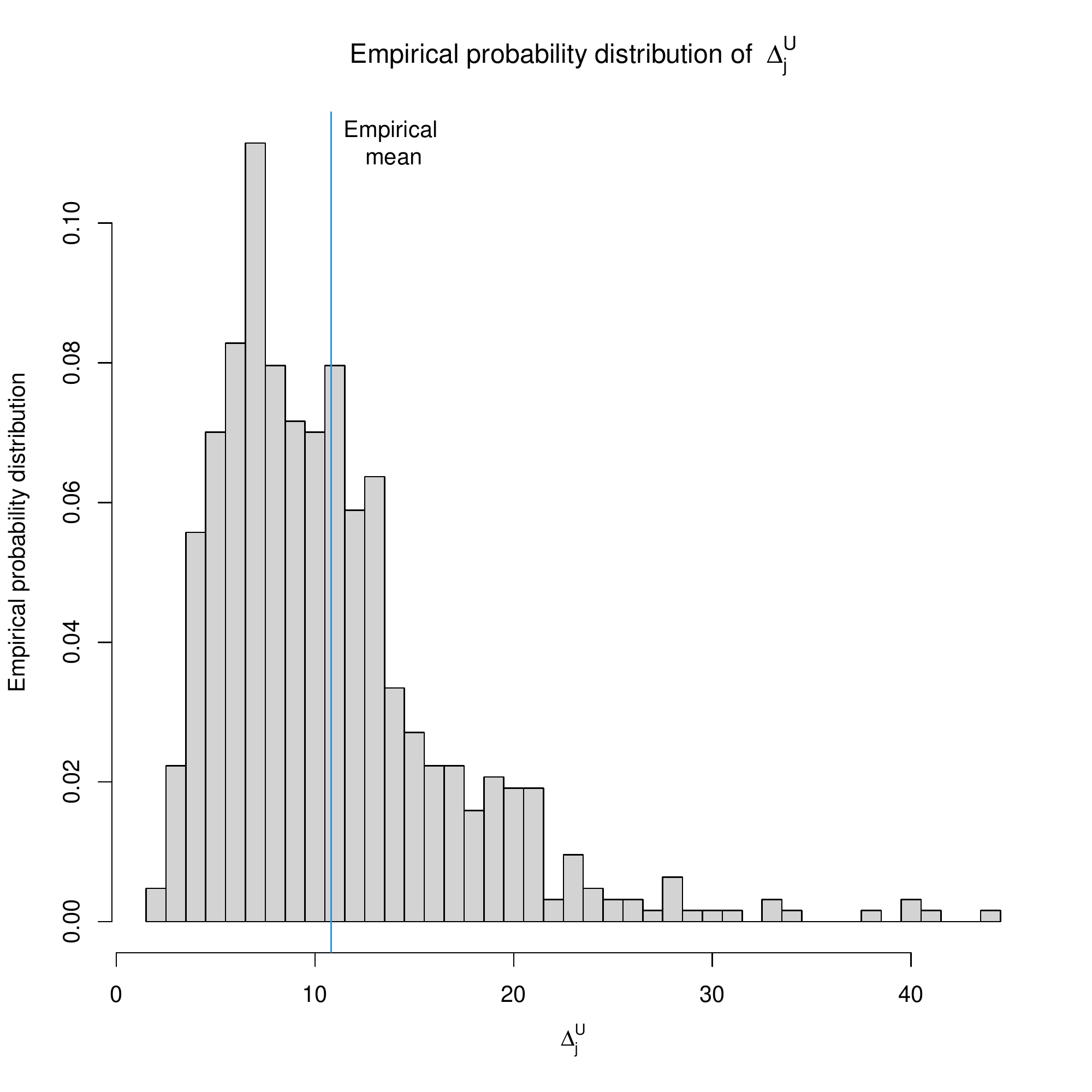}
\includegraphics[width=0.24\linewidth]{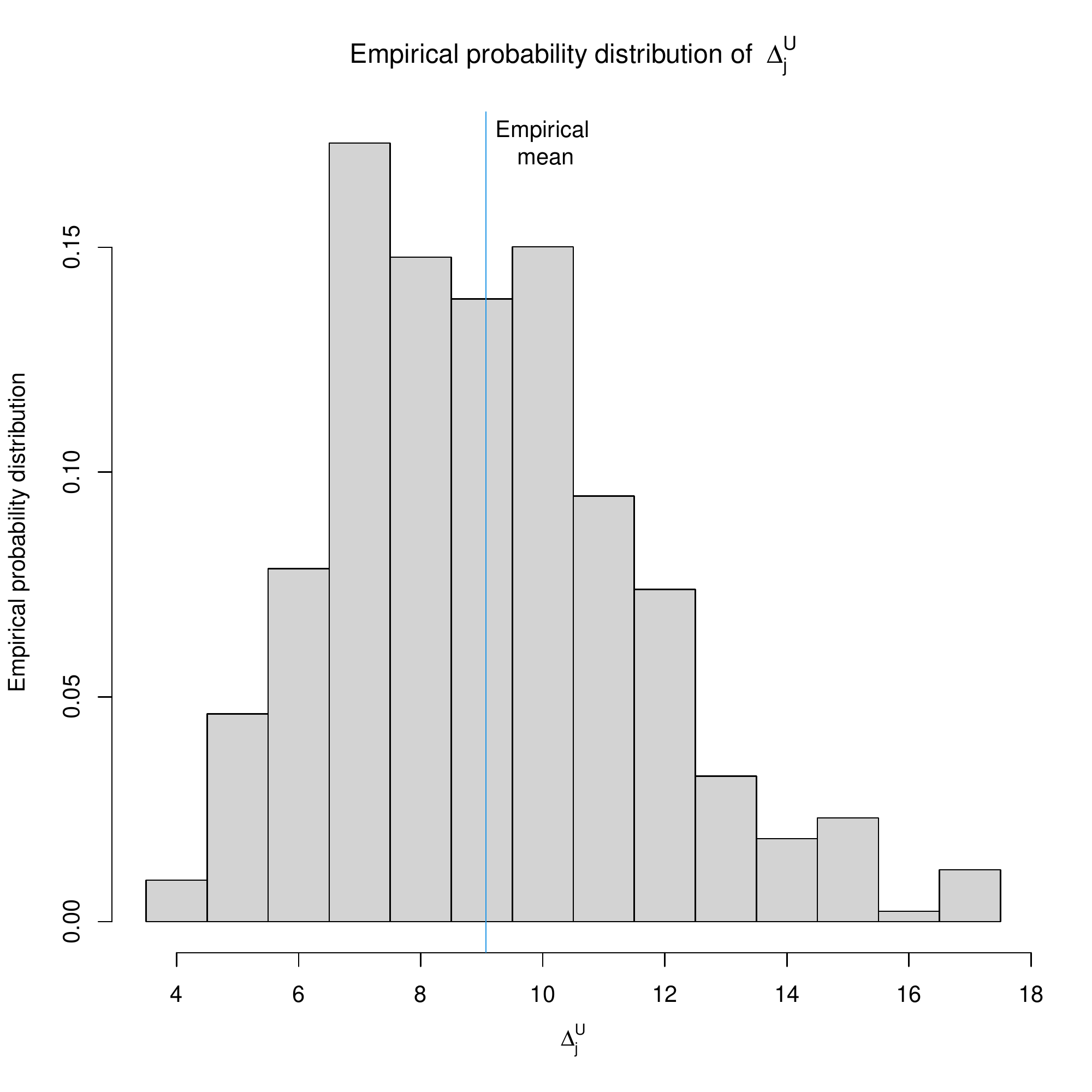}
\includegraphics[width=0.24\linewidth]{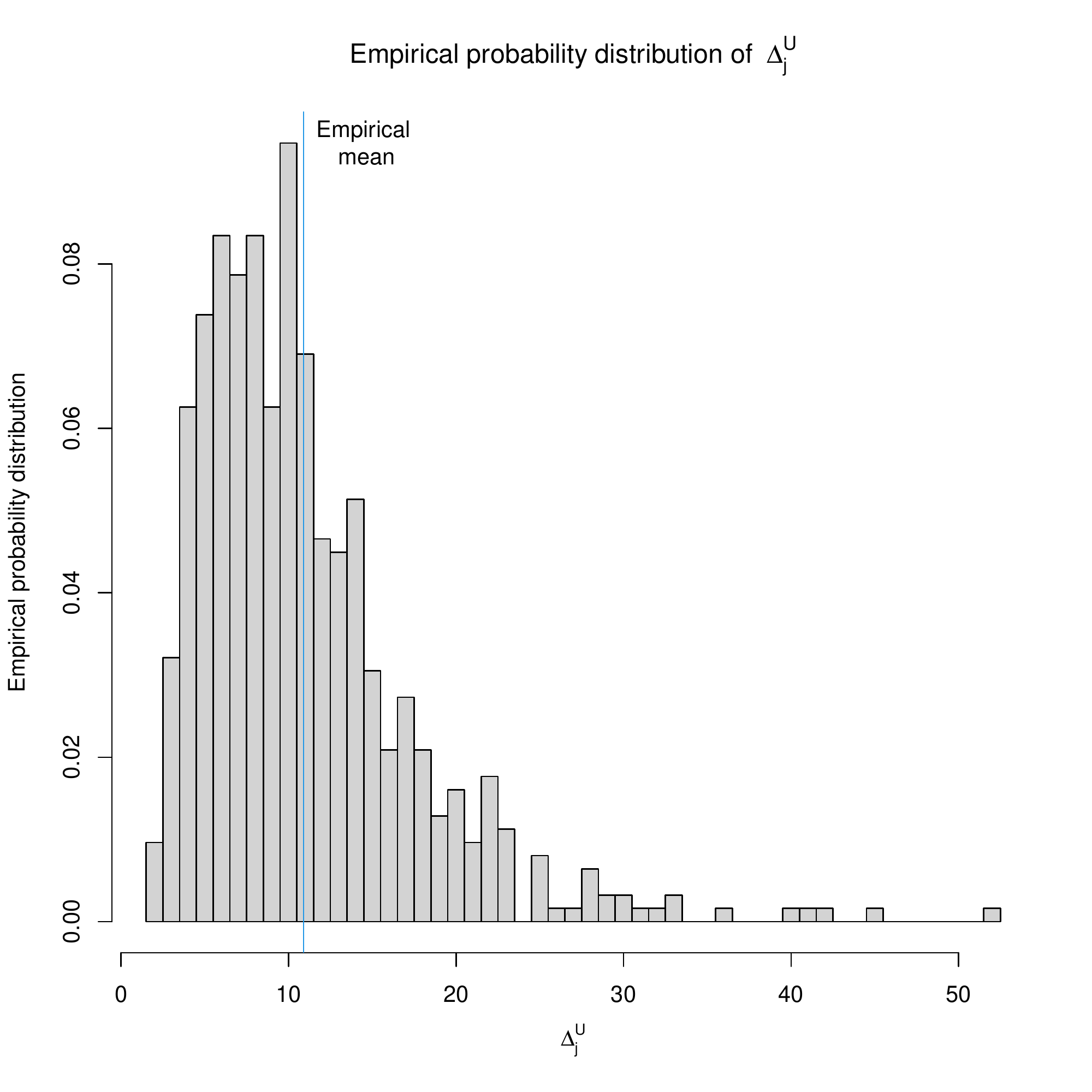}

    \caption{In each column the numerical Experiments 1, 2, 3, and 4. In the first row, the process $Z_{n}$ for $n=10^4-10^2, \dots, 10^4$. In the second and third row, the empirical probability distribution of $\{ \Delta_{j}^{L} \}$ and $\{ \Delta_{j}^{U} \}$, respectively.}
\label{plots_BPRET_empirical_stationary_distribution}
\end{figure}

Next, we investigate the behavior of the process when the thresholds $L_{j}$ and $U_{j}$ increase with $j$. To this end, we let $L_{0}=10^{2}$, $L_{U}=10^{4}$, and $n \in \{0,1, \dots, 10^3\}$ and take initial distribution $Z_{0}$, immigration distribution $I_{0}^{U}$, and offspring distributions $\xi_{0,1}^{U}$ and $\xi_{0,1}^{L}$ as in the numerical Experiment 1. We consider four different distributions for $L_{j}$ and $U_{j}$ as follows:
\begin{center}
\begin{tabular}{|p{1cm}|p{5cm}|p{5cm}|}
\hline
        & $U_{j}-L_{U}$ & $L_{j}$ \\ 
\hline
Exp.\ 5 & $\texttt{Zeta}(3)$ & $\texttt{Unif}_{d}(L_{0},10 L_{0})$ \\ 
\hline
Exp.\ 6 & $\texttt{Zeta}(3)$ & $\texttt{Unif}_{d}(L_{j,1},L_{j,2})$, where \newline $L_{j,1}=\min(L_{0}+100(j-1),L_{U})$, $L_{j,2}=\min(10 L_{0}+100(j-1),L_{U})$ \\ 
 \hline
Exp.\ 7 & $\texttt{Zeta}(3)+500(j-1)$ & $\texttt{Unif}_{d}(L_{0},10 L_{0})$ \\ 
 \hline
Exp.\ 8 & $\texttt{Zeta}(3)+500(j-1)$ & $\texttt{Unif}_{d}(L_{j,1},L_{j,2})$, where \newline $L_{j,1}=\min(L_{0}+100(j-1),L_{U})$, $L_{j,2}=\min(10 L_{0}+100(j-1),L_{U})$ \\ 
 \hline
\end{tabular}
\end{center}
The results of the numerical Experiment 5-8 are shown in Figure \ref{plots_BPRET_increasing_thresholds}. From the plots, we see that the number of cases after crossing the upper thresholds are between $10^{4}$ and $2 \cdot 10^{4}$, whereas when the thresholds increase they almost reach the $6 \cdot 10^{4}$ mark. Also, the number of regimes up to time $n=10^3$ reduces as it takes more time to reach a larger threshold. As a consequence the overall number of cases also  increases.
\begin{figure}
\centering
\includegraphics[width=0.24\linewidth]{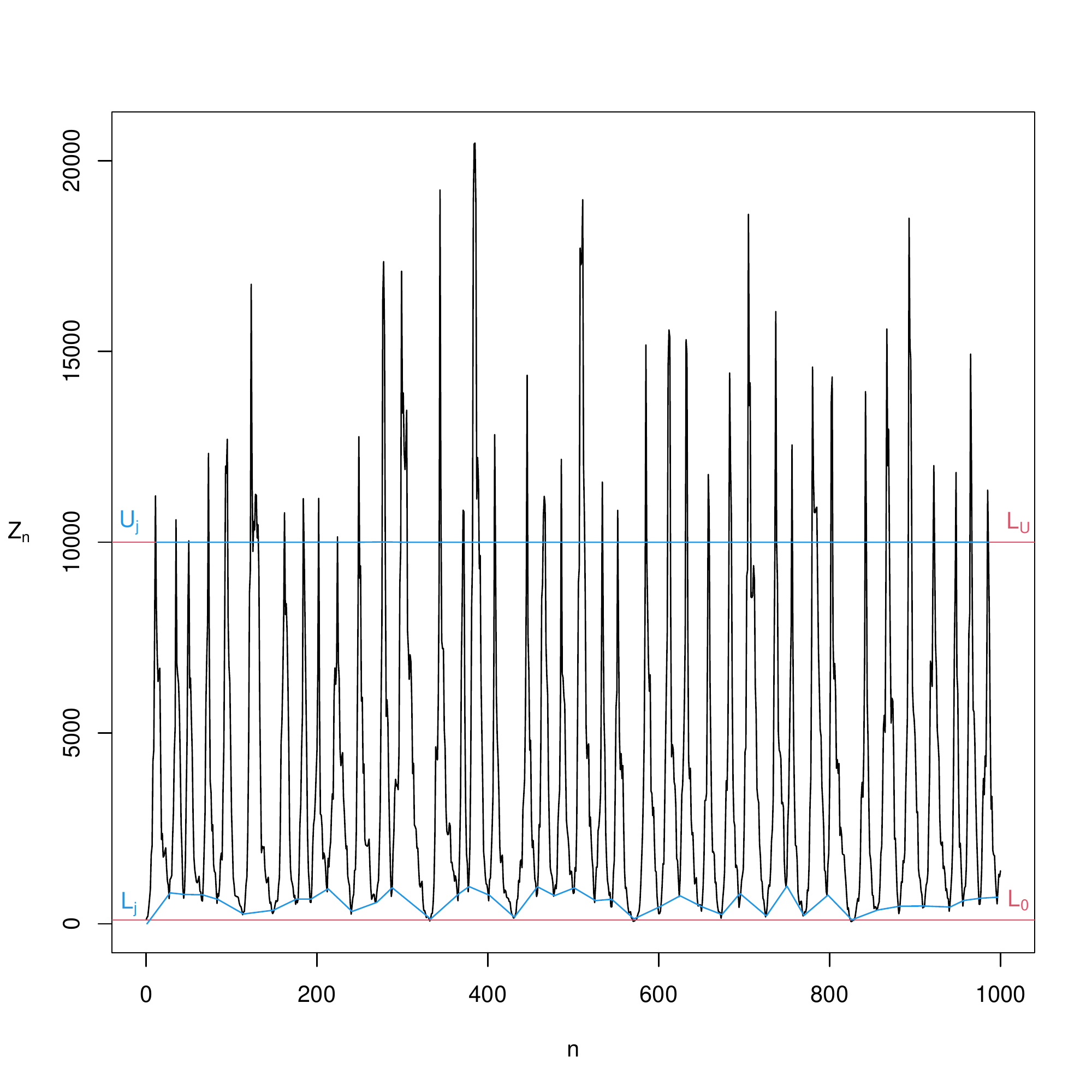}
\includegraphics[width=0.24\linewidth]{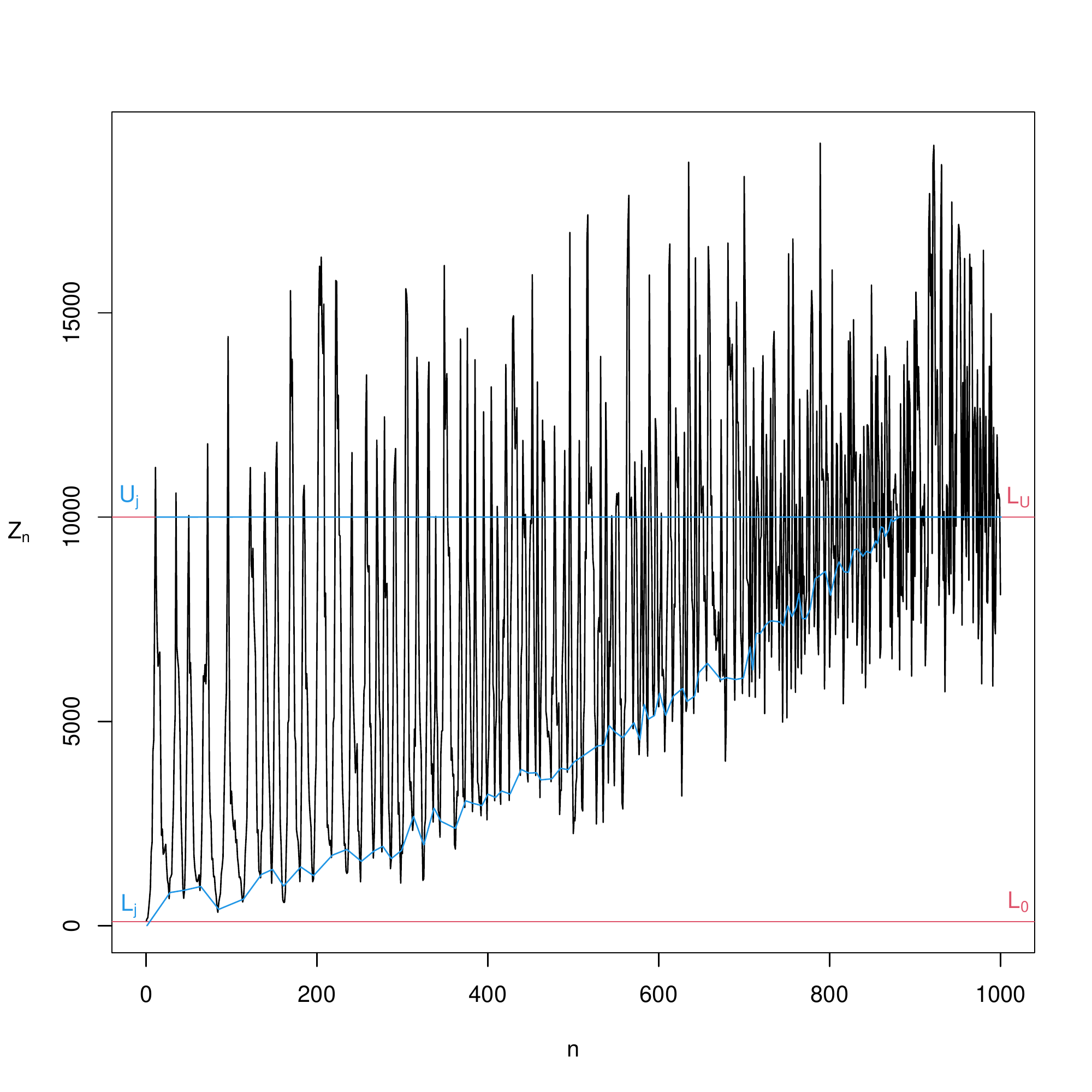}
\includegraphics[width=0.24\linewidth]{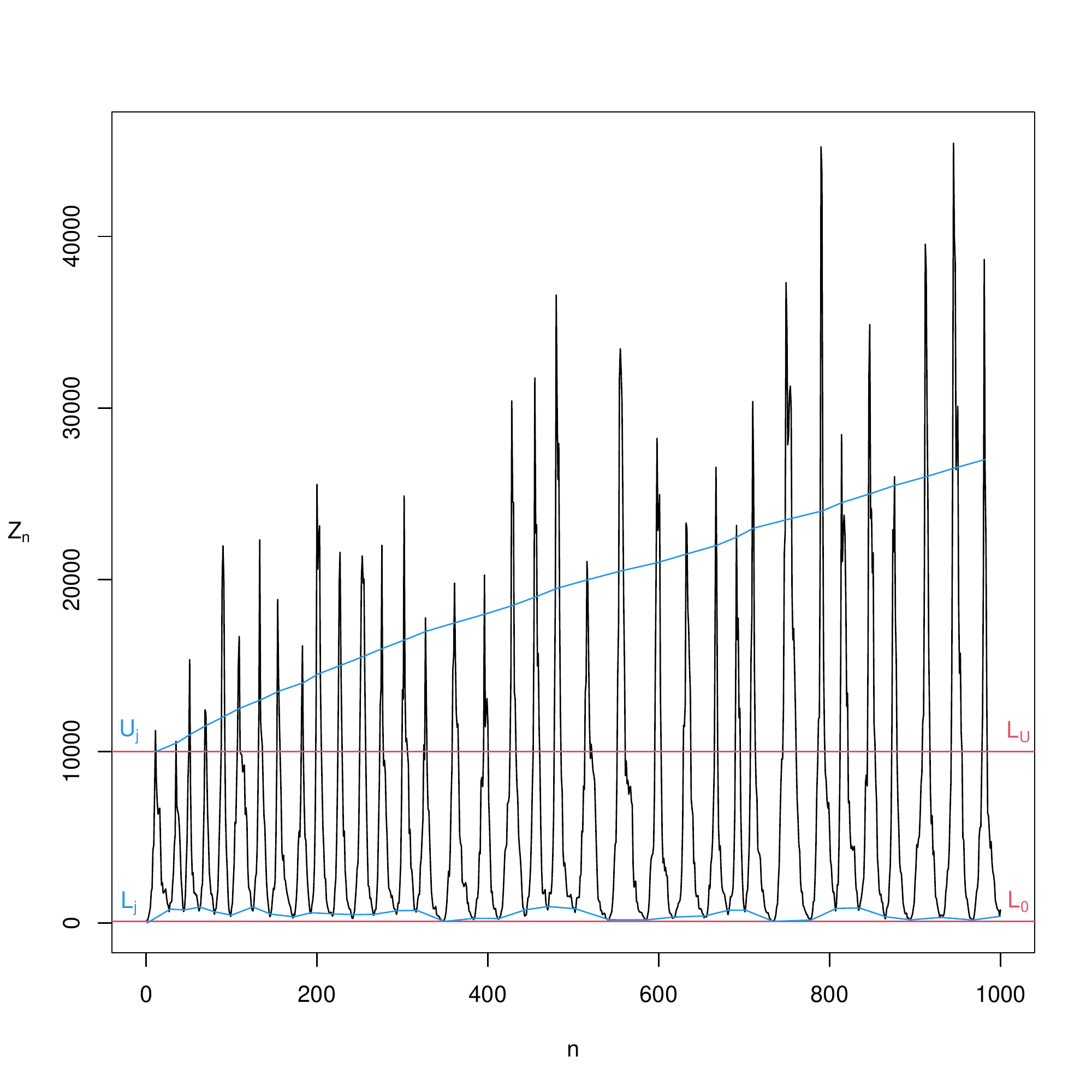}
\includegraphics[width=0.24\linewidth]{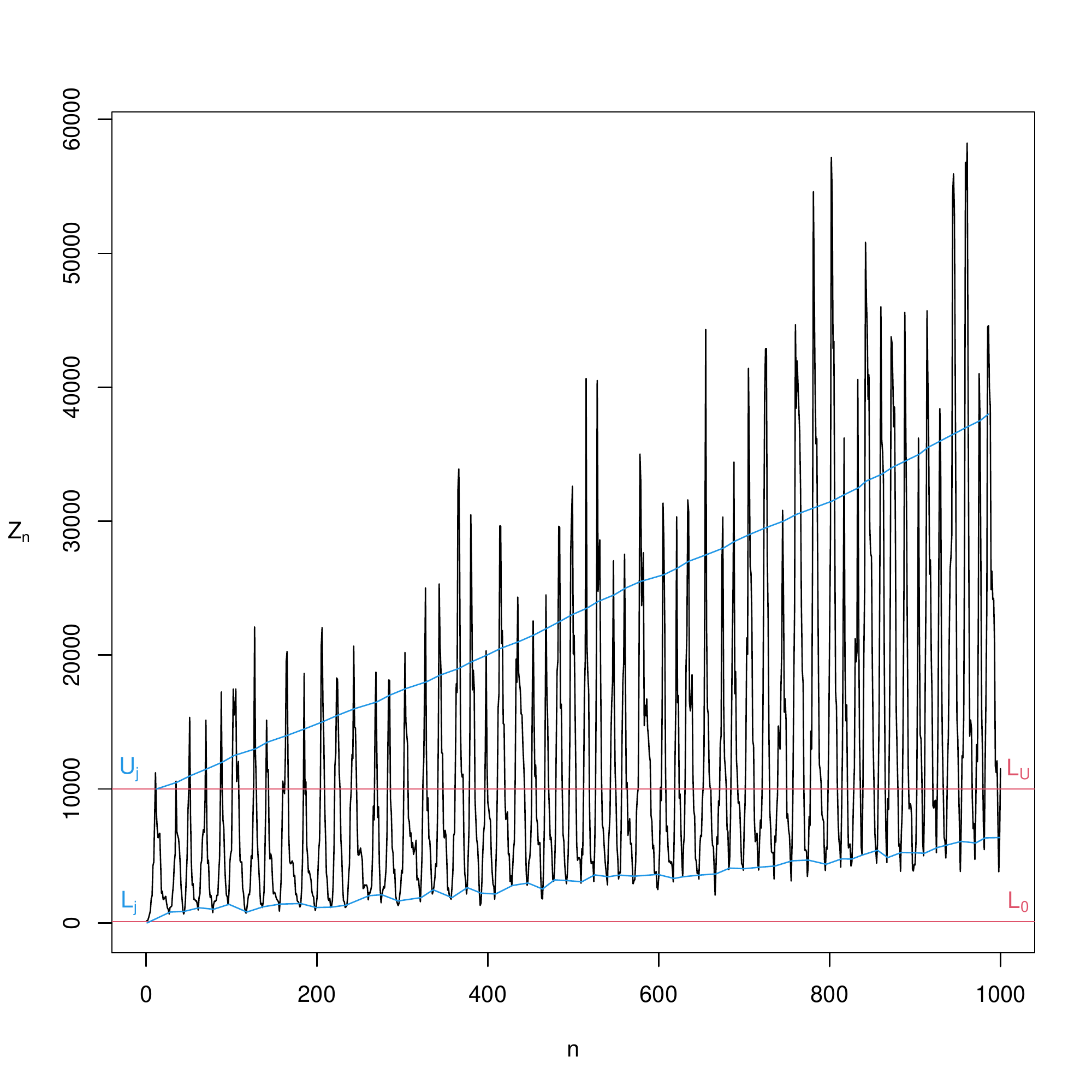}

\caption{In black the process $Z_{n}$ for $n \in \{0,1,\dots,10^{3} \}$, in red horizontal lines at $L_{0}$ and $L_{U}$, and in blue the thresholds $U_{j}$ and $L_{j}$. From left to right, the numerical Experiments 5, 6, 7, and 8.}
\label{plots_BPRET_increasing_thresholds}
\end{figure}


\end{document}